\documentclass[leqno,11pt]{amsart}

\usepackage{amsmath}
\usepackage{graphicx, color}
\usepackage{amscd}
\usepackage{amsfonts}
\usepackage{amssymb}
\usepackage{mathrsfs}
\usepackage{mathtools}
\usepackage{ulem}

\newcommand{\bel}{\begin{equation} \label}

\newtheorem{thm}{Theorem}[section]

\newtheorem{lem}[thm]{Lemma}
\newtheorem{prop}[thm]{Proposition}

\newtheorem{rem}{Remark}[section]

\numberwithin{equation}{section}
\newcommand{\ep}{\varepsilon}
\newcommand\be{\begin{equation}}
\newcommand\ee{\end{equation}}
\newcommand\R{\mathbb R}

\newcommand\N{\mathbb N}

\newcommand\ds{\displaystyle}

\allowdisplaybreaks
\def\eps{\varepsilon}

\newcommand\rn{\R^n}
\newcommand\lcal{\mathcal{L}}

\newcommand\ld{\mathcal{L}}

\newcommand\lu{\lcal[u]}
\newcommand\ldu{\ld u}

\title[Elliptic regularity estimates with optimized constants]
{Elliptic regularity estimates with optimized constants and applications}

\author[Sirakov]{Boyan Sirakov}
\address{PUC-Rio, Departamento de Matematica \\
Rua Marqu\^es de S\~ao Vicente 225 \\
G\'avea, Rio de Janeiro - CEP 22451-900, Brazil}
\email{bsirakov@puc-rio.br}

\author[Souplet]{Philippe Souplet}%
\address{Universit\'e Sorbonne Paris Nord,
CNRS UMR 7539, Laboratoire Analyse, G\'{e}om\'{e}trie et Applications,
93430 Villetaneuse, France}
\email{souplet@math.univ-paris13.fr}

\begin{document}

\begin{abstract}  We revisit the classical theory of linear second-order uniformly elliptic equations in divergence form whose solutions have H\"older continuous gradients, and prove versions of the generalized maximum principle, the $C^{1,\alpha}$-estimate, the Hopf-Oleinik lemma, the boundary weak Harnack inequality and the differential Harnack inequality,  in which the constant is optimized with respect to the norms of the coefficients of the operator and the size of the domain.
 Our estimates are complemented by counterexamples which show their optimality.
We also give applications
to the Landis conjecture and spectral estimates.
\end{abstract}

\maketitle

\section{Introduction}
We consider equations governed by a uniformly elliptic second order operator in general divergence form
\begin{equation}\label{defdiv}
\ldu:=\mathrm{div}(A(x)Du +  b_1(x)u) + b_2(x){\cdot} Du +c(x) u,
\end{equation}
in a bounded domain $\Omega\subset\rn$ ($n\ge 2$ unless otherwise specified) with $C^{1,\bar\alpha}$-boundary
for some $\bar\alpha\in (0,1]$, and study weak (sub-, super-)solutions $u\in H^1(\Omega)$   of
\begin{equation}\label{defeq}
\ldu(x)=f(x),\quad x\in\Omega.
\end{equation}
In most of our results we make the following standard assumptions on the coefficients:
\begin{eqnarray}
&\hskip 1cm\hbox{$A(x)$ is a  matrix such that $\Lambda I \ge A\ge\lambda I$ in $\Omega$, for some $\Lambda\ge \lambda>0$,} \label{hyp1} \\
\noalign{\vskip 1mm}
&\hskip 1cm\hbox{$A, b_1\in C^{\alpha}(\Omega)$} , \quad\hbox{$ b_2, c,f \in L^q(\Omega)$, for some $q\in(n,\infty], \alpha\in(0,1)$,} \label{hyp2}
\end{eqnarray}
which guarantee that the solutions of \eqref{defeq} are continuously differentiable.

The following (a priori, regularity) classical estimates for \eqref{defeq} are fundamental in the theory of elliptic PDE:
\begin{itemize}
\item for (sub-)solutions: \begin{itemize} \item $L^\infty$/Lipschitz estimate;
   \item $L^\infty$-to-$C^1$ (and $C^{1,\alpha}$) estimates; \end{itemize}
\item and for positive (super-)solutions:
\begin{itemize} \item (quantitative) strong maximum principle/weak Harnack inequality; \item (quantitative) Hopf-Oleinik lemma;
\item logarithmic gradient estimates/differential Harnack inequalities.
\end{itemize}
\end{itemize}

The main goal of this work is to obtain versions of these estimates in which {\it the constant is optimized with respect to the norms in \eqref{hyp2} of the coefficients of $\ld$ and to the size of $\Omega$.}

First, we prove a Lipschitz estimate which in particular provides an optimized constant in the classical  Stampacchia-Trudinger generalized maximum principle, and strongly improves the already available bound. As a consequence we obtain optimized gradient and $C^{1,\alpha}$-bounds for solutions of the Dirichlet problem.

Then we proceed to optimize estimates for positive solutions such as  the Hopf-Oleinik lemma and the  global Harnack inequality, for which no such optimal results were available at all. Even without the optimality considerations, our results on positive solutions are new in the generality considered here. Furthermore, we prove a new logarithmic gradient bound for equations with right-hand side.

Our research was initially triggered by a series of recent works  (details on the available literature will be given below) on the so-called Landis conjecture, in which deep connections between this old and hard conjecture and optimized $C^1$ and logarithmic $C^1$-estimates were uncovered; however, these estimates were proved only for  particular cases of \eqref{defdiv}, and under the additional assumption that the coefficients are bounded or $n=2$. As an application of our results we extend the estimates from these papers to the general setting above, with a simpler proof, and get new quantitative estimates related to the Landis conjecture.

Another application are optimized upper and lower bounds for the first eigenvalue of the elliptic operator.
Finally, our theorems are complemented by  counterexamples which prove their optimality.

The next three sections contain our main results. In Section \ref{sec-main} we state our theorems for the important case when the domain is a ball -- then the optimal estimates have simple forms in terms of the norms of the coefficients and the size of the ball. These estimates are extended to arbitrary $C^{1,\bar\alpha}$-domains in Section \ref{sec-ext} where, as we will see, the geodesic diameter of the domain plays an important role. The statements and a discussion on the optimality of our results is given in Section \ref{sec-optim}. Sections \ref{sec-prelim}-\ref{sec-app} are devoted to proofs and auxiliary results.

\bigskip

 We fix a few notations to be used throughout the paper. We always assume that
 $$\bar\alpha\in (0,1], \quad n<q\le \infty, \quad\mbox{ and }\quad  0<\alpha< \min\{1-n/q,\bar\alpha\}.$$
 We denote with $c,C>0$ (possibly with indices) generic positive constants which may depend only on $n,\lambda,\Lambda, q,\alpha,\bar\alpha$, and may change from line to line. The distance to the boundary of $\Omega$ is denoted by $d=d(x)=\mathrm{dist}(x,\partial\Omega)$.
We  denote by $\|\cdot\|_{L^q(\Omega)}$ the usual Lebesgue norm,
by $\|\cdot\|_{L^q_d(\Omega)}$ the Lebesgue norm weighted by $d$,
and by $[\cdot]_{\alpha,\Omega}$ the usual H\"older seminorm on~$\Omega$.
We will use the following  uniformly local Lebesgue norm: for  $q\in[1,\infty]$, $r>0$,
\begin{equation}\label{deful}
\|h\|_{q,r,\Omega}:=\sup_{x\in \Omega} \|h\|_{L^q(\Omega\cap B_{{r}}(x))},
\quad h\in L^q(\Omega),
\end{equation}
(of course $\|h\|_{\infty,r,\Omega}=\|h\|_{L^{\infty}(\Omega)}$ for any $r>0$), and the uniformly local H\"older bracket
\begin{equation}\label{defHbracket}
[\psi]_{\alpha, r,\Omega} := \sup_{x\in\overline{\Omega}}\ \sup_{y,z\in B_{r}(x)\cap \Omega} |y-z|^{-\alpha}|\psi(y)-\psi(z)|,
\quad \psi\in C(\overline\Omega),\quad  \alpha\in(0,1).
\end{equation}
As an advantage on the usual $L^q$-norms for finite $q$, uniformly local $L^q$-norms essentially measure
the local integrability features of a function and do not deteriorate in large domains for functions which do not decay for large $|x|$
(one may think of a periodic function).
These norms have been used in the study of global existence
of solutions of Navier-Stokes or parabolic equations, e.g.~in the classical papers \cite{Ka,GV} (for $r=1$)
and more recently in \cite{HOS,IsSa,MaTe} where the possibility to vary $r$ was also exploited.

Our estimates depend on the quantity
\begin{equation}\label{defM}
M=M(\ld,\Omega) =
[A]_{\alpha,r_0,\Omega}^{\frac{1}{\alpha}} + \bigl\|b_1\bigr\|_{L^\infty(\Omega)} + [b_1]_{\alpha,r_0,\Omega}^{\frac{1}{1+\alpha}}+ \bigl\|b_2\bigr\|^{\beta_q}_{q,r_0,\Omega}+\|c\|^{\gamma_q}_{q,r_0,\Omega},
\end{equation}
$$\mathrm{where}\quad\beta_q=\frac{1}{1-{n}/{q}},\quad \gamma_q=\frac{1}{2-{n}/{q}} \qquad \left(\beta_\infty=1, \quad \gamma_\infty = {1}/{2}\right).$$
To avoid heavy notations at this stage we postpone the definition and properties of the number $r_0=r_0({ \ld,\Omega})$ to  Section~\ref{sec-ext} below;  to facilitate readers' understanding of the  theorems in the following section we note that they
remain valid if $r_0$
is replaced by $1$ and $M$ by $\max(M,1)$, which should be sufficient for most applications
(see Remark~\ref{r01M}; however the optimality of the estimates is guaranteed only with $r_0$).

\section{Main results. The case of a ball}\label{sec-main}

For readability, we will first state our theorems in the important case when $\Omega=B_R$ is a ball.
In the next section we will extend these results to general domains, and in Section \ref{sec-optim} we will show their optimality with respect to the coefficients and the radius of the ball.

\subsection{Optimized estimates for unsigned solutions}

Our first main result has two parts. First, we obtain a Lipschitz estimate which in particular establishes the optimal constant in the classical  generalized maximum principle, \cite[Theorem 8.16]{GT}. Second, we specify the optimal constant, in terms of $R$ and the norms in \eqref{hyp2}, in the classical $C^{1,\alpha}$-estimates (\cite[Section~8.11]{GT}, \cite[Section 5.5]{Mo}) for the Dirichlet problem.

\begin{thm} \label{thm1}
Let $R>0$.
Assume \eqref{hyp1}-\eqref{hyp2} with $\Omega=B_R$,
${\rm div}(b_1)+c\le 0$ in $\mathcal{D}'(B_R)$,
and let $M=M(\ld,B_R)$ {\rm(}cf.~\eqref{defM}{\rm)}.
\smallskip

(i) For each $f\in L^q(B_R)$, any solution  $v\in H^1(B_R)$ of
\be\label{ineqR}
-\mathcal{L}v\le f \; \mbox{ in } B_R,\qquad
v\le0 \; \mbox{on } \partial B_R,
\ee
satisfies
\be\label{ineqG0}
 \sup_{B_R} \frac{v}{d} \le C_0\,e^{C_0MR}\,R^{1-\frac{n}{q}} \|f^+\|_{L^q(B_R)}.
\ee

(ii) The unique solution $u\in H^1_0(B_R)$ of the problem
\be\label{solvR}
-\mathcal{L} u=f \; \mbox{ in } B_R,\qquad
u=0 \; \mbox{on } \partial B_R,
\ee
satisfies
\begin{equation}\label{sharpC1alphapositeig}
 R^{-1}\|u\|_{L^\infty(B_R)}+  \|\nabla u\|_{L^\infty(B_R)}+ R^{\alpha}[\nabla u]_{\alpha,B_R}
\le C_0\,e^{C_0MR}\, R^{1-\frac{n}{q}}\|f\|_{L^q(B_R)}.
\end{equation}
\end{thm}

\begin{rem}
 Theorem \ref{thm1} is  a consequence of either of  Theorems \ref{thm3} and \ref{thm1gen} below, which contain more general and precise statements. We  display the result in its simplest form above for the readers' convenience and since it may be sufficient for applications.
\end{rem}

To illustrate the improvement that Theorem \ref{thm1} (i) brings over \cite[Theorem 8.16]{GT}, let us  take $R=1$, $\mathcal{L}_0u = \mathrm{div}(A(x)Du) + b\cdot\nabla u$ in $B_1$ with $b\in L^q(B_1)$, $n<q\le\infty$. An examination of the proofs in \cite[Section~8.5]{GT} shows they produce a constant $C$ in \cite[(8.39)]{GT} whose dependence in $M_0=\|b\|_{L^q(B_1)}$ grows like $e^{(M_0)^\sigma}$ for $\sigma=1+\frac{\bar{n}(q-2)}{2(q-\bar{n})}\ge 1+\bar{n}/2>2$ ($\bar n>2$ if $n=2$, else $\bar n =n$), while Theorem \ref{thm1} (i) says we may take $ \sigma=1$; and this is optimal
by~Proposition~\ref{Prop-Optim-Linfty}(ii) below. The proof of \cite[Theorem 8.16]{GT} uses a reduction to \cite[Theorem 8.15]{GT} which in turn is based on  an iteration procedure; in contrast, our proof of \eqref{ineqG0} relies on  a different and new idea which involves an optimized version of the global Harnack inequality from \cite{GSS} and a duality argument.

In addition, in $M$ above only locally uniform Lebesgue  norms of the coefficients appear, which to our knowledge is completely new. In particular, the dependence in $R$ is further considerably improved  (note for instance $\|1\|_{L^q(B_R)} = c_0(n,q) R^{n/q}$ while $\|1\|_{q,1,B_R} = c_0(n,q) $). We also observe that the estimates in this section are exact both for small and large values of $R$ (they are actually scale invariant with respect to $R$).

As for Theorem \ref{thm1} (ii), a result of this type appeared only recently in \cite{LeB} for the particular case $A=Id$, $b_1=0$, $b_2, c, f$ bounded, $q=\infty$, $c\le 0$, where the weaker estimate $\|u\|_{C^{1}(B_R)}
\le  e^{C_0{KR}}\,\|f\|_{L^{\infty}(B_{R})}$, $R\ge 1$,  $K=1+\|b_2\|_{L^{\infty}(B_{R})}+ \|c\|_{L^{\infty}(B_{R})}^{1/2}$ for \eqref{solvR} was proved (in this very particular case of Theorem \ref{thm1} (ii) the operator is also in non-divergence form).
The restrictions on the operator come from the somewhat involved proof in~\cite{LeB}, which uses tools such as doubling of variables and existence of a positive ``multiplier" with log-gradient bounds (see below). Our proof of  \eqref{sharpC1alphapositeig} combines \eqref{ineqG0} with Theorem \ref{thm2} below, which in turn relies on a scaling argument.

\begin{rem} \label{remGT}
We observe  that \cite[Theorem 8.16]{GT} is not a Lipschitz estimate like \eqref{ineqG0}, but only an $L^\infty$-estimate (that is, in the left-hand side of \cite[(8.39)]{GT} there is no $d(x)$). The Lipschitz estimate requires some smoothness of the principal coefficients of the operator, such as H\"older continuity close to the boundary, while the  $L^\infty$-estimate
without optimized constants is proved in \cite[Theorem 8.16]{GT} only assuming that $A$ is uniformly elliptic and bounded. It is an open problem whether an optimized estimate such as \eqref{ineqG0} (without $d(x)$) can be proved under that hypothesis only. On the other hand, we will see below in Theorem~\ref{thm1gen} that the hypotheses on the lower-order coefficients can be relaxed in the first part of Theorem \ref{thm1}, which is valid if $b_1, b_2\in L^q(B_R)$, $c\in L^{q/2}(B_R)$ (this is the optimal Lebesgue integrability assumption for an $L^\infty$-estimate), with a  simplification of the quantity $M$. Also, the dependence of the constant in the gradient estimate \eqref{sharpC1alphapositeig} in $[b_1]_{\alpha,r_0,B_R}^{1/(1+\alpha)}$ and $\|c\|^{\gamma_q}_{q,r_0,B_R}$ will be improved to linear rather than exponential. 
\end{rem}

The second part of Theorem \ref{thm1} is a consequence of the first part and the following optimized $C^0$-to-$C^{1,\alpha}$ estimate.

\begin{thm} \label{thm2}
Assume \eqref{hyp1}-\eqref{hyp2} with $\Omega=B_R$,
and that $u\in H^1_0(B_R)$ solves \eqref{solvR}.
Then $u\in C^{1,\alpha}(B_R)$ and
\begin{equation}\label{sharpC1}
\|\nabla u\|_{L^\infty(B_R)}
\le C_0\Bigl({ (M+R^{-1})}\, \|u\|_{L^\infty(B_R)}
+ (M+R^{-1})^{\frac{n}{q}-1}\|f\|_{ q,r_0,B_R}\Bigr),\end{equation}
\begin{equation}\label{sharpC1alpha} [\nabla u]_{\alpha, B_R}
\le  C_0\Bigl({(M+R^{-1})}^{1+\alpha}\,\|u\|_{L^\infty(B_R)}
+ (M+R^{-1})^{\frac{n}{q}-1+\alpha}\|f\|_{ q,r_0,B_R}\Bigr),
\end{equation}
where $M=M(\ld,B_R)$ and $r_0=r_0(\ld,B_R)$ {\rm(}cf.~\eqref{defM}{\rm)}.
\end{thm}

We next state a more general result than Theorem \ref{thm1}, in which
the assumption ${\rm div}(b_1)+c\le 0$ in $\mathcal{D}'(B_R)$ is replaced by the uniform positivity of the first eigenvalue $\lambda_1=\lambda_1(-\ld,\Omega)$ of the operator $\ld$ (see the appendix for properties of $\lambda_1$). We recall  that the positivity of $\lambda_1$ is equivalent to the validity of the maximum principle for $\ld$ in $\Omega$ and ensures the general solvability of the Dirichlet problem for \eqref{defeq}.

\begin{thm} \label{thm3}
{ Let $R>0$, $\kappa\in(0,R]$.}
Assume \eqref{hyp1}-\eqref{hyp2} with $\Omega=B_{R+\kappa}$,
\be\label{hypLambda1}
\lambda_1(-\mathcal{L},B_{R+\kappa})\ge 0
\ee
and let $M=M(\ld,B_{R+\kappa})$ {\rm(}cf.~\eqref{defM}{\rm)}.
\smallskip

(i) For each $f\in L^q(B_R)$, any solution  $v\in H^1(B_R)$ of \eqref{ineqR}
satisfies
\be\label{estimzinftysubs}
{\sup_{B_R} \frac{v}{d}} \le e^{C_0(M+\kappa^{-1})R}\,R^{1-\frac{n}{q}}\|f^+\|_{L^q(B_R)}.
\ee

\smallskip

(ii) The unique solution $u\in H^1_0(B_R)$ of \eqref{solvR} satisfies
$$
R^{-1}\|u\|_{L^\infty(B_R)}+\|\nabla u\|_{L^\infty(B_R)}+R^{\alpha}[\nabla u]_{\alpha,B_R}
\le e^{C_0(M+\kappa^{-1})R}\, R^{1-\frac{n}{q}}\|f\|_{L^q(B_R)}.
$$
\end{thm}

 The optimality in $M$ and $R$ of the estimates in Theorem~\ref{thm3} (and Theorem~\ref{thm1})
is established in Proposition~\ref{Prop-Optim-Linfty}.

A discussion on the hypothesis \eqref{hypLambda1} is in order. The validity of a gradient estimate as in \cite{LeB} (see above) was conjectured in \cite[Conjecture 4.3]{LeB} only assuming uniqueness for the Dirichlet problem for $\ld$ in $B_R$ (which follows from the  maximum principle, that is, $\lambda_1(-\mathcal{L},B_{R})>0$).
 Such a result cannot hold, nor can \eqref{ineqG0} be true,
since it {can be seen}
that the solution of $-\Delta u  -\lambda u = 1$ in $B_1$, $u=0$ on $\partial B_1$ explodes together with its gradient as $\lambda\nearrow \lambda_1(-\Delta, B_1)$, while  $\lambda_1(-\Delta-\lambda, B_1) = \lambda_1(-\Delta, B_1) - \lambda >0$. Some ``uniform positivity" of the first eigenvalue is necessary and this is what \eqref{hypLambda1} expresses,  in terms of nonnegativity of the eigenvalue in a  larger {\it fixed} domain. An interesting consequence of Theorem \ref{thm3} is a quantification of the gap between the eigenvalues in two such domains,
Theorem~\ref{thm4} below.

\goodbreak

\subsection{Applications to the Landis conjecture and bounds for the first eigenvalue}

\subsubsection{Estimates for the first eigenvalue of $\ld$}

Theorem \ref{thm3} easily implies a lower bound on the first eigenvalue in~$B_R$ under \eqref{hypLambda1}. Thus we quantify the increase of the first eigenvalue of the general operator in \eqref{defdiv} when a neighbourhood of the boundary of the domain is removed. A more general result is given in Theorem \ref{thm3gen} (iii) below.

\begin{thm} \label{thm4}
Let $R>0$, $\kappa\in(0,R]$.
Assume \eqref{hyp1}-\eqref{hyp2}  with $\Omega=B_{R+\kappa}$ and~\eqref{hypLambda1}.
Then we have
\be\label{estimLambdaR}
\lambda_1(-\mathcal{L},B_R) \ge e^{-C_0(M+\kappa^{-1})R}R^{-2},
\ee
where $M=M(\ld,B_{R+\kappa})$ {\rm(}cf.~\eqref{defM}{\rm)}.
\end{thm}

Next, we prove an upper bound for the first eigenvalue in $B_R$, which is optimized with respect to the coefficients. Apart from being interesting per se, it plays an important role in the proof of Theorem  \ref{thm3}.

\begin{prop} \label{upperboundlambda1}
Let $R>0$, $\Omega=B_R$ and assume \eqref{hyp1}-\eqref{hyp2}.
Then
\be\label{upperbdeig1}
\lambda_1(-\ld,B_R) \le C_0\bigl(M+R^{-1}\bigr)^2,
\ee
where $M=M(\ld,B_R)$ (cf. \eqref{defM}).
\end{prop}
For operators {\it in non-divergence form with bounded coefficients and $R\le1$}, estimate \eqref{upperbdeig1} was proved in \cite[Lemma 1.1]{BNV}.
Both the lower and the upper estimates for the first eigenvalue are optimal  with respect to $M$ and $R$ -- see Proposition~\ref{Optim-spectral1}.

\subsubsection{Quantitative estimates related to  the Landis conjecture}
In \cite{KL}, Kondratiev and Landis conjectured that a nontrivial solution of a uniformly elliptic PDE with bounded coefficients in
the whole space or in an exterior domain cannot decay faster than exponentially at infinity, that is,
\be\label{limsupexp}
\limsup_{|x|\to\infty} e^{K|x|}|u(x)|>0
\ee
for some $K>0$ depending on bounds on the coefficients.
This property is known as "the Landis conjecture", also in its sharper form where the optimal $K$ is sought for,
or in a weaker form  brought up by Kenig in \cite{K2},
where the decay to rule out is $\exp(-|x|^{1+\epsilon})$, $\epsilon>0$.
Here we will restrict our attention to the whole space case.

The Landis conjecture has a long history, in particular in the case $n=2$, but is still largely unsolved for $n\ge3$ and operators which do not satisfy the maximum principle.  We refer to \cite{BK}, \cite{EK}, \cite{K2}, \cite{DW}, \cite{B1}.
\cite{KSW}  \cite{DKW}, \cite{KW}, \cite{DZ},  \cite{R},
\cite{ABG},  \cite{LM}, \cite{LeB}, \cite{LeBS}, \cite{DP}.
In  \cite{SS2},
under the assumption that  $\mathcal{L}$ satisfies the maximum principle in all bounded subdomains of $\rn$,
we proved the Landis conjecture  in any dimension, with the stronger estimate \begin{equation}\label{concllandis2}
 \sup_{|x|=R} |u(x)|\ge e^{-C_0KR},\quad R\ge R_0.
\end{equation}
The result in \cite{SS2} is valid
for both divergence and non-divergence operators,
with unbounded lower-order coefficients which are only uniformly locally integrable (and thus bounded coefficients are a very particular case).
To prove \eqref{concllandis2} we introduced a new approach based on a Harnack inequality
with sharp dependence on the coefficients and the size of the domain.

Here we are interested in the possibility to obtain a  more precise quantitative estimate than \eqref{concllandis2}, namely
\be\label{intexp}
\int_{\partial B_R}\, |u| d\sigma \ge e^{-C_0 KR},\quad R\ge R_0.
\ee
It was observed in \cite{LeB} that such an estimate can be obtained via a duality argument;
however, the results and proofs in \cite{LeB} are restricted to  $A=Id$, $b_2=0$, $b_1, c$ bounded with~$c\le 0$.

We will combine the Harnack inequality approach from \cite{SS2}, the duality method from \cite{LeB} and Theorem \ref{thm3}
to prove \eqref{intexp} for the general class of divergence operators \eqref{defdiv} with (even locally) unbounded coefficients,
under
the assumption that the first eigenvalue of $\mathcal{L}$ is positive on all bounded subdomains (which is of course much more general than $ b_2=0, c\le0$).
Apart from applying to very general equations, the Harnack inequality approach considerably simplifies the duality method from \cite{LeB}. In particular, we dissociate the proof of the Landis conjecture and \eqref{intexp} from the
use of the so-called ``multiplier'', which appears as an essential step in many previous works on the Landis conjecture for real solutions (see for instance \cite{KSW}, \cite{DW}, \cite{LeB}). The main property of that multiplier is  the log-gradient estimate below (we study that estimate separately, because of its importance,
although we do not rely on multiplier techniques in our result on the Landis conjecture).

In view of the statement, we denote
$$\|h\|_{q,ul}:=\displaystyle\sup_{x\in \R^n} \|h\|_{L^q(B_{1}(x))},\quad
[\psi]_{\alpha, ul} :=\displaystyle \sup_{y,z\in \R^n,\ 0<|y-z|<1} \textstyle\frac{|\psi(y)-\psi(z)|}{|y-z|^\alpha}.$$

\begin{thm} \label{ThmLandisDuality}
Let $q\in(n,\infty]$ and $\alpha\in(0,1)$.
Assume \eqref{hyp1}, $A, b_2\in C^\alpha(\R^n)$, $b_1, c,f \in L^q_{loc}(\R^n)$.
If  $\lambda_1(-\mathcal{L},\R^n)\ge 0$
and
\begin{equation}\label{defK}
K:=1+[A]_{\alpha,ul}^{\frac{1}{\alpha}} + \bigl\|b_2\bigr\|_{L^\infty(\R^n)} + [b_2]_{\alpha,ul}^{\frac{1}{1+\alpha}}+ \bigl\|b_1\bigr\|^{\beta_q}_{q,ul}+\|c\|^{\gamma_q}_{q,ul}<\infty,
\end{equation}
then any  weak solution  of $\mathcal{L}u=0$ in $\R^n$ satisfies the lower estimate
$$\int_{\partial B_R}\, |u| d\sigma \ge e^{-C_0KR} \int_{B_R} |u|,\quad R\ge 1.$$
\end{thm}

We note that the H\"older regularity assumptions on $b_1, b_2$ in Theorem \ref{ThmLandisDuality} are reversed with respect to~\eqref{hyp2}.
This is due to the use of a duality argument.
We will  actually establish a slightly more precise result in  Theorem~\ref{ThmLandisDualityB} below.

Another remarkable fact about \eqref{intexp} is that the positivity of the first eigenvalue turns out to be {\it necessary} for its validity. In other words, we prove that without assumption on the validity of the maximum principle, the Landis conjecture
can only be expected to hold with the weaker lower bound  \eqref{limsupexp}.

\begin{prop} \label{PropOptimality}
For every $\lambda<0$, there exist bounded coefficients $b, c \in C^\infty(\R^n)$
such that
the operator $\mathcal{L}:=\Delta+b\cdot\nabla+c$ satisfies
$\lambda_1(-\mathcal{L},\R^n)=\lambda$
and $\mathcal{L}u=0$ admits a classical solution $u$ on $\R^n$ such that,
for some sequence $R_i\to\infty$ and some  $ K>0$,
$$u(x)\equiv 0\quad\hbox{on $|x|=R_i$}, \qquad
\mbox{and}\qquad \limsup_{|x|\to\infty} e^{K|x|}|u(x)|=1.$$
\end{prop}

A simple  one-dimensional example  is given by  $u(x)=e^{-x}\cos x$, $\mathcal{L}u=u''+2u'+2u$, whose first eigenvalue in $\mathbb{R}$ is $-2$.

\goodbreak

\subsection{Optimized estimates for positive solutions}

\subsubsection{Optimized quantitative Hopf lemma}

A fundamental property of superharmonic functions is that positivity entails a quantitative version of itself.  Specifically, if $u\ge0$ is $\ld$-superharmonic in a bounded $C^{1,\bar\alpha}$-domain $\Omega\subset \rn$, then for some $c_0=c_0(u,\Omega)>0$
\begin{equation}\label{hopf1}
u(x)\ge c_0\, d(x)\quad x\in\Omega, \qquad (\mbox{unless }\; u\equiv0).
\end{equation}
In other words, if $u$ attains a zero minimum inside $\Omega$ then $u\equiv0$ and if $u$ attains a zero minimum at a boundary point, then the normal derivative of $u$ at this point does not vanish. The latter is the famous Zaremba-Hopf-Oleinik lemma, also called boundary point lemma or boundary point principle. We refer to the recent survey \cite{AN} for a thorough review of the history and ramifications of this important result, and many references.

There are two types of quantifications of \eqref{hopf1}. The first expresses the dependence of $c_0$ in $u$ through the positivity of $u$ itself (specifically, an integral norm of $u$), and leads to estimates of ``weak Harnack" type or ``growth lemmas", which play a fundamental role in the regularity theory of elliptic PDE. Recently, global variants of the classical De Giorgi-Moser weak Harnack inequality (\cite[Theorem 8.18]{GT}) were obtained in \cite{Sir3}, \cite{GSS}. In Section~\ref{sec-prelim}
we present a version of the result in \cite{GSS} with an optimized constant in a general domain, which plays a pivotal role in  our results here.

The second way of quantifying \eqref{hopf1} expresses the dependence of $c_0$ in $u$ through an integral norm of $-\ldu$ instead of $u$. This type of inequality for $\ld = \Delta$ (due to Morel and Oswald, unpublished work) gained a lot of popularity and was used  in various contexts after the work by Brezis and Cabr\'e \cite{BrC}. Specifically, their result states that if $u\ge0$ in a smooth domain $\Omega$, $-\Delta u \ge f \ge0$ in $\Omega$, then
\be\label{HopfLapl}
u(x) \ge C(n, \Omega){\|f\|_{L^1_d(\Omega)}}\,d(x), \quad x\in \Omega, \qquad \mbox{where }\; \|f\|_{L^1_d(\Omega)}:=\int_\Omega |f(y)|d(y)dy.
\ee
This inequality also follows from the properties of the Green function of the Laplacian, see \cite{Zh}.
For extensions to operators div$(A(x)Du)$ with $A$ Lipschitz in a neighborhood of $\partial \Omega$, see \cite{EDR}, to very weak solutions of Schr\"odinger operators, see~\cite{OP}.
Some related results for integro-differential operators and operators in non-divergence form can be found in \cite{KK} and \cite{Sir1}.
We refer also to the references in these works.

As for more general divergence-form operators, to our surprise we did not find a reference for the Morel-Oswald inequality for operators as in \eqref{defdiv}, with first-order coefficients or less regularity of the leading coefficients. No previous results on optimality in the dependence of the constant in the coefficients of the operator  are available at all.

The following theorem establishes the validity of the Morel-Oswald estimate for the general operator in \eqref{defdiv}, and provides an optimal constant in terms of the norms of the coefficients of $\ld$ and the size of the domain.

\begin{thm} \label{OptimizedHopf}
Let $R>0$ and assume \eqref{hyp1}-\eqref{hyp2} with $\Omega=B_R$.
If $u\in H^1(B_R)$ is a weak supersolution
of $-\mathcal{L}u\ge f$ and $f,u\ge 0$ in $B_R$,
we have
\be\label{conclHopf}
u(x)
\ge {e^{-C_0(1+MR)}\,R^{-n}}{\|f\|_{L^1_d(B_R)}} \,
d(x),\quad x\in B_R,
\ee
where $M=M(\ld,B_R)$ {\rm(}cf.~\eqref{defM}{\rm)}.
\end{thm}

Note that we do not assume the solvability of the Dirichlet problem for $\ld$ (or the maximum principle), and there is no assumption for $u$ on the boundary.
The optimality of estimate~\eqref{conclHopf}
is established in Proposition~\ref{Prop-Optim-Hopf1}.

The proof of Theorem \ref{OptimizedHopf} is somewhat delicate. It involves splitting the domain into a controlled number of small balls with controlled radius such that a suitable rescaling of the operator in each ball satisfies the maximum principle in the rescaled ball, then using  ideas from \cite{BrC}
(where the case $\mathcal{L}=\Delta$ was treated),
as well as additional arguments based on variants of the boundary Harnack inequality (which ensure the uniformity of the estimates with respect
to the norm of the nonprincipal coefficients and the domain) and a reduction to operators in simpler form, in order to obtain a variant of the theorem in each ball, and finally observing that the $L^1_d$-norm of $f$ in some  ball is comparable (with a constant depending on the right quantities) to its norm in $B_R$.

\subsubsection{Optimized logarithmic gradient estimate}

Here we study bounds for $|\nabla \log(u)|$, for a positive solution of \eqref{defeq}. Such estimates are often referred to as differential Harnack inequalities (DHI), since in many cases they can be integrated to obtain the usual Harnack inequality. A fundamental DHI for the heat equation is due to Li-Yau \cite{LY}; it has generated huge number of extensions, the method has been successfully adapted
and generalized to various nonlinear geometric heat flows (see for instance the book \cite{RM}). For the elliptic case, the Li-Yau estimate has been extended to semi-linear equations of Lane-Emden type (see for instance \cite{LiJ}, \cite{MZ}, and the references there).  Most frequently DHI are proved by variations of the Bernstein method, which requires smoothness of the coefficients in the equation and is not applicable under our weaker hypotheses.

In our setting, an upper bound for $d|\nabla u|/u$ if $u$ is a positive solution of $\ldu=0$ follows immediately from \eqref{sharpC1} in Theorem \ref{thm2} and \eqref{sharpBHI2} in Theorem \ref{BHIoptim} below. However, the resulting exponential in~$M$ constant is not optimal (even though each of the estimates \eqref{sharpC1} and \eqref{sharpBHI2} is optimal per se). The optimal dependence in $M$ of the upper bound for $d|\nabla u|/u$ is actually linear. That was an important observation in  \cite{KSW} where, as an important step to the proof of the Landis conjecture in two dimensions, it was proved that  if  $-\Delta u + Vu = 0$, $u>0$ in $B_2\subset \mathbb{R}^2$ then the interior estimate $\|\nabla \log(u)\|_{L^\infty(B_1)}\le C_0 \|V\|_{L^\infty(B_2)}^{1/2} (=C_0M)$ is valid.
Extensions to more general two-dimensional operators are given in \cite{B1}, \cite{DW}. The same type of interior estimate for the $n$-dimensional equation $\ldu=0$ is obtained in \cite{LeB},  under the restrictions $A=Id$, $b_1=0$, $b_2, c\in L^\infty$.

In the following theorem we first obtain a log-gradient estimate  with an optimal constant for positive solutions of the full equation $\ldu=0$ with unbounded coefficients, which is also up-to-the-boundary in the domain. Then we proceed to give an estimate for the non-homogeneous equation  $\ldu=f$, which seems to be the first of its kind.
\begin{thm}
\label{lgeu}
Assume \eqref{hyp1}-\eqref{hyp2} with $\Omega=B_R$,
and set $M=M(\ld,B_R)$ {\rm(}cf.~\eqref{defM}{\rm)}.
\smallskip

(i) For any weak solution $u$ of $-\mathcal{L}u=0$, $u>0$ in $B_R$, we have the estimate
\be\label{concllgeu}
\frac{d(x)|\nabla u(x)|}{u(x)}\le C_0\max\{1, M d(x)\},\quad x\in B_R.
\ee
\smallskip

(ii) If $f\in L^q(B_R)$ with $f\ge 0$, $f\not\equiv 0$, then
for any weak solution $u$ of $-\mathcal{L}u=f$, $u>0$ in $B_R$, we have the
estimate
\be\label{concllgeu2}
\frac{d(x)|\nabla u(x)|}{u(x)}\le C_0\Bigl(\max\{1, Md(x)\} + R^n e^{ C_0MR}
\frac{\|f\|_{L^q(B_R)}}{\|f\|_{L^1_d(B_R)}}\, d(x)^{1-\frac{n}{q}}\Bigr),\quad x\in B_R.
\ee
\end{thm}

The optimality of estimate~\eqref{concllgeu}
is established in Proposition~\ref{Optimality_Harnack}.
As for \eqref{concllgeu2}, we only have partial optimality results; see the end of Section~\ref{sec-optim} and Proposition \ref{PropLoggradOptimality}.

\begin{rem}\label{remloggrad1d}
For conciseness we have not included the case $n=1$ in the results above,  since this case is of course not the most challenging, while would require some changes in the definitions.
However, when $n=1$, there are some interesting phenomena in connection with Theorem \ref{lgeu}. Specifically,
for any nonnegative $f\in L^1$ (modifying accordingly the definition of a weak solution), estimate \eqref{concllgeu} is true. In other words, for $n=1$ the maximum of $d|u^\prime|/u$ is bounded above independently of $f$,
see Proposition \ref{propn1}.
On the other hand, \eqref{concllgeu} fails for  $n\ge2$ and $f\gneqq0$, as a consequence of Proposition~\ref{PropLoggradOptimality} below.
\end{rem}

\begin{rem}\label{remloggrad1d2}
Estimate \eqref{concllgeu2} is easily seen to be false without the positivity assumption $f\ge 0$, even for $n=1$, see \eqref{counterfpositive} below.
\end{rem}

We will give two proofs of Theorem \ref{lgeu}. The first one uses the more or less standard method of combining $C^1$ estimates and the Harnack inequality to prove a gradient bound in terms of the quantity $u/d$ (see for instance \cite[p.2864]{Sir3}),
in which we do a careful rescaling to optimize the resulting constant. The second proof is longer but more versatile, since it avoids the use of the Harnack inequality for general operators. Instead, it relies on a contradiction argument and the doubling lemma from \cite{PQS}. This second proof adapts to yield the statement in Remark \ref{remloggrad1d} above.

\medskip

The rest of the paper is organized as follows.
In Section~\ref{sec-ext}, we give our precise definitions and some extensions of the main results,
especially to general $C^{1,\bar\alpha}$ domains.
In Section~\ref{sec-optim},
we present our results which guarantee the optimality of the main theorems
in terms of the norms of the lower order coefficients of the operator
(see Propositions~\ref{Prop-Optim-Linfty}--\ref{Optimality_Harnack};
for simplicity we restrict ourselves to the coefficients $b_2$ and $c$,
which covers the most classical cases of Schr\"odinger and drift operators).
In Section~\ref{sec-prelim} we provide some auxiliary results and important ingredients of our proofs,
especially optimized Harnack inequalities,
as well as properties associated with uniformly local norms.
The proofs of the optimized $L^\infty$ and $C^1$ estimates (Theorems
\ref{thm1}-\ref{thm3} and their extensions) are given in Section~\ref{sec-proof1}.
Our results concerning the Landis conjecture (Theorem~\ref{ThmLandisDuality} and Proposition~\ref{PropOptimality})
and the spectral estimate (Theorem~\ref{thm4})
are proved in Section~\ref{sec-proof2}.
In Sections~\ref{proofHopf} and \ref{proofloggrad}, we respectively prove the optimized quantitative Hopf lemma (Theorems~\ref{OptimizedHopf} and \ref{thm4gen}(i))
and the logarithmic gradient estimate (Theorems~\ref{lgeu} and \ref{thm4gen}(ii)(iii)).
The results about the optimality of {our main estimates
(Propositions~\ref{Prop-Optim-Linfty}--\ref{Optimality_Harnack})
are proved in Sections~\ref{sec-proofoptim} and \ref{sec-proofoptim2}. The latter
also contains} additional results about the optimality of global Harnack and log-grad estimates,
mentioned in Section~\ref{sec-optim}. As we will see, almost each of the main theorems above requires a different, taylor-made, example for its optimality. Finally,  the proofs of the auxiliary results in Section~\ref{sec-prelim}
and of some eigenvalue properties (including Proposition~\ref{upperboundlambda1}),
as well as a suitable integration by parts formula,
are collected in an appendix.

\section{Precise definitions and extensions to general domains} \label{sec-ext}

 Throughout this section  $\Omega$ is a bounded domain with $C^{1,\bar\alpha}$ boundary.
Set
$$D_0:={\rm diam}(\Omega),\quad D:=\mbox{ geodesic diameter of }\Omega \;= \sup_{x,y\in \Omega} \inf_{\sigma\in\Sigma(x,y)} l(\sigma),$$
where $\Sigma(x,y)$ is the set of all paths in $\Omega$ connecting $x$ and $y$, and $l(\sigma)$ is the length of $\sigma$.

We will now define a constant $r_\Omega$ which quantifies the well-known facts that each point of the boundary of a $C^{1,\bar\alpha}$-domain: (i) has a uniform neighborhood in which the domain can be ``flattened", and (ii) can be touched by a $C^{1,\bar\alpha}$-paraboloid with fixed size and opening.

Let $\rho_\Omega>0$ be the supremum of all
$r\le D_0$
such that, for each $\xi\in \partial\Omega$, there is a $C^{1,\bar\alpha}$-diffeomorphism $\Phi$ between the domain $(\Omega\cap B_{r}(\xi) - \xi)/r$ (resp. the surface $(\partial \Omega\cap B_{r}(\xi) - \xi)/r$) and $B_1^+=B_1(0)\cap\{x_n>0\}$ (resp. $B_1^0 = B_1(0)\cap\{x_n=0\}$) such that $D\Phi(0)=I$, $ |D\Phi| \le 2$, $ |D\Phi^{-1}| \le 2$, $[D\Phi]_\alpha\le1$. The existence of such $r>0$ is well known, see for instance \cite{Liebe}.
Also, there exist $\bar\rho_\Omega, k_\Omega>0$ with the following property:
for each $\xi\in \partial\Omega$,
there exists an orthonormal coordinate system $y=(y',y_n)\in\R^{n-1}\times\R$ with origin $\xi$ and a $C^{1,\bar\alpha}$ function
$\varphi$ on $B'_0=\{|y'|<\bar\rho_\Omega\}$, such that,
setting $\Sigma_0=B'_0\times(-\bar\rho_\Omega,\bar\rho_\Omega)$, we have
$\Sigma_0\cap\Omega=\Sigma_0\cap \{y_n>\varphi(y')\}$,
$\Sigma_0\cap\partial\Omega=\Sigma_0\cap \{y_n=\varphi(y')\}$,
$\varphi(0)=\xi$, $D\varphi(0)=0$ and $[D\varphi]_{\bar\alpha}\le k_\Omega$; hence
$$\bigl\{y=(y';y_n)\in \R^{n-1}\times\R:\ |y'|< \bar\rho_\Omega\ \hbox{and}\ k_\Omega |y'|^{1+\bar\alpha}<y_n< \bar\rho_\Omega\bigr\}\subset \Omega.$$
For a very general result on these geometrical considerations see \cite{AM}, in particular Corollary~3.14 in that paper.
We set
\be\label{def_romega}
r_\Omega=\min\bigl(\rho_\Omega, \bar\rho_\Omega/4,(120 k_\Omega)^{-1/\bar\alpha})\bigr).
\ee
 It is easy to see that if $\Omega=B_R$ then we can take $\bar\alpha=1$, and we have
$$
r_{B_R}= c(n)R, \quad \mbox{and}\quad D=2R.
$$
Given an operator $\mathcal{L}$ as in \eqref{defdiv}
satisfying \eqref{hyp1}, \eqref{hyp2}, and recalling definitions \eqref{deful}, {\eqref{defHbracket},
we define the number $r_0 \,=r_0(\ld,\Omega) \in (0,r_\Omega]$ by
\begin{equation}\label{defr0}
r_0:= \sup\Bigl\{r>0:\: r\bigl(r^{-1}_\Omega+ [A]^{\frac{1}{\alpha}}_{\alpha, r, \Omega} + \|b_1\|_{L^\infty(\Omega)} + [b_1]^{\frac{1}{\alpha+1}}_{\alpha, r, \Omega}+\|b_2\|^{\beta_q}_{q,r,\Omega} + \|c\|^{\gamma_q}_{q, r, \Omega}\bigr)\le 1\Bigr\},
\end{equation}
which} in turn gives the precise definition of the constant $M$ from \eqref{defM} and the locally uniform spaces
from \eqref{deful}, \eqref{defHbracket},
that appear in the theorems above.
We warn the reader that $\|\cdot\|_{q,r_0,\Omega}$
does {\it not} behave as a norm when applied to the coefficients of the operator themselves,
due to the dependence of $r_0$ on these coefficients (see Proposition \ref{lemr0ul}
below.)

\smallskip

Next we give an extension of Theorem~\ref{thm1} which uses the following variant of $M, r_0$:
\be\label{defMstar}
M_*(\ld,\Omega)=[A]^{\frac{1}{\alpha}}_{\alpha, r_*, \Omega} +
\|b_1+b_2\|^{\beta_q}_{q,r_*,\Omega},
\ee
\be\label{defrstar}
r_*=r_*(\ld,\Omega)= \sup\bigl\{r>0:\: r\bigl(r^{-1}_\Omega+ [A]^{\frac{1}{\alpha}}_{\alpha, r, \Omega} +
\|b_1+b_2\|^{\beta_q}_{q,r,\Omega}\bigr)\le 1\bigr\}.
\ee

\begin{thm} \label{thm1gen}
(i) Assume
$A\in  C^{\alpha}(\Omega)$ satisfies \eqref{hyp1},
$$ b_1, b_2 \in L^q(\Omega),\ c\in L^{q/2}(\Omega),\qquad{\rm div}(b_1)+c\le 0\quad\hbox{ in $\mathcal{D}'(\Omega)$.}$$
Then for each $f\in L^q(\Omega)$, any solution  $v\in H^1(\Omega)$ of
\be\label{ineqRR0}
-\mathcal{L}v\le f \; \mbox{ in } \Omega,\qquad
v\le0 \; \mbox{on } \partial \Omega,
\ee
satisfies (with $M_*=M_*(\ld,\Omega)$ defined in \eqref{defMstar})
\be\label{ineqG02}
\sup_{\Omega} \frac{v}{d}
\le e^{C_0(r^{-1}_\Omega+M_*)D} D^{1-\frac{n}{q}} \|f^+\|_{L^q(\Omega)}.
\ee

(ii) Assume \eqref{hyp1}-\eqref{hyp2} and div$(b_1)+c\le 0$ in $\mathcal{D}'(\Omega)$. Then the solution of
\be\label{solvOmega}
-\mathcal{L} u=f \; \mbox{ in } \Omega,\qquad
u=0 \; \mbox{on } \partial\Omega
\ee
satisfies (with $M=M(\ld,\Omega)$ defined in \eqref{defM} and $M_*=M_*(\ld,\Omega)$ defined in \eqref{defMstar})
$$
D^{-1}\|u\|_{L^\infty(\Omega)}+\|\nabla u\|_{L^\infty(\Omega)}+D^\alpha[\nabla u]_{\alpha,\Omega}
\le  e^{C_0(r^{-1}_\Omega+ M_*)D} \bigl(r_\Omega^{-1}+M\bigr)D^{2-\frac{n}{q}}\|f\|_{L^q(\Omega)}.
$$
\end{thm}

 As mentioned in Remark~\ref{remGT},
Theorem~\ref{thm1gen} improves Theorem~\ref{thm1} not only with respect to the domain, but also at the level of the weaker assumptions on the coeffcients $b_1,c$.
This is a consequence of a comparison argument used in the proof of Theorem~\ref{thm1gen}
which relies on
a problem with pure drift.
Also no explicit dependence on the zero order coefficient appears; the latter is actually ``hidden'' in the
fact that $c\le -{\rm div}(b_1)$.  Furthermore, the constant in the gradient estimate is exponential only in $M_*$, while being just  linear in the larger quantity $M$.

Theorem~\ref{thm1gen} will be a consequence of the following Theorems \ref{thm2gen}-\ref{thm3gen}.
These are extensions of Theorems \ref{thm2}-\ref{thm4}, and contain an optimized $C^0$-to-$C^{1,\alpha}$ estimate and an optimized $L^\infty$/Lipschitz estimate under
a spectral assumption on a larger domain~$\Omega'$.

\begin{thm} \label{thm2gen}
Assume \eqref{hyp1}-\eqref{hyp2}.
If $u\in H^1_0(\Omega)$ solves \eqref{solvOmega} then $u\in C^{1,\alpha}(\Omega)$ and
\begin{equation}\label{sharpC1gen}
\|\nabla u\|_{L^\infty(\Omega)}\le
C_0\Bigl({ \bigl(M+r_\Omega^{-1}\bigr)} \|u\|_{L^\infty(\Omega)}
+ \bigl(M+r_\Omega^{-1}\bigr)^{\frac{n}{q}-1}\|f\|_{q,r_0,\Omega}\Bigr),
\end{equation}
\begin{equation}\label{sharpC1alphagen}
[\nabla u]_{\alpha, \Omega}
\le  C_0\Bigl({\bigl(M+r_\Omega^{-1}\bigr)}^{1+\alpha} \|u\|_{L^\infty(\Omega)}
+ \bigl(M+r_\Omega^{-1}\bigr)^{\frac{n}{q}-1+\alpha}\|f\|_{q,r_0,\Omega}\Bigr),
\end{equation}
where $M=M(\ld,\Omega)$, $r_0=r_0(\ld,\Omega)$ are defined in \eqref{defM}, \eqref{defr0}.
\end{thm}

The formulation of Theorem \ref{thm3gen} is slightly more technical.
  Given a domain $\hat\Omega$ in $\mathbb{R}^n$, numbers $\kappa>0$, $q>n$, $p=q/2$ and $b_1,b_2\in L^q(\hat\Omega)$, $c\in L^{q/2}(\hat\Omega)$,
 we define the following weaker analogues of the quantities $M,r_0$
\begin{equation}\label{defMhat}
\hat M=\hat M(\ld,\hat\Omega,\kappa) =
\bigl\|b_1\bigr\|^{\beta_q}_{q,\hat r_0,\hat\Omega}+\bigl\|b_2\bigr\|^{\beta_q}_{q,\hat r_0,\hat\Omega}+\|c\|^{\gamma_p}_{q,\hat r_0,\hat\Omega},
\end{equation}
\be\label{defr0hat}
\hat r_0 =\hat r_0(\ld,\hat\Omega,\kappa)
:= \sup\Bigl\{r>0:\: r\bigl(\kappa^{-1}+\|b_1\|^{\beta_q}_{q,r,\hat\Omega}+ \|b_2\|^{\beta_q}_{q,r,\hat\Omega}+\|c\|^{\gamma_p}_{p,r,\hat\Omega}\bigr)\le 1\Bigr\}.
\ee

\begin{thm} \label{thm3gen}
 Assume \eqref{hyp1}-\eqref{hyp2}. Let   $\kappa\in(0,D]$, and
 $\hat\Omega\supset \Omega$ be such that  $\kappa={\rm dist}(\overline\Omega,\hat\Omega^{c})$.
 Assume  that the coefficients of $\mathcal{L}$ can be extended to $\hat\Omega$ so that
 \be\label{hyp1a22}
A\in  L^\infty(\hat\Omega)\; \mbox{ with }\;\Lambda I \ge A\ge\lambda I, \qquad
  b_1, b_2 \in L^q(\hat\Omega),\ c\in L^p(\hat\Omega),
   \ee
 as well as
\be\label{hypLambda1b2}
\lambda_1(-\mathcal{L},\hat\Omega)\ge 0.
\ee
Let $M=M(\ld,\Omega)$, $r_0=r_0(\ld,\Omega)$, $\hat M=\hat M(\ld,\hat\Omega,\kappa)$ be defined in
 \eqref{defM},  \eqref{defr0},  \eqref{defMhat}. Then

(i) for each $f\in L^q(\Omega)$, any solution $v\in H^1(\Omega)$ of \eqref{ineqRR0}
 satisfies
\be\label{estimzinftysubsb}
\sup_{\Omega} \frac{v}{d}
\le e^{C_0(M+\hat M+r_\Omega^{-1}+\kappa^{-1})D} D^{1-\frac{n}{q}} \|f^+\|_{L^q(\Omega)};
\ee

(ii) the unique solution $u\in H^1_0(\Omega)$ of problem \eqref{solvOmega} satisfies
\begin{equation}\label{sharpC1alphapositeigGen}
D^{-1}\|u\|_{L^\infty(\Omega)}+\left\|\nabla u\right\|_{L^\infty(\Omega)}+D^\alpha [\nabla u]_{\alpha,\Omega}
\le e^{C_0(M+\hat M+r_\Omega^{-1}+\kappa^{-1})D} D^{1-\frac{n}{q}}\|f\|_{L^q(\Omega)}.
\end{equation}

(iii) we have
$$
\lambda_1(-\mathcal{L},\Omega)\ge e^{-C_0(M+\hat M+r_\Omega^{-1}+\kappa^{-1})D} D^{-2}.
$$
\end{thm}

Finally, we give extensions to general domains of the results on positive solutions in the previous section, which we gather in the following theorem.

\begin{thm}\label{thm4gen}
 Assume \eqref{hyp1}-\eqref{hyp2}
and let $M$, $r_0$ be defined in \eqref{defM}, \eqref{defr0}.

\smallskip

(i) If $u\in H^1(\Omega)$ is a weak supersolution
of $-\mathcal{L}u\ge f$ and $f,u\ge 0$ in $\Omega$, then
\be\label{conclHopfgen}
u(x)
\ge
e^{-C_0(r^{-1}_\Omega+ M)D}D^{-n}
 {\|f\|_{L^1_d(\Omega)}} \, d(x),\quad x\in \Omega.
\ee

(ii) For any weak solution $u$ of $-\mathcal{L}u=0$, $u>0$ in $\Omega$, we have the estimate
$$
\frac{d(x)|\nabla u(x)|}{u(x)}\le C_0\max\bigl\{1, (r_\Omega^{-1}+M)d(x)\bigr\},\quad x\in \Omega.
$$

(iii) If $f\in L^q(\Omega)$ with $f\ge 0$, $f\not\equiv 0$, then
for any weak solution $u$ of $-\mathcal{L}u=f$, $u>0$ in~$\Omega$, we have the
estimate
$$
\frac{d(x)|\nabla u(x)|}{u(x)}\le C_0\max\bigl\{1,(r_\Omega^{-1}+M)d\bigr\} +
e^{C_0(r^{-1}_\Omega+ M)D} D^n
\frac{\|f\|_{L^q(\Omega)}}{\|f\|_{L^1_d(\Omega)}}\, d^{1-\frac{n}{q}}(x),\quad x\in \Omega.
$$
\end{thm}

\section{Optimality of the main theorems} \label{sec-optim}

We begin with the optimality of the $L^\infty$-estimate in Theorems \ref{thm1} and  \ref{thm3}
with respect to the coefficients and the size of the domain.
It is verified for both pure Schr\"odinger type operators and pure drift operators.
We restrict to the case $\kappa=R$ for simplicity.

\begin{prop} \label{Prop-Optim-Linfty}
(i) For each $R>0$ there exist sequences $c_i\in C^\infty(\R^n)$ and $f_i\in C^\infty(\overline B_R)$, $f_i\not\equiv0$,
such that $\|c_i\|_{q,r_i,B_{2R}}\to\infty$ with $r_i=r_0(\ld_i,B_{2R})$,
the operator $\ld_i:=-\Delta+c_i$ satisfies
$\lambda_1(-\ld_i,\R^n)\ge 0$,
and the classical solution $u_i$ of
\be \label{optimLinfty1}
\left\{\hskip 2mm\begin{aligned}
-\Delta u_i+c_iu_i&=f_i, &\quad&x\in B_R,\\
u_i&=0, &\quad&x\in \partial B_R
\end{aligned}
\right.
\ee
satisfies, for some $C=C(n)>0$,
\be \label{optimLinfty2}
{\Bigl(\frac{u_i}{d}\Bigr)(0)\ge CR^{1-n/q}}e^{CM_iR}\|f_i\|_{L^q(B_R)}
\quad\hbox{with $M_i:=M(\ld_i,B_{2R})={\|c_i\|^{\gamma_q}_{q,r_i,B_{2R}}}$}.
\ee

(ii) For each $R>0$ there exist sequences $b_i\in C^\infty(\R^n)$ and $0\not\equiv f_i\in C^\infty(\overline B_R)$
such that $\|b_i\|_{q,r_i,B_{2R}}\to\infty$ with {$r_i=r_0(\ld_i,B_{2R})$},
{the operator $\ld_i:=-\Delta+b_i\cdot\nabla$ satisfies} $\lambda_1(-\ld_i,\R^n)\ge 0$,
and the classical solution of
\be \label{optimLinfty11}
\left\{\hskip 2mm\begin{aligned}
-\Delta u_i+b_{i}\cdot \nabla u_i&=f_i, &\quad&x\in B_R,\\
u_i&=0, &\quad&x\in \partial B_R
\end{aligned}
\right.
\ee
satisfies \eqref{optimLinfty2} with $M_i=\|b_i\|^{\beta_q}_{q,r_i,B_{2R}}$. This also shows the estimate in Theorem \ref{thm1} is optimal.
\end{prop}

 We will also see that the same example as in Proposition \ref{Prop-Optim-Linfty} (ii) can be used to show the optimality of the gradient bounds in Theorem \ref{thm2} for \eqref{optimLinfty11}. For more details and examples for problems as in \eqref{optimLinfty1}, see Remark \ref{gradopt} below.

Next, we turn to the optimality of the spectral bounds in Theorem~\ref{thm4} and Proposition~\ref{upperboundlambda1}.

\begin{prop} \label{Optim-spectral1} (i) For each $R>0$ there exist sequences $b_i, c_i\in C^\infty(\overline{B_{2R}})$
such that the operator $\ld_i:=-\Delta+c_i$ (resp. $-\Delta+b_i\cdot\nabla$) satisfies
\be \label{optim-spectralineq1}
0\le \lambda_1(-\mathcal{L},B_{2R}) \le\lambda_1(-\mathcal{L},B_R) \le \tilde C e^{-CM_iR}R^{-2},
\ee
with $M_i:=M(\ld_i,B_{2R})=\|c_i\|^{\gamma_q}_{q,r_i,B_{2R}}$ (resp., $\|b_i\|^{\beta_q}_{q,r_i,B_{2R}}$),
where $r_i=r_0(\ld_i,B_{2R})$, $M_i\to\infty$, and $C, \tilde C>0$ depend only on $n$.

\smallskip

(ii) Assertion (i) remains true with $B_{2R}$ replaced by $B_R$ and \eqref{optim-spectralineq1} replaced by
$$\lambda_1(-\mathcal{L},B_R) \ge  C(M_i+R^{-1})^2.$$
\end{prop}

We now state the optimality of the Hopf estimate in Theorem~\ref{OptimizedHopf}.

\begin{prop} \label{Prop-Optim-Hopf1}
For all $R, M>0$, there exist
$c\in C^\infty(\overline B_R)$ and $f\in C^\infty(\overline B_R)$ with $c\ge 0$, $f>0$,
such that the classical solution $u>0$ of
\be \label{optimHopf1}
\left\{\hskip 2mm\begin{aligned}
-\ld u:=-\Delta u+cu&=f, &\quad&x\in B_R,\\
u&=0, &\quad&x\in \partial B_R
\end{aligned}
\right.
\ee
satisfies, for some $C=C(n)>0$,
\be \label{optimHopf2}
\ds\inf_{\partial B_R} |u_\nu| \le C(n)R^{-n}e^{-MR/4} \int_{B_R}fd,
\ee
with $M=M(\ld,B_R)=\|c\|^{\gamma_q}_{q,r_0,B_R}$, $r_0=r_0(\ld,B_R)$, $u_\nu$ is the normal derivative of $u$ at $\partial B_R$.

(ii) Assertion (i) remains valid if problem \eqref{optimHopf1} is replaced by
$$
\left\{\hskip 2mm\begin{aligned}
-\ld u:=-\Delta u+b\cdot\nabla u&=f, &\quad&x\in B_R,\\
u&=0, &\quad&x\in \partial B_R
\end{aligned}
\right.
$$
and
$\|b\|^{\beta_q}_{q,r_0,B_R}=M.$
\end{prop}

 As we already noted, we will make essential use of optimized versions of the Harnack-type results in \cite{GSS}, which are stated below in \eqref{sharpWBHI}-\eqref{sharpBHI2}. Because of the importance of these estimates, we discuss their optimality too. First, the optimality of the constant in the global weak Harnack estimate \eqref{sharpWBHI2} is trivial and follows similarly to the optimality of the interior estimate (see \cite[Remark~2.3]{SS2}),  by considering one-dimensional solutions of $u^{\prime\prime} +cu = 0$ or $u^{\prime\prime} +bu^\prime = 0$. The same examples show the optimality of the  constant in the log-grad estimate \eqref{concllgeu} in the case $f=0$,  if we do not require a Dirichlet boundary condition.

Next, we will prove the
optimality of the constant in the global Harnack inequality \eqref{sharpBHI2} (when $u=0$ on $\partial\Omega$ is required). The same example
can be used to show that the log-grad estimate \eqref{concllgeu} in the case $f=0$ is optimal even for solutions which vanish on the boundary.

\begin{prop}\label{Optimality_Harnack}
Let $R>0$.
There exists a constant
$\lambda_0=\lambda_0(n)>0$,  sequences $b_i, c_i\in C^\infty(\overline{B_R})$,
and a classical solution $u_i>0$ of
\be \label{optimLinfty1ci}
\left\{\hskip 2mm\begin{aligned}
-\ld_iu_i:=-\Delta u_i+c_iu_i&=0, &\quad&x\in B_R,\\
u_i&=0, &\quad&x\in \partial B_R,
\end{aligned}
\right.
\ee
resp. of
\be \label{optimLinfty1bi}
\left\{\hskip 2mm\begin{aligned}
-\ld_iu_i:=-\Delta u_i + b_i\cdot\nabla u_i - \lambda_0R^{-2} u_i&=0, &\quad&x\in B_R,\\
u_i&=0, &\quad&x\in \partial B_R,
\end{aligned}
\right.
\ee
such that $\|b_i\|_{q,r_i,B_R}$, $\|c_i\|_{q,r_i,B_R}\to\infty$ with $r_i=r_0(\ld_i,B_R)$,
\be \label{optimLinfty23}
\frac{\sup_{B_R}(u_i/d)}{\inf_{B_R}(u_i/d)}\ge Ce^{CM_iR}
\ee
and
\be \label{optimLinfty2b}
\frac{|\nabla u_i|}{u_i}\ge CM_i\quad\hbox{on $\{|x|=R/2\}$},
\ee
where $C=C(n)>0$ and $M_i=\|c_i\|^{\gamma_q}_{q,r_i,B_R}$, resp. $M_i=\|b_i\|^{\beta_q}_{q,r_i,B_R}$.
\end{prop}

\begin{rem}\label{remoptloggrad}
We observe that, under the assumptions of Theorem \ref{lgeu}(i), we always have
\be \label{optimLinfty2b2}
\lim_{|x|\to R}\frac{d(x)|\nabla u(x)|}{u(x)}=1,
\ee
and \eqref{optimLinfty2b}-\eqref{optimLinfty2b2} thus
show the optimality of the log-grad estimate \eqref{concllgeu}.
\end{rem}

\begin{rem} We  will  see
that the exponential constant in front of the right-hand side $f$ in the Harnack inequality \eqref{sharpBHI2}
 below is optimal with respect to the coefficients too (see Proposition~\ref{Optimality_Harnack2}).
On the other hand, for the log-grad estimate \eqref{concllgeu2} in the non-homogeneous case (for which our result is the first of its kind), we do not know whether the exponential dependence in the coefficients in the constant in the second term of the estimate is optimal.
\end{rem}

\smallskip

Even though our focus here is on the optimal dependence of the constants in the coefficients and the size of the domain, it is also worth recalling  and discussing the much more classical question of the optimality of the Lebesgue norms which appear in the estimates. It is very well known that  $C^1$-estimates fail for coefficients or right-hand sides which are only in $L^n$ instead of $L^q$, $q>n$ (for instance $u(x)=|x|\log|\log|x||$ solves $\Delta u = f\in L^n$, and variations and smooth approximations of that function can be used to violate our estimates for right-hand sides in $L^n$).

\smallskip

Further, the norm $\|f\|_{L^1_d}$ in the right-hand side of the quantitative Hopf inequality \eqref{conclHopf} cannot be replaced by $\|f\|_{L^p_d}$ nor by $\|f\|_{L^1_{d^\epsilon}}$, for $p>1$ or $\epsilon<1$. This follows from classical estimates for the Green function of the Laplacian, see Remark~\ref{Hopf-optL1d}.
\smallskip

Finally, the question whether  the norm $\|f\|_{L^1_d}$ in the right-hand side of the log-grad estimate
\eqref{concllgeu2} can be improved appears to be more delicate. We have only partial results. In particular, we are able to prove that, in the case $n=2$, for any $p>1$ the estimate fails if $\|f\|_{L^1_d}$ is replaced by $\|f\|_{L^p_d}$ and $q$ is strictly larger but close to $n$. In other words, in this situation $C^1$-estimates are available but the log-grad estimate fails. See Proposition~\ref{PropLoggradOptimality} for  detailed statement.

\section{Auxiliary results} \label{sec-prelim}

\subsection{Optimized Harnack inequalities}
\label{precisedef}

We will make important use of the following optimized version of the recent global weak Harnack inequality, proved in \cite{GSS}.

\begin{thm}[Optimised BHI] \label{BHIoptim}
 Let $\Omega$ be a bounded $C^{1,\bar\alpha}$-domain of $\mathbb{R}^n$ with  geodesic diameter $D$, and let  $r_\Omega$ be defined in \eqref{def_romega}. Assume \eqref{hyp1}-\eqref{hyp2}.
There exist  constants $\epsilon, C_0>0$ depending only on $n,q,\alpha,\lambda,\Lambda$, such that if $u\ge0$ in $\Omega$ is a weak solution of $-\mathcal{L}u\ge f $ in $\Omega$, for some $f\in L^q(\Omega)$,
then
\begin{equation}\label{sharpWBHI}
\left(\int_{\Omega} \left(\frac{u}{d}\right)^\epsilon\right)^{1/\epsilon} \le
D^{n/\epsilon}
e^{C_0(r^{-1}_\Omega+ M)D} \left( \inf_{\Omega} \frac{u}{d} +
D^{1-\frac{n}{q}} \|f\|_{L^q(\Omega)}\right),
\end{equation}
where $M=M(\ld,\Omega)$ is defined in \eqref{defM}.
If $-\mathcal{L}u= f $, $u\ge0$ in $\Omega$, $u=0$ on $\partial \Omega$, then
\begin{equation}\label{sharpBHI}
\sup_{\Omega} \frac{u}{d}\le
e^{C_0(r^{-1}_\Omega+ M)D}\left( \inf_{\Omega} \frac{u}{d} +
D^{1-\frac{n}{q}} \|f\|_{L^q(\Omega)}\right).
\end{equation}
\end{thm}

In the particular case $\Omega=B_R$, we trivially have
\begin{equation}\label{scaleinvtilde1b}
r_{B_R}=c(n)R,\quad D=D_0=2R,
\end{equation}
and estimates  \eqref{sharpWBHI}, \eqref{sharpBHI} become
\begin{equation}\label{sharpWBHI2}
\left(-\hskip -4.2mm\int_{B_R} \left(\frac{u}{d}\right)^\epsilon\right)^{1/\epsilon} \le
 R^{n/\epsilon} e^{C_0(1+MR)} \left( \inf_{B_R} \frac{u}{d} + R^{1-\frac{n}{q}} \|f\|_{L^q(B_R)}\right),
\end{equation}
\begin{equation}\label{sharpBHI2}
\sup_{B_R}\frac{u}{d}\le
e^{C_0(1+MR)} \left( \inf_{B_R} \frac{u}{d} +
R^{1-\frac{n}{q}}  \|f\|_{L^q(B_R)} \right).
\end{equation}

The proof of this theorem uses a Harnack chain argument similar to the one in the
proof of \cite[Theorem 2.1]{SS2}, replacing Theorem A there by \cite[Theorem 1.1]{GSS},
combined with a (sharp) estimate of the length of the chains in terms of the geodesic diameter (Proposition~\ref{geodes}).
For readers' convenience we list the technical differences and give a proof sketch of Theorem~\ref{BHIoptim} in the appendix.

We will also use a slightly more precise version of the interior Harnack inequality
with optimized constants from \cite[Theorem 2.1]{SS2}.
It follows from straightforward modifications of the proof in that paper, by adjusting the sizes and centers of the balls used to cover
$\Omega$ and by using Proposition~\ref{geodes}(i).

\begin{prop}\label{sharpharn}
 Let $\Omega_1\subset \Omega_2$ be bounded domains  such that  $\kappa={\rm dist}(\overline\Omega_1,\Omega_2^{c})\in(0,D]$ and $\Omega_1$ has geometric diameter $D>0$.
Assume that
$A(x)\in L^\infty(\Omega_2)$ is a matrix such that $\Lambda I \ge A\ge\lambda I$ in $\Omega_2$ and that
$b_1, b_2 \in L^q(\Omega_2)$, $c, g\in L^{q/2}(\Omega_2)$, for some $q>n$.
If $u\ge0$ satisfies $\lu= g$ in $\Omega_2$, then
\begin{equation}\label{sharpHarnack}
\sup_{\Omega_1} u\le e^{C_0(\kappa^{-1}+ \hat M)D}\left( \inf_{\Omega_1}u + D^{1-\frac{n}{q}} \|g\|_{L^q(\Omega_2)}\right),
\end{equation}
where $\hat M=\hat M(\ld,\Omega_2,\kappa)$ is defined in \eqref{defMhat}.
\end{prop}

\subsection{Properties associated with uniformly local norms}

In this subsection we give some useful properties of the
the uniformly local norms associated with the quantities $r_0, M$ and their variants.

\begin{prop}\label{basicr0}  Recalling definitions \eqref{def_romega}, \eqref{defr0}, \eqref{defrstar}, \eqref{defr0hat}, we have
\begin{equation}\label{relMr0}
r_0^{-1}=r^{-1}_\Omega+ M,
\end{equation}
\be\label{relMstarr0}
r_*^{-1}=r^{-1}_\Omega+M_*,
\ee
\be\label{relMhatr0}
\hat r_0^{-1}=\kappa^{-1}+\hat M.
\ee
In particular
$\textstyle\frac12 \min\bigl(r_\Omega,M^{-1}\bigr)\le r_0\le\min\bigl(r_\Omega,M^{-1}\bigr)$,
and similarly for $r_*, \hat r_0$.
\end{prop}

\begin{rem} \label{r01M}
As mentioned above, all the results in Section~\ref{sec-main} remain valid if $r_0$ is replaced by $1$ and $M$ by $\max(M,1)$.
Indeed, if $r_0\le 1$, then  $\|\cdot\|_{q,r_0,\Omega}\le \|\cdot\|_{q,1,\Omega}$
and $[\cdot]_{\alpha, r_0, \Omega}\le [\cdot]_{\alpha, 1, \Omega}$,
whereas if $r_0>1$, then $M<1$ owing to \eqref{relMr0}.
\end{rem}

We give a simple monotonicity property for general domains.

\begin{prop}\label{scaleinvtilde0}
Let $\omega\subset\Omega$ be bounded domains with $C^{1,\bar\alpha}$ boundaries,
$\mathcal{L}$ be an operator as in \eqref{defdiv} satisfying \eqref{hyp1}, \eqref{hyp2},
 and let $\theta\in(0,1]$. There exists $C(n,q,\alpha,\bar\alpha,\theta)>0$ such that,
 if $r_\Omega\ge\theta r_\omega$, then
$M(\ld,\omega)\le C(n,q,\alpha,\theta)M(\ld,\Omega)$.
\end{prop}

We now consider the case of balls.
It is immediate that, if $u$ is a solution of $\mathcal{L}u=f$ in $\Omega=B_R$, then $\tilde u(x)=u(Rx)$ is a solution of
$\tilde{\mathcal{L}}u=\tilde f$ in $B_1$,
where $\tilde f$ and the coeffcients of $\tilde{\mathcal{L}}$ are given by
\begin{equation}\label{scaledcoeff}
\tilde A(x)=A(Rx),\quad \tilde b_i(x)=Rb_i(Rx),\quad \tilde c(x)=R^2c(Rx),\quad \tilde f(x)=R^2f(Rx).
\end{equation}
In the next proposition, after giving a basic monotonicity property of the quantity $M$ with respect to $R$,
we show that
$r_0, M$, as well as each of the (uniformly local)
norms and seminorms appearing in $M$
enjoy some natural invariance properties with respect to the scaling transformation \eqref{scaledcoeff}.
As a consequence, all main estimates in Section~\ref{sec-main}
are scale-invariant with respect to $R$, and
in particular the general case $R>0$ can be reduced to the case $R=1$.

\begin{prop}\label{scaleinvtilde}
Let $R>0$ and assume \eqref{hyp1}-\eqref{hyp2} with $\Omega=B_R$.
\smallskip

{(i) We have $M(\mathcal{L},B_\rho)\le M(\mathcal{L},B_R)$ for all $\rho\in(0,R)$.
\smallskip

(ii) Let $f\in L^q(B_R)$} and let $\tilde f$ and the coefficients of the operator $\tilde{\mathcal{L}}$ be given by \eqref{scaledcoeff}.
Let $\tilde r_0$ be defined by \eqref{defr0} with $\mathcal{L},B_R$
replaced by $\tilde{\mathcal{L}},B_1$,
Then
\begin{equation}\label{scaleinvtilde2}
\tilde r_0=r_0/R,
\end{equation}
\begin{equation}\label{scaleinvtilde2b}
[\tilde A]_{\alpha, \tilde r_0, B_1}= R^{\alpha}[A]_{\alpha, r_0, B_R},\,
[\tilde b_1]_{\alpha, \tilde r_0, B_1}= R^{1+\alpha}[b_1]_{\alpha, r_0, B_R},\,
 \|\tilde b_1\|_{L^\infty(B_1)}=R\|b_1\|_{L^\infty(B_R)},
\end{equation}
\begin{equation}\label{scaleinvtilde4}
\|\tilde b_2\|_{q,\tilde r_0,B_1}=R^{1-\frac{n}{q}}\|b_2\|_{q,r_0,B_R},\
\|\tilde c\|_{q,\tilde r_0,B_1}=R^{2-\frac{n}{q}}\|c\|_{q,r_0,B_R},\
\|\tilde f\|_{q,\tilde r_0,B_1}=R^{2-\frac{n}{q}}\|f\|_{q,r_0,B_R}.
\end{equation}
Consequently,
\begin{equation}\label{scaleinvtilde5}
M(\tilde{\mathcal{L}},B_1)=RM(\mathcal{L},B_R).
\end{equation}
\end{prop}

We next
consider the properties of the quantity $\|\cdot\|_{q,r_0,B_R}$
(for $q\in(n,\infty)$)
when applied to the coefficients of the operators themselves:
in this situation it does not obey the scaling of a norm
but a {\it nonlinear} scaling (due to the dependence of $r_0$ on the coefficients of the operators).
Here we restrict to a smaller class of operators, for simplicity and since it will be enough for our purposes.

\begin{prop} \label{lemr0ul}
Let $R>0$ and $b, c\in L^\infty(B_R)$.
\smallskip

(i) Let $\mathcal{L}=\Delta+b(x)\cdot\nabla +c(x)$.
Then we have
\be\label{bulinfty}
\|b\|_{q,r_0,B_R} \le C(n) \|b\|^{1-n/q}_{L^\infty(B_R)}
\ee
and
\be\label{culinfty}
\|c\|_{q,r_0,B_R} \le C(n) \|c\|^{1-n/2q}_{L^\infty(B_R)}.
\ee

(ii) Consider the family of operators
$\mathcal{L}_\lambda=\Delta+\lambda b(x)\cdot\nabla$
(resp., $\Delta+\lambda c(x)$) for $\lambda>0$,
and
set $r_\lambda:=r_0(\mathcal{L}_\lambda,B_R)$ {\rm(}cf.~\eqref{defr0}{\rm)}.
Then we have
\be\label{culinfty2}
C_1(n)\lambda^{1-n/q} \|b\|_{q,r_1,B_R}\le \|\lambda b\|_{q,r_\lambda,B_R} \le
C_2(n) \lambda^{1-n/q}\|b\|^{1-n/q}_{L^\infty(B_R)} ,\quad \lambda\ge 1
\ee
$$({\rm resp.,}\quad C_1(n)\lambda^{1-n/2q} \|c\|_{q,{r_1},B_R}
\le \|\lambda c\|_{q,{r_\lambda},B_R} \le
C_2(n) \lambda^{1-n/2q}\|c\|^{1-n/2q}_{L^\infty(B_R)} ,\quad \lambda\ge 1).$$
Moreover the function $\lambda\mapsto
\|\lambda b\|_{q,r_\lambda,B_R}$
is continuous on $(0,\infty)$.
\end{prop}

\section{Proofs of optimized $L^\infty$ and $C^1$ estimates (Theorems \ref{thm1}-\ref{thm3} and \ref{thm1gen}-\ref{thm3gen})}

\label{sec-proof1}

We begin with the proof of Theorem~\ref{thm3gen}(i).

\begin{proof}[Proof of Theorem \ref{thm3gen}(i)]
We will use the adjoint operator of $\mathcal{L}$, whose expression is
\begin{equation}\label{adj}
\ld^*[u] = \mathrm{div}(A^T(x)Du -  b_2(x)u) - b_1(x){\cdot} Du +c(x) u.
\end{equation}
Set $\lambda_1:=\lambda_1(-\ld^*,{ \hat\Omega})=\lambda_1(-\ld,{ \hat\Omega})$.
We will use the following notation (cf.~\eqref{defM}, \eqref{defr0} and
\eqref{defr0hat}-\eqref{defMhat}):
$$M_0 =M(\ld,\Omega),\quad r_0 =r_0(\ld,\Omega),\quad
\hat M_0 =\hat M(\ld,\hat\Omega,\kappa),\quad \hat r_0 =\hat r_0(\ld,\hat\Omega,\kappa),$$
$$\hat M_1 =\hat M(\ld+\lambda_1,\hat\Omega,\kappa),\quad \hat r_1 =\hat r_0(\ld+\lambda_1,\hat\Omega,\kappa).$$
Note that $\hat M(\ld,\hat\Omega,\kappa)=\hat M(\ld^*,{ \hat\Omega},\kappa)$, and similarly we can replace $\ld$ by $\ld^*$ in $\hat r_0$, $\hat M_1$, $\hat r_1$. We have $\lambda_1\ge 0$ by assumption and
there exists a first eigenfunction $\phi_1\in H^1_0({ \hat\Omega})$, $\phi_1>0$, satisfying:
\be\label{supersolPhib}
\ld^*_1\phi_1:=(\ld^*+\lambda_1)\phi_1=0, \quad x\in { \hat\Omega}.
\ee
By \cite[Theorem 8.29]{GT} (see also the remark at the end of \cite[Section 8.10]{GT}) we know that $\phi_1\in C(\overline\Omega)$. By the interior Harnack inequality
with optimized constants in Proposition~\ref{sharpharn}, we have
\be\label{supersolPhi2b}
\sup_{\Omega} \phi_1 \le e^{C_0(\kappa^{-1}+\hat M_1)D} \inf_{\Omega}\phi_1.
\ee
Since $\lambda_1(-\ld,\Omega)>\lambda_1(-\ld,{ \hat\Omega})\ge 0$, the problem
\be\label{solwR}
\left\{\hskip 2mm\begin{aligned}
-\mathcal{L}w&=f^+, &\quad&x\in \Omega,\\
w&=0, &\quad&x\in \partial\Omega
\end{aligned}
\right.
\ee
has a unique solution, and it satisfies $w \ge 0$.
By  \cite[Theorem 8.33]{GT}, \cite[Theorem 5.5.5']{Mo}, in view of assumptions \eqref{hyp1}-\eqref{hyp2},
we have $w\in C^1(\overline\Omega)$.
We may thus integrate by parts, applying Lemma~\ref{weakdiv2}
in whose statement $\ld$ is replaced by $\ld^*$, $u=\phi_1$ and $v=w$,
to obtain
$$\begin{aligned}
\int_{\Omega} \phi_1 f^+
&= -\int_{\Omega} \phi_1\mathcal{L}w
=- \int_{\Omega} w \mathcal{L}^*\phi_1 -\int_{\partial\Omega}\nu\cdot A(x)\phi_1\nabla w\, d\sigma\\
&=\lambda_1 \int_{\Omega} w\phi_1 -\int_{\partial\Omega}\nu\cdot A(x)\phi_1\nabla w\, d\sigma
\ge -\int_{\partial\Omega}\nu\cdot A(x)\phi_1\nabla w\, d\sigma.
\end{aligned}$$
Observe that $\nabla w=(\partial_\nu w)\nu$, $\partial_\nu w\le 0$ on $\partial\Omega$, and $\nu\cdot A(x)\nu\ge \lambda$.
Consequently,
$$ \lambda \inf_{\Omega} \phi_1  \int_{\partial\Omega} |\partial_\nu w|\, d\sigma
\le \|f^+\|_{L^1(\Omega)}\sup_{\Omega} \phi_1.$$
Using \eqref{supersolPhi2b}, we deduce that
\be\label{estimwnu0}
\inf_{\partial\Omega} |\partial_\nu w|\le \frac{1}{|\partial\Omega|} \int_{\partial\Omega} |\partial_\nu w|\, d\sigma
\le  \frac{ \|f^+\|_{L^1(\Omega)}}{\lambda|\partial\Omega|} \frac{\sup_{\Omega} \phi_1}{\inf_{\Omega} \phi_1}
\le e^{C_0(\kappa^{-1}+\hat M_1)D}\frac{\|f^+\|_{L^1(\Omega)}}{|\partial\Omega|}.
\ee
Now, by the sharp boundary Harnack inequality \eqref{sharpBHI} we have
\be\label{estimwnu}
\sup_{\Omega} \frac{w}{d}\le
e^{C_0(r^{-1}_\Omega+ M_0)D}\left( \inf_{\Omega} \frac{w}{d} +D^{1-\frac{n}{q}} \|f^+\|_{L^q(\Omega)}\right).
\ee
Since $t^{-1}w(x_0-t\nu)\to |\partial_\nu w(x_0)|$ as $t\to 0^+$ for each $x_0\in \partial\Omega$,
it follows from \eqref{estimwnu0}, \eqref{estimwnu} that
\be\label{estimwd0}
\begin{aligned}
\sup_{\Omega} \frac{w}{d}
&\le e^{C_0(r^{-1}_\Omega+ M_0)D}\left(\inf_{\partial\Omega} |\partial_\nu w|+D^{1-\frac{n}{q}} \|f^+\|_{L^q(\Omega)}\right)\\
&\le e^{C_0(r^{-1}_\Omega+ M_0)D}
\left(e^{C_0(\kappa^{-1}+\hat M_1)D}\frac{\|f^+\|_{L^1(\Omega)}}{|\partial\Omega|} +D^{1-\frac{n}{q}} \|f^+\|_{L^q(\Omega)}\right).
\end{aligned}
\ee

We now proceed to simplify the quantity on the right hand side of \eqref{estimwd0}.
We first claim that
\be\label{estimhatM}
\hat M_1
\le C(M_0+\hat M_0+r_\Omega^{-1}+\kappa^{-1}).
\ee
If $\hat r_1\ge \min(r_0,\hat r_0)$, then
$\hat M_1=\hat r_1^{-1}-\kappa^{-1}\le \max(r_0^{-1},\hat r_0^{-1})\le M_0+r_\Omega^{-1}+\hat M_0+\kappa^{-1}$
by \eqref{relMr0} and \eqref{relMhatr0}.
We may thus assume
\be\label{estimhatr1}
\hat r_1\le \min(r_0,\hat r_0).
\ee
Owing to the definition of $r_\Omega$,
$\Omega$~contains some ball $B$ of radius $r=c(n)r_\Omega$, hence
$M(\mathcal{L},B)\le CM(\mathcal{L},\Omega)=CM_0$ by Proposition~\ref{scaleinvtilde0}
 with $\omega=B$ and \eqref{scaleinvtilde1b}.
It then follows from Proposition~\ref{upperboundlambda1} and \eqref{relMr0} that
$$\lambda_1=\lambda_1(\mathcal{L},{ \hat\Omega})\le \lambda_1(\mathcal{L},B)
\le C\bigl(M(\mathcal{L},B)+r^{-1}\bigr)^2
\le C\bigl(M_0+r_\Omega^{-1}\bigr)^2=Cr_0^{-2}.$$
Using \eqref{estimhatr1}, $1/\gamma_q=2-n/q$ and \eqref{relMr0} again, we deduce that
$$
\begin{aligned}
\hat M_1
&= \|b_1\|^{\beta_q}_{q,\hat r_1,{ \hat\Omega}}+ \|b_2\|^{\beta_q}_{q,\hat r_1,{ \hat\Omega}}+\|c+\lambda_1\|^{\gamma_q}_{q,\hat r_1,{ \hat\Omega}}\\
&\le  \|b_1\|^{\beta_q}_{q,\hat r_0,{ \hat\Omega}}+ \|b_2\|^{\beta_q}_{q,\hat r_0,{ \hat\Omega}}+C\|c\|^{\gamma_q}_{q,\hat r_0,{ \hat\Omega}}
+C(r_0^{n/q}\lambda_1)^{\gamma_q}
\le C\hat M_0+C\bigl(M_0+r_\Omega^{-1}\bigr),
\end{aligned}$$
hence \eqref{estimhatM}.
On the other hand, by the definition of $r_\Omega$, it is easy to see that $|\partial\Omega|\ge c(n)r_\Omega^{n-1}$
which, combined with H\"older's inequality and $r_0\le r_\Omega\le D_0\le D$, yields
\be\label{estimhatM2}
\frac{\|f^+\|_{L^1(\Omega)}}{|\partial\Omega|}
\le C(n)r_\Omega^{1-n}D_0^{n(1-\frac{1}{q})}\|f^+\|_{L^q(\Omega)}
\le C(n)(Dr_0^{-1})^n D^{1-\frac{n}{q}}\|f^+\|_{L^q(\Omega)}.
\ee
Combining \eqref{estimwd0}, \eqref{estimhatM}, \eqref{estimhatM2} and observing that
$(Dr_0^{-1})^n=(D(M_0+r^{-1}_\Omega))^n\le e^{C_0(M_0+r^{-1}_\Omega)D}$ in view of \eqref{relMr0}, we obtain
\be\label{estimwnu2}
\sup_{\Omega} \frac{w}{d}
\le e^{C_0(M_0+\hat M_0+r^{-1}_\Omega+\kappa^{-1})D} D^{1-\frac{n}{q}} \|f^+\|_{L^q(\Omega)}.
\ee
Finally, for any solution $u\in H^1(\Omega)$ of \eqref{ineqRR0},
since $\lambda_1(\mathcal{L},\Omega)> \lambda_1(\mathcal{L},{ \hat\Omega})\ge 0$,
we may apply the maximum principle to get $u\le w$,
hence \eqref{estimwnu2} yields estimate \eqref{estimzinftysubsb}.
\end{proof}

We next prove Theorem~\ref{thm2gen}, of which Theorem \ref{thm2} is a special case.

\begin{proof}[Proof of Theorem~\ref{thm2gen}]
Fix any $x_0\in \overline\Omega$, $r\in (0,r_0]$,
{where $r_0$ is the number from \eqref{defr0},} and let $v(y)=u(x_0+ry)$.
The function $v$ satisfies
$$\tilde{\mathcal{L}}v=\tilde{f}\quad\hbox{ in $\omega:=B_1\cap r^{-1}(\Omega-x_0)$,}$$
where the coefficients of the modified operator $\tilde{\mathcal{L}}$ are $\tilde A(y)= A(x_0+r y)$,
$\tilde b_i(y)=r
b_i(x_0+r y)$, $\tilde c(x)=r^2c(x_0+ry)$, and $\tilde{f}(x)=r^2 f(x_0+ry)$. We compute
\be\label{rescaledLq}
\begin{aligned}
\|\tilde b_2\|_{L^q(\omega)}
&=r\Bigl(\int_{\omega} |b_2(x_0+ry)|^q\,dy\Bigr)^{1/q}\\
&= r^{1-n/q}\Bigl(\int_{\Omega\cap B_{r}(x_0)} |b_2(x)|^q\,dx\Bigr)^{1/q}
\le r^{1-n/q}\|b_2\|_{q,r_0,\Omega}
\end{aligned}
\ee
and similarly $\|\tilde c\|_{L^p(\omega)} \le r^{2-n/p}\|c\|_{p,r_0,\Omega}$,
$\|\tilde g\|_{L^p(\omega)} \le r^{2-n/p}\|g\|_{p,r_0,\Omega}$. Hence, by the definition of $r_0$,
$$\|\tilde b_2\|_{L^q(\omega)}\le 1,\quad \|\tilde c\|_{L^p(\omega)}\le 1$$
and, similarly,
$$[\tilde A]_{\alpha, 1, \omega}\le 1,\quad \|\tilde b_1\|_{L^\infty(\omega)}\le 1,\quad  [\tilde b_1]_{\alpha, 1, \omega}\le 1.$$

If $x_0\in \partial\Omega$, then we take $r=r_0$ and we have
$v=0$ on $T:=B_1\cap \partial(r^{-1}(\Omega-x_0))$.
Moreover, the local charts defining $T$ are bounded in $C^{1,\alpha}$,
uniformly with respect to $x_0\in \partial \Omega$
(due to $r\le r_0\le r_\Omega$ and the definition of $r_\Omega$ at the beginning of Section~\ref{sec-ext}).
By the interior-boundary gradient estimate in \cite[Chapter 8.11]{GT}, \cite[Chapter 5.5]{Mo}, we deduce that
$$\sup_{B_{1/2}\cap  r^{-1}(\Omega-x_0)}
|\nabla v|\le C_0\left(\sup_\omega |v|+ \|\tilde g\|_{L^p(\omega)}\right).$$
Scaling back to $u$ and $f$, we get
$$\sup_{B_{r_0/2}(x_0)\cap \Omega}|\nabla u|\le C_0\left(r_0^{-1}\sup_{B_{r_0}(x_0)\cap \Omega} |u|+
r_0^{1-n/q}\|g\|_{L^p({B_{r_0}(x_0)\cap \Omega})}\right),$$
hence in particular, setting $\Omega'=\{x\in\Omega;\ {\rm dist}(x,\partial\Omega)\le r_0/2\}$,
\begin{equation}\label{EstC1A}
\sup_{\Omega'}|\nabla u|\le C_0\left(r_0^{-1}\sup_{\Omega} |u|
+ r_0^{1-n/q}\|g\|_{p,r_0,\Omega}\right).
\end{equation}

Next, for each $x_0\in \Omega\setminus\Omega'$,
we take $r=r_0/2$, so that $B_r(x_0)\subset \Omega$, hence $\omega=B_1$.
 It follows from the interior gradient estimate in \cite[Chapter 8.11]{GT}, \cite[Chapter 5.5]{Mo}, that
$$\sup_{B_{1/2}} |\nabla v|\le C_0\left(\sup_{B_1} |v|+ \|\tilde g\|_{L^p(B_1)}\right).$$
Scaling back to $u$ and $f$, we deduce that
$$|\nabla u(x_0)|\le 2C_0\left(r_0^{-1}\sup_{B_{r_0}(x_0)} |u|+ r_0^{1-n/q}\|g\|_{L^p({B_{r_0}(x_0)})}\right),$$
hence
\begin{equation}\label{EstC1B}
\sup_{\Omega\setminus\Omega'}|\nabla u|\le C_0\left(r_0^{-1}\sup_{\Omega} |u|
+ r_0^{1-n/q}\|g\|_{p,r_0,\Omega}\right).
\end{equation}
Inequality \eqref{sharpC1gen} follows by combining \eqref{EstC1A}, \eqref{EstC1B} and \eqref{relMr0}.
The proof of \eqref{sharpC1alphagen} is similar.
\end{proof}

 We now prove Theorem~\ref{thm3gen}(ii)
as a consequence of Theorems~\ref{thm3gen}(i) and \ref{thm2gen}.

\begin{proof}[Proof of Theorem \ref{thm3gen}(ii)]
Since $\lambda_1(\mathcal{L},\Omega)> \lambda_1(\mathcal{L},\Omega')\ge 0$,
there exists a unique solution $u\in H^1_0(\Omega)$ of the equation $-\mathcal{L} u=f$.
It follows from \eqref{estimzinftysubsb} in Theorem \ref{thm3gen}(i) that
\be\label{estforuh}
\|u\|_{L^\infty(\Omega)}
\le e^{C_0(M+\hat M+r_\Omega^{-1}+\kappa^{-1})D} D^{2-\frac{n}{q}} \|f\|_{L^q(\Omega)},
\ee
since $d(x)\le D$. By combining this inequality with Theorem~\ref{thm2gen}, we obtain
$$\begin{aligned}
\|\nabla u\|_{L^\infty(\Omega)}
&\le C_0\Bigl({ \bigl(M+r_\Omega^{-1}\bigr)}\|u\|_{L^\infty(\Omega)}
+ \bigl(M+r_\Omega^{-1}\bigr)^{\frac{n}{q}-1}\|f\|_{q,r_0,\Omega}\Bigr)\\
&\le {\bigl(M+r_\Omega^{-1}\bigr)} e^{C_0(M+\hat M+r_\Omega^{-1}+\kappa^{-1})D} D^{2-\frac{n}{q}}\|f\|_{L^q(\Omega)}
+ C_0\bigl(M+r_\Omega^{-1}\bigr)^{\frac{n}{q}-1}\|f\|_{q,r_0,\Omega}\\
&=\Bigl({\bigl(M+r_\Omega^{-1}\bigr)D} \,e^{C_0(M+\hat M+r_\Omega^{-1}+\kappa^{-1})D}
+ C_0\bigl((M+r_\Omega^{-1})D\bigr)^{\frac{n}{q}-1}\Bigr) D^{1-\frac{n}{q}}\|f\|_{L^q(\Omega)}\\
&\le e^{C_0(M+\hat M+r_\Omega^{-1}+\kappa^{-1})D}  D^{1-\frac{n}{q}}\|f\|_{L^q(\Omega)},
\end{aligned}$$
where we used $(M+r_\Omega^{-1})D\ge 1$ and $q>n$ in the last inequality.
The proof of the bound for the H\"older bracket of $\nabla u$ is similar, taking $\alpha\le 1-n/q$ into account.
\end{proof}

\begin{proof}[Proof of Theorem \ref{thm3gen}(iii)]
We will apply Theorem \ref{thm3gen}(ii) with $f=1$. Let $u\in H^1_0(\Omega)$ be the solution of $-\mathcal{L} u=1$ in $\Omega$. By the maximum principle $u$ is positive in $\Omega$. By using that $\|1\|_{L^q(\Omega)}= |\Omega|^{1/q}\le D_0^{n/q}\le D^{n/q}$ in  inequality \eqref{estforuh},  we get
$$
-\mathcal{L} u=1\ge e^{-C_0(M+\hat M+r_\Omega^{-1}+\kappa^{-1})D} D^{-2}\,u \quad \mbox{ in }\Omega.
$$
By the characterization \eqref{charlambda1} of the first eigenvalue, the result follows.
\end{proof}

As a consequence of Theorem~\ref{thm3gen}, we now derive Theorem~\ref{thm1gen}.

\begin{proof}[Proof of Theorem \ref{thm1gen}]
(i) By \cite[Theorem~8.1]{GT} or \cite{TR7}, $-\ld$ satisfies the maximum principle.
Replacing $f$ with $f_+$, we may thus assume $f\ge 0$ and $v_{|\partial\Omega}=0$
without loss of generality, hence $v\ge 0$.
Define the operator
$$\hat\ld w:=\mathrm{div}(A\nabla w)+(b_1+b_2)\cdot \nabla w$$
(i.e., $\hat A=A$, $\hat b_1=\hat c=0$, $\hat b_2=b_1+b_2$).
There exists a unique $w\in H^1_0(\Omega)$ such that $-\hat\ld w=f$.
We claim that $v\le w$ in $\Omega$.
Indeed, formally, we have
$$\hat\ld w=-f=\ld v=\mathrm{div}(A\nabla v+b_1 v)+b_2\cdot \nabla v+cv=\hat \ld v+({\rm div}(b_1)+c)v\le \hat\ld v,$$
so the claim should follow from the maximum principle for $-\hat\ld$.
It is standard to make this rigorous, it suffices to argue as in the proof of \cite[Theorem~8.1]{GT} or \cite{TR7},
testing with $\varphi=(v-w-\eps)_+$ and letting $\eps\to 0^+$.

Now, taking $\Omega'=\Omega+B_{r_\Omega}$, hence $\kappa=r_\Omega$, we extend $\hat\ld$ by setting
$\hat A(x)=\lambda I$, $\hat b_1=\hat b_2=\hat c=0$ for $x\in \Omega'\setminus\Omega$.
We will apply Theorem \ref{thm3gen} to the function $w$.
The extended operator $\hat\ld$ satisfies the assumptions of Theorem \ref{thm3gen} (the inequality \eqref{hypLambda1b2}  holds since $\hat\ld 1=0$ in $\Omega'$).
In addition, we have
$$\begin{aligned}
r_0=r_0(\hat\ld,\Omega)
&= \sup\Bigl\{r>0:\: r\bigl(r^{-1}_\Omega+ [\hat A]^{\frac{1}{\alpha}}_{\alpha, r, \Omega} +
\|\hat b_2\|^{\beta_q}_{q,r,\Omega}\bigr)\le 1\Bigr\}\\
&= \sup\Bigl\{r>0:\: r\bigl(r^{-1}_\Omega+ [A]^{\frac{1}{\alpha}}_{\alpha, r, \Omega} +
\|b_1+b_2\|^{\beta_q}_{q,r,\Omega}\bigr)\le 1\Bigr\},
\end{aligned}$$
and we easily check that
$$\begin{aligned}
\hat r_0=\hat r_0(\hat\ld,\Omega',\kappa)
&= \sup\Bigl\{r>0:\: r\bigl(\kappa^{-1}+\|\hat b_2\|^{\beta_q}_{q,r,\Omega'}\bigr)\le 1\Bigr\}\\
&=\sup\Bigl\{r>0:\: r\bigl(r^{-1}_\Omega+\|b_1+b_2\|^{\beta_q}_{q,r,\Omega}\bigr)\le 1\Bigr\}.
\end{aligned}$$
Consequently, $\hat r_0\ge r_0$ and it follows from \eqref{relMr0}, \eqref{relMhatr0} that
$$\kappa^{-1}+\hat M=r^{-1}_\Omega+\hat M=\hat r_0^{-1}\le r_0^{-1}=r^{-1}_\Omega+ M.$$
Inequality \eqref{ineqG02} then readily follows from Theorem \ref{thm3gen}(i)
applied to $w$ and the fact that $v\le w$.

\smallskip

(ii) The proof is similar to that of Theorem \ref{thm3gen}(ii), using
\eqref{ineqG02} instead of \eqref{estimzinftysubsb}.
\end{proof}

Finally, we deduce Theorems~\ref{thm1} and \ref{thm3}
from Theorems~\ref{thm1gen} and \ref{thm3gen}, respectively.

\begin{proof}[Proof of Theorem~\ref{thm1}]
Set $\Omega=B_R$. We claim that $M_0=M(\ld,\Omega)$ and $M_*=M_*(\ld,\Omega)$ satisfy
\be\label{estimMstar}
M_*\le C(r_\Omega^{-1}+M_0).
\ee
If $r_*\ge r_0$, then \eqref{estimMstar} with $C=1$ follows from \eqref{relMr0}, \eqref{relMstarr0}.
Thus assume $r_*\le r_0$. Then
\be\label{estimMstar2}
M_*
\le [A]^{\frac{1}{\alpha}}_{\alpha, r_0, \Omega} + \|b_2\|^{\beta_q}_{q,r_0,\Omega}+ \|b_1\|^{\beta_q}_{q,r_0,\Omega}
\le M_0+\|b_1\|^{\beta_q}_{q,r_0,\Omega}
\ee
and, using $\|b_1\|_{L^\infty(\Omega)}\le M_0\le r_0^{-1}$,
$((n/q)-1)\beta_q=-1$ and \eqref{relMr0}, we obtain
$$
\|b_1\|^{\beta_q}_{q,r_0,\Omega}\le C_0(r_0^{n/q}\|b_1\|_{L^\infty(\Omega)}))^{\beta_q}\le
C_0r_0^{((n/q)-1)\beta_q}=C_0(r^{-1}_\Omega+ M_0),$$
so that \eqref{estimMstar2} implies \eqref{estimMstar}.
Since $D=2R$ and $r_\Omega=c(n)R$,
the conclusion  follows from Theorem \ref{thm1gen}.
\end{proof}

\begin{proof}[Proof of Theorem~\ref{thm3} and Theorem \ref{thm4}.]
By a similar argument as in the proof of Theorem~\ref{thm1},
we obtain that
$\hat M(\ld,B_{R+\kappa},\kappa)\le C(\kappa^{-1}+M(\ld,B_{R+\kappa}))$.
On the other hand, since $\kappa\le R$, we have $M(\ld,B_R)\le CM(\ld,B_{R+\kappa})$ by Proposition~\ref{scaleinvtilde0}.
 Since $D=2R$ and $r_\Omega=c(n)R$,
Theorem \ref{thm3} follows from Theorem~\ref{thm3gen} (ii), while Theorem \ref{thm4} is a consequence of Theorem~\ref{thm3gen} (iii).
\end{proof}

\section{Proofs of Theorem \ref{ThmLandisDuality} and Proposition~\ref{PropOptimality}}
\label{sec-proof2}

We start with the proof of Theorem~\ref{ThmLandisDuality}.
We will actually establish the following, slightly more precise result.

\begin{thm} \label{ThmLandisDualityB}
Let $R\ge1$ and $\kappa\in(0,R]$.
Assume \eqref{hyp1} with $\Omega=B_{R+\kappa}$, \eqref{hypLambda1} and
\be\label{hyp2dual}
A, b_2\in C^{\alpha}(B_{R+\kappa}),\quad b_1, c \in L^q(B_{R+\kappa}),
\quad \hbox{ for some $q\in(n,\infty]$.}
\ee
Any nontrivial weak solution  of $\mathcal{L}u=0$ in $B_{R+\kappa}$ satisfies the lower estimate
$$\int_{\partial B_R}\, |u| d\sigma \ge e^{-C_0 (M_1+\kappa^{-1})R} R^{-1}\int_{B_R} |u|,$$
where $M_1=M(\ld^*,B_{R+\kappa})$ {\rm(}cf.~\eqref{defM}{\rm)}.
\end{thm}

\begin{proof}
Let $u$ be a solution of $\mathcal{L}u=0$ in $B_{R+\kappa}$.
By Theorem \ref{thm3}, in view of \eqref{hyp2dual} and
since
$\lambda_1(-\mathcal{L}^*,B_R)=\lambda_1(-\mathcal{L},B_R)>\lambda_1(-\mathcal{L},B_{R+\kappa})\ge 0$, the adjoint problem
\be\label{solvR1}
\left\{\hskip 2mm\begin{aligned}
-\mathcal{L}^*v_R&=sgn(u), &\quad&x\in B_R,\\
v_R&=0, &\quad&x\in S_R
\end{aligned}
\right.
\ee
has a unique solution and it satisfies
\be\label{solwRestim}
\Bigl\|\frac{v_R}{d}\Bigr\|_\infty\le e^{C_0 (M_1+\kappa^{-1})R}\,R^{1-n/q}\|1\|_{L^q(B_R)}
\le Re^{C_0 (M_1+\kappa^{-1})R}.
\ee
By \cite[Theorem 8.29]{GT} (see also the remark at the end of \cite[Section 8.10]{GT}) we know that $u\in C(\overline B_R)$.
Also, by \cite[Chapter 8.11]{GT}, \cite[Chapter 5.5]{Mo}, in view of assumption \eqref{hyp2}, we have $v_R\in C^1(\overline B_R)$.
We may thus apply Lemma~\ref{weakdiv2} with $v=v_R$, to obtain
$$
\int_{B_R} |u| = - \int_{B_R} u\mathcal{L}^*v_R = -\int_{S_R}\nu\cdot A^T(x)u\nabla v_R\, d\sigma
\le \|A\|_\infty \int_{S_R}\, |u| |\nabla v_R|\,d\sigma.
$$
Since
$v_R=0$ on $S_R$, it follows from \eqref{solwRestim} that
$$|\nabla v_R(x_0)|=\lim_{t\to 0^+} t^{-1}|v_R(x_0-t\nu)|\le Re^{C_0 (M_1+\kappa^{-1})R},\quad\hbox{ for all $x_0\in S_R$.}$$
Consequently,
$$\int_{S_R}\, |u| d\sigma \ge \|A\|_\infty^{-1} R^{-1}e^{-C_0 (M_1+\kappa^{-1})R} \int_{B_R} |u|,$$
which implies the conclusion.
\end{proof}

\begin{proof} [Proof of Theorem~\ref{ThmLandisDuality}.]
Let $R\ge 1$. Applying Theorem~\ref{ThmLandisDualityB} with $\kappa=R$, we have
\be\label{estLandisStar}
\int_{\partial B_R}\, |u| d\sigma \ge e^{-C_0 (1+RM_1)} R^{-1}\int_{B_R} |u| \ge e^{-C_0 R(1+M_1)} \int_{B_R} |u|,
\ee
where $M_1=M(\ld^*,B_{2R})$.
If $r_1=r_0(\ld^*,B_{2R})\le 1$, then $M_1\le K$,
where $K$ is defined in \eqref{defK}, and otherwise $M_1\le r_1^{-1}\le 1\le K$
owing to \eqref{relMr0}. The conclusion thus follows from~\eqref{estLandisStar}.
\end{proof}

\begin{proof}[Proof of Proposition~\ref{PropOptimality}]
We set $v(x)=w(r):=e^{-r}\cos r$ for $r=|x|\ge 1$,
and extend $v$ to a $C^\infty$ function to the whole of $\R^n$, with $v>0$ for $|x|\le 1$.
We next set
$$b_1(x)=\Bigl(2-\frac{n-1}{|x|}\Bigr)\frac{x}{|x|},\quad |x|\ge 1$$
and extend $b_1$ to a $C^\infty$
function on $\R^n$.
We then let
$$c_1(x)=
\begin{cases}
2,&|x|\ge 1\\
\noalign{\vskip 1mm}
v^{-1}\bigl(-\Delta v-b_1\cdot\nabla v\bigr),& |x|<1.
\end{cases}
$$
For $|x|<1$ we have $\mathcal{L}_1v:=\Delta v+b_1\cdot\nabla v+c_1v=0$ by definition.
On the other hand, since $w'(r)=-e^{-r}(\cos r+\sin r)$ and $w''(r)=2e^{-r}\sin r=-2(w'+w)(r)$ for $r>1$, we obtain
$$\Delta v(x)=w''(r)+\frac{n-1}{r}w'(r)=-2w(r)+\Bigl[\frac{n-1}{r}-2\Bigr]w'(r)=
-[b_1\cdot\nabla v+c_1v](x),\quad |x|>1.$$
It follows that $\mathcal{L}_1v(x)=0$ for $|x|>1$ and that the expression in the definition of $c_1(x)$ for $|x|<1$ also
corresponds to $c_1(x)$ for $|x|>1$.
Since $v$ is $C^2$ in $\R^n$, we deduce that $c_1$ is continuous
near $\{|x|=1\}$,
hence in $\R^n$, and that $v$ solves
$\mathcal{L}_1v=0$ in $\R^n$.
\smallskip

Now, since $v(x)=0$ whenever $|x|=(k+\frac12)\pi$ with $k\in\N^*$, it follows from Theorem~\ref{ThmLandisDuality} that
$$\Lambda:=\lambda_1(\mathcal{L}_1,\R^n)<0.$$
For any $\eps>0$, we set $\tilde b_\eps(x)=\eps b_1(\eps x)$, $\tilde c_\eps(x)=\eps^2 c_1(\eps x)$
and we define the operator $\mathcal{L}_\eps:=\Delta+\tilde b_\eps\cdot\nabla+\tilde c_\eps$.
If for some $\mu\in\R$, $\phi$ is a positive solution of $-\mathcal{L}_1\phi\ge \mu\phi$ in $\R^n$, then
the function $\phi_\eps(x)=\phi(\eps x)$ satisfies
$$[-\mathcal{L}_\eps\phi_\eps](x)=-\eps^2\Delta\phi(\eps x)-\eps b_1(\eps x)\cdot\eps[\nabla\phi](\eps x)
-\eps^2 c_1(\eps x)\phi(\eps x)=-\eps^2[\mathcal{L}_1\phi](\eps x)\ge\eps^2\mu\phi_\eps(x).$$
Consequently, by the properties of first eigenvalue (see Appendix),
we have $\lambda_1(-\mathcal{L}_\eps,\R^n)=\eps^2\Lambda$.
But, by the same computation, $u_\eps=v(\eps x)$ is a solution of $\mathcal{L}_\eps u_\eps=0$ in $\R^n$.
Finally choosing $\eps=\sqrt{\lambda/\Lambda}$, the operator $\mathcal{L}=\mathcal{L}_\eps$ and the function
$u=u_\eps$ have all the required properties and the proof of the proposition is complete.
\end{proof}

\section{Proof of the optimized quantitative Hopf lemma (Theorems~\ref{OptimizedHopf} and \ref{thm4gen}(i))}
\label{proofHopf}

Let $\Omega$ be a domain with $C^{1,\bar\alpha}$ boundary and $\tilde \ld$ be an operator of the form \eqref{defdiv} defined in $\Omega$, with coefficients $\tilde A$, $\tilde b_i$, $\tilde c$ satisfying \eqref{hyp1}-\eqref{hyp2}. For  $B_r=B_r(x_0)$, $x_0\in \partial\Omega$, we denote $B_r^+=B_r\cap\Omega$, $B_r^0=B_r\cap\partial\Omega$.
 In view of the proof of Theorem~\ref{thm4gen}(i), we prepare
the following lemma, which provides quantitative interior and boundary Hopf type estimates
for normalized domains and  coefficients with small Lebesgue norms.

\begin{lem} \label{LemOptimizedHopf2}
 Assume that either (i) $\omega= B_2\subset\Omega$ or (ii) $\omega$ is a $C^{1,\bar\alpha}$ domain  such that   $B_1^+\subset\omega\subset B_2^+$ and the $C^{1,\bar\alpha}$ norm of $\partial\omega$ is bounded by $C=C(n,\bar\alpha)>0$. There exist constants $\ep_0, c_0>0$ depending on $n,\lambda,\Lambda, q, \alpha, \bar\alpha$, such that if
\be\label{hyplambda1b}
[\tilde A]_{\alpha,\omega}+\|\tilde b_1\|_{C^\alpha(\omega)}\le1, \qquad  \|\tilde b_1\|_{L^q(\omega)}
+ \|\tilde b_2\|_{L^q(\omega)}+\|\tilde c\|_{L^{q}(\omega)}\le \ep_0,
\ee
and $v\ge0$ is a weak supersolution of $-\tilde{\mathcal{L}}v\ge g$ in $\omega$ for some $g\in L^q(\omega)$, then
$$\inf_{\omega} \frac{v}{d}\ge c_0 \int_{\omega}g d,\qquad d(x)=\mathrm{dist}(x,\partial\omega).$$
\end{lem}

\smallskip

The proof of Lemma~\ref{LemOptimizedHopf2} uses some ideas from \cite{BrC}
(where the case $\tilde \ld=\Delta$ was treated),
as well as additional arguments based on  the boundary Harnack inequality and a reduction argument which permits us to consider simpler operators and right-hand sides.

\begin{proof}[Proof of Lemma~\ref{LemOptimizedHopf2}] We will give the proof of statement (ii) (the proof of (i) is simpler and goes the same way).

We note that we can assume that $\lambda_1(-\tilde{\mathcal{L}},\omega)>0$ and that $v>0$ is a solution of $-\tilde{\mathcal{L}}v= g$ in $\omega$ with $v=0$ on $\partial \omega$, and so $v\in C^{1,\alpha}(\overline{\omega})$. Indeed, the positivity of the first eigenvalue follows from Proposition \ref{lowerbdeig} and an appropriate choice of $\ep_0$, and  then  it is sufficient to replace $v$ by the solution of the Dirichlet problem $-\tilde{\mathcal{L}}\hat v= g$ in $\omega$ with $\hat v=0$ on $\partial \omega$, and note that $v\ge \hat v$ by the maximum principle (the maximum principle and the solvability of the Dirichlet problem are available thanks to $\lambda_1(-\tilde{\mathcal{L}},\omega)>0$, see the appendix).

It is easy to see that  there is a universal $\rho>0$ and a point $\xi$ such that the ball $B_{4\rho}(\xi)\subset B_1^+$ (by the assumption on the norm of the boundary each point of $\partial \Omega$ can be touched by an interior cone with opening 1 and fixed size, see the beginning of Section \ref{sec-ext}).  Set $B_\rho=B_\rho(\xi)$.

By \eqref{hyplambda1b} and Proposition~\ref{upperboundlambda1}  we also  have
\be\label{upperlam2}
0<\lambda_1(-\tilde{\mathcal{L}^*},\omega)=\lambda_1(-\tilde{\mathcal{L}},\omega)\le \lambda_1(-\tilde{\mathcal{L}},B_{2\rho})\le C_0.\ee

{\bf Step 1.}
Set $\ld_0u= \mathrm{div}(A(x)Du)$, with $\|A\|_{C^\alpha(\omega)}\le \Lambda+1$. The solution of the auxiliary problem
$$
\left\{\hskip 2mm\begin{aligned}
-\ld_0w_0&=\chi_{B_\rho} &\quad&x\in \omega,\\
w_0&=0, &\quad&x\in \partial \omega,
\end{aligned}\right.
\qquad
\mbox{is such that}\qquad
\inf_{\omega} \frac{w_0}{d} \ge c_0.
$$

\noindent{\it Proof}. We  consider the first eigenvalue and eigenfunction  $\lambda_1,\varphi_1>0$
of $\ld_0^*$=div$(A^TD\cdot)$ in $B_{2\rho}$,
normalized by $\max_{B_{2\rho}}\varphi_1=1$. Clearly $0<\lambda_1\le C_0$.
By the boundary Harnack inequality we have
$\inf_{B_{2\rho}}\frac{\varphi_1}{d_0}\ge c_0 \sup_{B_{2\rho}}\frac{\varphi_1}{d_0}\ge c_0 \sup_{B_{2\rho}}\varphi_1$,
where $d_0(x)={\rm dist}(x,\partial B_{2\rho})$, hence in particular
$$
\inf_{B_{\rho}}\varphi_1\ge c_0.
$$

Since $-\ld_0w_0 \ge 0$, $w_0>0$ in $\omega$,   by  the boundary weak Harnack inequality \eqref{sharpWBHI}
\be\label{intharn1}\inf_{\omega} \frac{w_0}{d}\ge c_0 \Bigl\| \frac{w_0}{d}\Bigr\|_{L^\eps(\omega)}\ge c_0 \Bigl\|
\frac{w_0}{d}\Bigr\|_{L^\eps(B_{2\rho})}\ge c_0 \|{w_0}\|_{L^\eps(B_{2\rho})}\ge c_0 \inf_{ B_{2\rho}}w_0.\ee
whereas by the interior weak Harnack inequality
\cite[Theorem 8.18]{GT} (and remark at the end of \cite[Section 8.10]{GT}),
\be\label{intharn2}
\inf_{B_{2\rho}}w_0\ge c_0 \int_{B_{2\rho}} w_0,
\ee
and by using \eqref{upperlam2} and standard integration by parts
\be\label{HopfIPP}
\begin{aligned}
\int_{B_{2\rho}} w_0
&\ge \int_{B_{2\rho}} w_0\varphi_1= \frac{1}{\lambda_1} \int_{B_{2\rho}} w_0(-\ld_0\varphi_1)\\
&= \frac{1}{\lambda_1} \int_{B_{2\rho}} (-\ld_0w_0)\varphi_1 -\frac{1}{\lambda_1}\int_{\partial B_{2\rho}}  \nu\cdot A(x) w_0
\nabla\varphi_1\, d\sigma\\
&\ge c_0 \int_{B_{2\rho}} \chi_{B_\rho}\varphi_1 = c_0 \int_{B_{\rho}} \varphi_1 \ge  c_0 \inf_{B_{\rho}}\varphi_1\ge c_0,
\end{aligned}
\ee
since $\nabla\varphi_1=(\partial_\nu \varphi_1)\nu$ and $\partial_\nu\varphi_1\le 0$, $\nu\cdot A(x)\nu\ge 0$ on $\partial B_{2\rho}$.
Combining this with \eqref{intharn1}-\eqref{intharn2} gives the claim.

\smallskip

{\bf Step 2.}
Set $\ld_1u= \mathrm{div}(A(x)Du)+b(x)\cdot Du$, $b\in L^q(\omega)$.  There exists a constant $\ep_1\in(0,1)$ depending on $n,\lambda,\Lambda, q, \alpha,$ such that if $\|A\|_{C^\alpha(\omega)}\le \Lambda+1$ and $\|b\|_{L^q(\omega)}\le \ep_1$ then the solution of the auxiliary problem
$$
\left\{\hskip 2mm\begin{aligned}
-\ld_1w_1&=\chi_{B_\rho} &\quad&x\in \omega,\\
w_1&=0, &\quad&x\in \partial \omega,
\end{aligned}\right.
\qquad
\mbox{is such that}\qquad
\inf_{\omega} \frac{w_1}{d} \ge c_0.
$$

\noindent{\it Proof}. By standard $L^\infty$ and $C^1$ estimates (see \cite[Theorem 8.16]{GT}, \cite[Corollary 8.36]{GT},  the remark at the end of \cite[Section~8.10]{GT},  \cite[Theorem 5.5.5']{Mo}) if $\|b\|_{L^q(\omega)}\le 1$ any solution of $-\ld_iz_i = g$ in $\omega$, $z_i=0$ on $\partial\omega$ ($i=0,1$, $\ld_0$ is the operator from Step 1), $g\in L^q(\omega)$, is such that
$$
\|z_i\|_{C^1(\omega)}\le C_0\|g\|_{L^q(\omega)}.
$$

Now notice that $w_1 = w_0-z_0$, where $w_0$ is the function from Step 1, and $z_0$ solves
$$-\ld_0z_0 = - b\cdot Dw_1=:g\quad\mbox{ in }\quad \omega.$$
Since $\|g\|_{L^q(\omega)}\le C_0\ep_1$, by choosing $\ep_1$ small enough we have
$\|z_0\|_{C^1(\omega)}\le  c_0/2$, where $ c_0$ is the constant from Step 1. Thus
$$
\inf_{\omega} \frac{w_1}{d} \ge \inf_{\omega} \frac{w_0}{d} - \sup_{\omega} \frac{z_0}{d}\ge  c_0/2.
$$

\smallskip

{\bf Step 3.}
There exists a constant $\ep_0>0$ depending on $n,\lambda,\Lambda, q, \alpha,$ such that if \eqref{hyplambda1b} holds then the solution of the auxiliary problem
$$
\left\{\hskip 2mm\begin{aligned}
-\tilde \ld^*w&=\chi_{B_\rho} &\quad&x\in \omega,\\
w&=0, &\quad&x\in \partial \omega,
\end{aligned}\right.
\qquad
\mbox{is such that}\qquad
\inf_{\omega} \frac{w}{d} \ge c_0.
$$

\noindent{\it Proof}. We recall that $\ld^*w = \mathrm{div}(\tilde A^TDw - \tilde b_2w) - \tilde b_1Dw + \tilde cw$. Notice that the solutions of the Dirichlet problem for $\tilde \ld^* $ are not necessarily $C^1$, so we will need a workaround for
this lack of smoothness.

Let $\psi\in H^1_0(\omega)$ be the solution of the Dirichlet problem
$$
\tilde \ld^*\psi = -div(\tilde b_2) + \tilde c
$$
(this problem has a unique solution since $\lambda_1(-\tilde \ld^*,\omega)>0$, cf.~\eqref{upperlam2}).
By \cite[Theorem 8.16]{GT} we have \be\label{estlinfpsi}
\|\psi\|_{L^\infty(\omega)}\le C\eps_0.\ee
We can also write
$$
\mathrm{div}(\tilde A^TD\psi) - \tilde b_1D\psi  = -div(\tilde b_2(1-\psi)) + \tilde c(1-\psi)
$$
and apply  Morrey's $W^{1,q}$-regularity estimate (see \cite[Section~5.5]{Mo}) to this equation.
Together with  \eqref{estlinfpsi} and Sobolev embeddings  this gives
$$
\|\psi\|_{C^\alpha(\omega)}\le C\|\psi\|_{W^{1,q}(\omega)}\le
C(\|\psi\|_{L^\infty(\omega)} + \|\tilde b_2(1-\psi)\|_{L^q(\omega)} + \|\tilde c(1-\psi)\|_{L^q(\omega)})\le C_1\ep_0,
$$
provided  $\ep_0\le 1$.
We further diminish $\ep_0$ so that $C_1\ep_0<1/2$. Then the function $\phi=1-\psi$ is such that
\be\label{propfi}
\tilde \ld^*\phi= 0 \;\mbox{ in } \omega,\qquad 1/2\le\phi\le3/2 \;\mbox{ in } \omega\qquad \mbox{and} \qquad [\phi]_{\alpha,\omega}\le C\|D\phi\|_{L^q(\omega)}\le C\ep_0.
\ee

Set $w_1= w/\phi$. It is easy to compute that $\tilde \ld^*w = \tilde \ld^*(\phi w_1) = \hat L(w_1)$, where $\tilde L$ is an operator of the form \eqref{defdiv} whose coefficients are given by
$$
\hat A = \phi \tilde A^T, \quad \hat b_1 = \tilde A^TD\phi - \tilde b_2 \phi, \quad
\hat b_2 = \tilde A^TD\phi -\tilde b_1\phi,
\quad \hat c =- \tilde b_1D\phi + \tilde c\phi.
$$
Since div$(\hat b_1) + \hat c = \tilde \ld^*\phi= 0$ in the sense of distributions in $\omega$, we see that $-\tilde \ld^*w=-\hat L(w_1) = \chi_{B_\rho}$  is equivalent in the weak sense to
$-\ld_1w_1 = \chi_{B_\rho}$ where
$$
\ld_1 w_1 = \mathrm{div}(\phi \tilde A^T Dw_1) +(2 A^TD\phi -(\tilde b_1+\tilde b_2)\phi)Dw_1$$
is an operator to which Step 2 applies, thanks to \eqref{propfi} and an appropriate choice of $\ep_0$. Again by \eqref{propfi} $w_1\ge cd$ implies $w\ge (c/2)d$ in $\omega$.

\smallskip
{\bf Step 4.} We claim that
\be \label{HopfAuxClaim2}
\inf_{B_\rho}v\ge c_0  \int_{\omega} g d.
\ee
Indeed, by the interior weak Harnack inequality and assumption \eqref{hyplambda1b} with $\ep_0\le1$, we have $\inf_{B_\rho}v\ge c_0\int_{B_\rho} v$. Therefore,
$$\inf_{B_\rho}v\ge c_0\int_{\omega} v\chi_{B_\rho}=c_0\int_{\omega} v(-\tilde{\mathcal{L}}^*w)
\ge c_0\int_{\omega} w(-\tilde{\mathcal{L}}v)\ge c_0\int_{\omega} gw,$$
where the second inequality follows similarly as in \eqref{HopfIPP}, by using also Lemma \ref{weakdiv2}.
Step 3 then gives \eqref{HopfAuxClaim2}.

\smallskip

{\bf Step 5.} We introduce the auxiliary problem:
$$
\left\{\hskip 2mm\begin{aligned}
-\tilde{\mathcal{L}}z&=0, &\quad&x\in \omega\setminus B_\rho,\\
z&=0, &\quad&x\in \partial \omega,\\
z&=1, &\quad&x\in \partial B_\rho.
\end{aligned}
\right.
$$
Note that this problem is solvable owing to $\lambda_1(-\tilde{\mathcal{L}},\omega)>0$, and $z>0$ by the maximum principle.
We claim that
\be \label{HopfAux6}
\inf_{\omega\setminus B_\rho} \frac{z}{d} \ge c_0.
\ee

Indeed, by the H\"older estimate (cf.~\cite[Theorem 8.29]{GT} and the remark at the end of \cite[Section 8.10]{GT}) and \eqref{hyplambda1b},
there exists $\delta\in(0,\rho)$ depending only on $n,q,\lambda,\Lambda$ such that
\be \label{HopfAux4}
z\ge 1/2\ge c_0 d,\quad x\in B_{\rho+\delta}\setminus B_\rho
\ee
so  by the boundary  Harnack inequality,  with $\hat d(x)={\rm dist}(x,\partial (\omega\setminus B_{\rho}))$,
\be \label{HopfAux5}
\inf_{\omega\setminus B_{\rho}} \frac{z}{\hat d} \ge c_0 \sup_{\omega\setminus B_{\rho}} \frac{z}{\hat d}\ge c_0 \sup_{ B_{2\rho}\setminus B_{\rho}} {z}\ge c_0/2.
\ee
Combining \eqref{HopfAux4} and \eqref{HopfAux5} proves the claim.

\smallskip
{\bf Step 6.} Conclusion.
Set $Z=v-c_0(\int_{B_1^+} g d)z$.
We have $-\tilde{\mathcal{L}}Z\ge 0$ in $\omega\setminus B_\rho$, $Z=0$ on $\partial \omega$ and $Z\ge 0$ on $\partial B_\rho$
in view of \eqref{HopfAuxClaim2}.
By the maximum principle (which applies owing again to $\lambda_1(-\tilde{\mathcal{L}},\omega)>0$), we deduce that
$Z\ge 0$ in $\omega\setminus B_\rho$. Therefore, by \eqref{HopfAux6},
$$\inf_{\omega\setminus B_\rho} \frac{v}{d} \ge c_0 \int_{\omega} g d.$$
This combined with \eqref{HopfAuxClaim2} completes the proof of the lemma.
\end{proof}

The following lemma will permit us to cover the domain $\Omega$
with small balls such that a suitably rescaled version of the operator in each of these balls
satisfies the assumptions of Lemma \ref{LemOptimizedHopf2}.
\begin{lem} \label{LemOptimizedHopf3}
Let $A, b_1, b_2, c$ satisfy the assumptions of Theorem~\ref{thm4gen}(i).
Let $x_0\in \overline\Omega$
and $r_0$ be the number from \eqref{defr0}. There exists $\delta\in(0,1/2)$ depending only on $n, \lambda, \Lambda, q, \alpha$, such that
the (rescaled) operator
\be\label{def-A-resc1}
\tilde{\mathcal{L}} v=\tilde{\mathcal{L}}_\delta v=\mathrm{div}(\tilde A(y)\nabla v+\tilde b_1(y) v)+\tilde b_2(y)\cdot \nabla v+\tilde c(y)v,
\quad y\in  \omega,
\ee
where $$r=\delta r_0,\qquad\omega=\omega_\delta= r^{-1}(\Omega-x_0)\cap B_2,$$
\be\label{def-A-resc2}
\tilde A(y)=A(x_0+ry), \quad \tilde b_i(y)=rb_i(x_0+ry),\quad \tilde c(y)=r^2c(x_0+ry).
\ee
is such that \eqref{hyplambda1b} holds.
 If, moreover, $x_i\in\partial\Omega$ then
there exists a $C^{1,\bar\alpha}$ domain $\tilde\omega$
 with norm bounded by~$C=C(n,\bar\alpha)>0$, such that $r^{-1}(\Omega-x_0)\cap B_{3/2}\subset\tilde\omega\subset r^{-1}(\Omega-x_0)\cap B_2$.
\end{lem}

\begin{proof} The first statement follows by similar computations as in \eqref{rescaledLq},
whereas the second statement is clear from the definition of $r_\Omega$ at the beginning of Section~\ref{sec-ext}
and the fact that $r_0\le r_\Omega$.
\end{proof}

We are now in a position to prove Theorem~\ref{thm4gen}(i), of which Theorem \ref{OptimizedHopf} is a special case.

\begin{proof}[Proof of Theorem~\ref{thm4gen}(i)]
Since $f\ge 0$, the optimized weak
boundary Harnack inequality \eqref{sharpWBHI} gives
\be\label{BHIud}
\inf_\Omega \frac{u}{d}\ge
D^{-n/\eps}e^{-{C_0(r_\Omega^{-1}+M)D}} \Bigl\| \frac{u}{d}\Bigr\|_{L^\epsilon(\Omega)}.
\ee
Let $r$ be the number given  by the previous lemma.
By Proposition~\ref{geodes}(ii), we may cover $\Omega$ by $N$ balls $B_{r}(x_i)$ in such a way that:
$$\begin{aligned}
&\hbox{$N\le C(n)(D_0/r)^n$}\\
&\hbox{$|\Omega\cap B_r(x_i)|\ge c(n)r^n$}\\
&\hbox{either $x_i\in\partial\Omega$ or $d(x_i)\ge 3r/2$}.
\end{aligned}$$
Denote $\tilde\Omega=r^{-1}(\Omega-x_i)$, $B_R^+:=B_R\cap \tilde\Omega$, and set
$$
\omega_i=
\begin{cases}
B_2,& \hbox{if $x_i\in \Omega$}\\
\tilde \omega,& \hbox{if $x_i\in\partial\Omega$,}
\end{cases}
$$
where $\tilde \omega$ is the $C^{1,\bar\alpha}$ domain given by Lemma~\ref{LemOptimizedHopf3},
with norm less than~$C(n,\bar\alpha)$, which satisfies
$B_{3/2}^+\subset\tilde \omega\subset B_{2}^+$.
For each $i$, the function $v_i(y):=u(x_i+ry)$ solves
$$\tilde{\mathcal{L}}_iv_i
=g_i(y),
\quad y\in \omega_i$$
where $\tilde{\mathcal{L}}_i$ is defined by \eqref{def-A-resc1}-\eqref{def-A-resc2} with $x_0=x_i$, and $g_i(y)=r^2f(x_i+ry)$.
If $x_i\not\in\partial\Omega$, it follows from Lemma \ref{LemOptimizedHopf2}(i) that
$$\inf_{B_1}v_i\ge C_0\int_{B_1}g_i,
\quad\hbox{equivalently:}\ \
\inf_{B_r(x_i)}u\ge C_0r^{2-n}\int_{B_r(x_i)} f$$
hence,
\be\label{infvi1}
\inf_{B_r(x_i)} \frac{u}{d}\ge C_0D_0^{-2}r^{2-n}\int_{B_r(x_i)} fd.
\ee
If $x_i\in\partial\Omega$, it follows from Lemma \ref{LemOptimizedHopf2}(ii) that
$$\inf_{B_1^+}\frac{v_i}{d_i}\ge C_0 \int_{B_1^+} g_id_i,
\quad\hbox{where } d_i(y)={\rm dist}(y,\omega_i).$$
We claim that
\be\label{compdistomegai}
{\rm dist}(y,\partial B_{3/2}^+)\ge \min\bigl(1/2,{\rm dist}(y,\partial \tilde\Omega)\bigr),\quad y\in B_1^+.
\ee
Indeed, for given $y\in B_1^+$, denote $p$ the projection of $y$ onto $\partial B_{3/2}^+$.
If $p\in \partial B_{3/2}$, then $|p-y|\ge 1/2$. Otherwise, $p\in \partial\tilde\Omega$
and $B_{|p-y|}(y)\subset B_{3/2}^+\subset\tilde\Omega$, hence $|p-y|={\rm dist}(y,\partial \tilde\Omega)$.
This proves \eqref{compdistomegai}.
Since ${\rm dist}(y,\partial \tilde\Omega)=r^{-1}{\rm dist}(x_i+ry,\partial\Omega)$ and $r\le D_0$, it follows that
$$
\begin{aligned}
d_i(y)
&\ge {\rm dist}(y,\partial B_{3/2}^+)\ge
\min\bigl(1/2,{\rm dist}(y,\partial \tilde\Omega)\bigr)
\ge (2D_0)^{-1}d(x_i+ry),\quad y\in B_1^+.
\end{aligned}$$
Scaling back to $u$, we get
\be\label{infvi2}
\begin{aligned}
\inf_{\Omega\cap B_r(x_i)}\frac{u}{d}
&\ge (2D_0)^{-1}
\inf_{y\in B_1^+} \frac{u(x_i+ry)}{d_i(y)}\ge
C_0D_0^{-1}
\int_{B_1^+} g_i(z) d_i(z)\,dz \\
&\ge C_0 D_0^{-2}r^2
\int_{B_1^+} f(x_i+rz)d(x_i+rz)\,dz=C_0D_0^{-2}r^{2-n}
\int_{\Omega\cap B_r(x_i)} fd.
\end{aligned}
\ee

Now, since
$$
\sum_{i=1}^N \int_{\Omega\cap B_r(x_i)}fd\ge \int_{\Omega}fd,
$$
we may choose $i=i_0$ such that
\be\label{choice}\int_{\Omega\cap B_r(x_i)}fd\ge N^{-1}\int_{\Omega}fd.\ee
Inequalities \eqref{BHIud}, \eqref{infvi1} and \eqref{infvi2} then yield:
$$\begin{aligned}
\inf_\Omega \frac{u}{d}&\ge
D^{-n/\eps}e^{-{C_0(r_\Omega^{-1}+M)D}} \Bigl\| \frac{u}{d}\Bigr\|_{L^\epsilon(\Omega)}
\ge
D^{-n/\eps}e^{-{C_0(r_\Omega^{-1}+M)D}}
\Bigl\| \frac{u}{d}\Bigr\|_{L^\epsilon(\Omega\cap B_r(x_i))} \\
&\ge
D^{-n/\eps}e^{-{C_0(r_\Omega^{-1}+M)D}}
|\Omega\cap B_r(x_i)|^{1/\eps}
\inf_{\Omega\cap B_r(x_i)} \frac{u}{d} \\
&\ge e^{-{C_0(r_\Omega^{-1}+M)D} (r/D)^{(n/\eps)+2-n}D^{-n}}
\int_{\Omega\cap B_r(x_i)} fd.
\end{aligned}$$
By using \eqref{choice}, $N^{-1}\ge C(n) (r/D)^n$
and $r/D\ge \delta r_0/D\ge C_0\bigl((r_\Omega^{-1}+M)D\bigr)^{-1}$, owing to
\eqref{relMr0},
we deduce that
$$\inf_\Omega  \frac{u}{d}\ge
e^{-{C_0(r_\Omega^{-1}}+M)D}(r/D)^{(n/\eps)+2}D^{-n}
\int_{\Omega}fd
\ge e^{-{C_0(r_\Omega^{-1}+M)D}} D^{-n} \int_{\Omega}fd,$$
which is the desired result.
\end{proof}

Theorem~\ref{OptimizedHopf}
is a consequence of Theorem \ref{thm4gen} (i) and \eqref{scaleinvtilde1b}.

\section{Proof of the log-grad estimate (Theorems \ref{lgeu} and \ref{thm4gen}{\rm(ii)-(iii)})}
\label{proofloggrad}

We prove Theorem \ref{thm4gen} (ii)-(iii), of which Theorem \ref{lgeu} is a special case.

\begin{proof}[Proof of Theorem~\ref{thm4gen} (ii)-(iii)]
{Fix $x_0\in\Omega$
 with $d_0=d(x_0)= \mathrm{dist}(x_0,\partial\Omega)$ and set
$$
r= \min\{r_0, d_0\},
$$
where $r_0$ is the number from \eqref{defr0}.
We rescale our equation
$-\ldu=f$
setting $y=(x-x_0)/r$, $u(x)=\tilde u(y)$,
and the coefficients of $\tilde{\ld}$ being defined by \eqref{def-A-resc2}. We obtain
\begin{equation}\label{rescl1}
 -\tilde{\ld}[\tilde u] = r^2f(x_0+ry)=:\tilde f(y)
\end{equation}
in the unit ball, since $r\le d_0$. Since $r\le r_0$, by the choice of $r_0$ we know (see for instance the computations in \eqref{rescaledLq}) that
$$
\|\tilde b_1\|_{C^\alpha(B_1)}\le C_0, \qquad \|\tilde b_2,c\|_{L^q(B_1)}\le C_0.
$$
Applying successively the standard $C^1$-to-$C^0$ estimate and the interior Harnack inequality for \eqref{rescl1} we obtain
$$
\begin{aligned}
|\nabla \tilde u(0)|& \le C_0 \sup_{B_{3/4}}\tilde u + C_0\|\tilde f \|_{L^q(B_1)}\\
&\le C_0\tilde u (0) +  C_0r^{2-n/q}\|f\|_{L^q(B_r(x_0))}
\end{aligned}
$$
which implies that
$$
r|\nabla u (x_0)|\le C_0u(x_0) + C_0r^{2-n/q}\|f\|_{q,r_0,\Omega}.
$$
Hence, multiplying by $d_0/ru(x_0)$ and recalling the definition of $r$ and \eqref{relMr0}, we get
$$\begin{aligned}
\frac{ d_0|\nabla u(x_0)|}{u(x_0)}
&\le C_0 \max\Bigl\{ 1,\frac{d_0}{r_0}\Bigr\} +
C_0 \bigl(\min\{r_0,d_0\}\bigr)^{1-n/q}
 \frac{\|f\|_{q,r_0,\Omega}}{u(x_0)/d(x_0)}\\
&\le C_0 \max \bigl\{1, (r_\Omega^{-1}+M)d_0 \bigr\} +
C_0 \bigl(\min\{(r_\Omega^{-1}+M)^{-1},d_0\}\bigr)^{1-\frac{n}{q}}
 \frac{\|f\|_{q,r_0,\Omega}}{u(x_0)/d_0}.
\end{aligned}$$
Since $x_0\in \Omega$ is arbitrary, this in particular proves assertion~(ii) (case $f\equiv 0$). If $f\not\equiv 0$, we can next
apply the optimized Hopf estimate \eqref{conclHopfgen} to the last denominator in this expression, and deduce
$$
{\frac{ d|\nabla u|}{u}}\le C_0 \max\bigl\{1,(r_\Omega^{-1}+M)d\bigr\} +
C_0 \bigl(\min\{(r_\Omega^{-1}+M)^{-1},d\}\bigr)^{1-\frac{n}{q}} e^{C_0(r^{-1}_\Omega+ M)D}D^n\frac{\|f\|_{q,r_0,\Omega}}{\|f\|_{L^1_{d}(\Omega)}},
$$
hence assertion~(iii).
(We note that, in view of the exponential factor, we can discard the min in the last inequality without loss of information;
likewise, we cause no loss by replacing $\|f\|_{q,r_0,\Omega}$ with $\|f\|_{L^q(\Omega)}$ -- cf.~Remark~\ref{remnoloss}.)}
\end{proof}

\smallskip
We now turn to the second, completely different, proof, based on a contradiction and doubling-rescaling argument
instead of the Harnack inequality.
For simplicity we will give this proof only in the case $\Omega=B_1$, but it can be easily
modified to the case of $C^{1,\bar\alpha}$ domains,
as the reader could verify. This proof is somewhat longer but it allows to show the following logarithmic gradient bound independent of $f$ (cf.~Remark~\ref{remloggrad1d}).

\begin{prop}\label{propn1} Assume $n=1$ and \eqref{hyp1}-\eqref{hyp2}. For any nonnegative $f\in L^q(-R,R)$, $q>1$,
and any weak solution $u>0$ of $-\mathcal{L}u=f$ in $\Omega=(-R,R)$, there holds
$$\frac{d(x)|u'(x)|}{u(x)}\le C_0\max\{1, M d(x)\},\quad x\in (-R,R).$$
Here a weak solution is a function
\be\label{defsol1m}
u\in W^{1,p}(-R,R)\quad\hbox{with $p=\max(2,q')$,}
\ee
which satisfies the equation in the usual $H^1$ sense.\end{prop}

\begin{rem}\label{remdefsol1}
Note that under
assumption \eqref{defsol1m}
the term $bu'$ (as well as $cu$) belongs to $L^1(\Omega)$, and that this is not the case if $u$ is merely in $H^1(\Omega)$ when $q<2$.
Actually if we take $u\in W^{1,\infty}(-R,R)$, the above statement remains true if $b,c,f\in L^1(\Omega)$ instead of~$L^q(\Omega)$.
\end{rem}

\begin{proof}[Second proof of Theorem \ref{lgeu} and proof of Proposition \ref{propn1}.]
Let $\Omega=B_R$. By rescaling and using Proposition \ref{scaleinvtilde} (ii) we can suppose $R=1$ (see the proof of Proposition \eqref{Prop-Optim-Hopf1} in the next section for details).

We first transform the equation by the change of variable $v:=\log u$, $u=e^v$.
We compute
$$
\begin{aligned}
\mathcal{L}u
&=\nabla\cdot(A(x)e^v\nabla v+b_1(x)e^v)+e^vb_2(x)\cdot\nabla v+c(x)e^v \\
&=e^v\Bigl\{\nabla\cdot(A(x)\nabla v)+\nabla\cdot b_1(x)+\nabla v\cdot A(x)\nabla v
+b_1(x)\cdot\nabla v+b_2(x)\cdot\nabla v+c(x)\Bigr\} \\
\end{aligned}
$$
hence, setting $b=b_1+b_2$,
\be\label{eqforvk}
\mathcal{L}_1v
:=\nabla\cdot(A(x)\nabla v)+\nabla v\cdot A(x)\nabla v+b(x)\cdot\nabla v+\nabla\cdot b_1(x)+c(x)
=-e^{-v}f(x).
\ee
In terms of $v$, the sought-for estimate is equivalent to
$$
d|\nabla v| \le
\begin{cases}
C_0\max(1,Md),& \hbox{if $f\equiv 0$ or $n=1$}\\
\noalign{\vskip 1mm}
C_0\left(\max(1,Md)+d^{1-n/q}e^{C_1M}\frac{\|f\|_{q,r_0,B_1}}{\|f\|_{L^1_d(B_1)}}\right),
& \hbox{otherwise}
\end{cases}
$$
where $M=M(\ld,B_1)$, $r_0=r_0(\ld,B_1)$ (cf.~\eqref{defM}) and, in the second case, $C_1$ is the constant from the Hopf estimate
in Theorem~\ref{OptimizedHopf}.

Assume for contradiction that there exist  sequences of  operators $\ld_k$ in the form \eqref{eqforvk}, with coefficients
$A_k, b_k, b_{1,k}, c_k$
satisfing \eqref{hyp1}-\eqref{hyp2}, $f_k\in L^q(B_1)$, solutions $v_k$ of \eqref{eqforvk} and points $y_k\in B_1$ for which
$$L_k(y_k)\ge 2k (d(y_k))^{-1},$$
with $M_k=M(\ld_k,B_1)$, $r_k=r_0(\ld_k,B_1)$ and
\be\label{defMk}
L_k(x):=
\begin{cases}
\ds\frac{|Dv_k(x)|}{\max(1,M_kd(x))},& \hbox{if $f_k\equiv 0$ or $n=1$}\\
\noalign{\vskip 1mm}
\ds\frac{|Dv_k(x)|}{{\max(1,M_kd(x))+e^{C_1M_k}d^{1-n/q}(x)\frac{\|f_k\|_{q,r_k,B_1}}{\|f_k\|_{L^1_d(B_1)}}}
},& \hbox{otherwise}.
\end{cases}
\ee
By the doubling lemma (cf.~\cite[Lemma 5.1]{PQS}), there are points $x_k\in B_1$ such that
\be\label{largenessMk}
L_k(x_k)\ge L_k(y_k), \qquad |Dv_k(x_k)|\ge L_k(x_k)\ge 2k (d(x_k))^{-1} \ge 2k,
\ee
and
\be\label{doublingMk}
L_k(z)\le 2 L_k(x_k)\qquad \mbox{if }\: |z-x_k|< k (L_k(x_k))^{-1}.
\ee
Set
\be\label{defrk}
\rho_k= |Dv_k(x_k)|^{-1} \qquad (\rho_k\to 0 \mbox{ as } k\to \infty).
\ee
Observe for later purposes that
\be\label{comprk}
\rho_k
\le (2k)^{-1}\min\{M_k^{-1},d(x_k)\}
\ee
owing to \eqref{defMk}, \eqref{largenessMk}, \eqref{defrk}.
We now rescale
\be\label{rescalingvk}
x=\rho_ky + x_k, \quad w_k(y) =  v_k(\rho_ky+x_k)-v_k(x_k),
\ee
i.e., $v_k(x) = w_k(\rho_k^{-1}(x-x_k))+v_k(x_k)$.
Note that, while in \eqref{eqforvk} we have an equation in $v$, we do the rescaling in such a way that only the gradients $\nabla w_k$ of the rescaled $v_k$ (rather than the $w_k$ themselves) stay bounded.
We obtain from \eqref{eqforvk} (omitting the variable $y$):
\be\label{eqforwk}
\nabla\cdot(\tilde A_k\nabla w_k)+\nabla w_k\cdot \tilde A_k\nabla w_k+\tilde b_k\cdot\nabla w_k+
\nabla\cdot \tilde b_{1,k}+\tilde c_k=-g_k
\ee
where
$$
\tilde{A}_k(y) := A_k(\rho_ky + x_k),\quad
\tilde{b}_k(y) := \rho_k b_k(\rho_ky + x_k),\quad \tilde{b}_{1,k}(y) := \rho_k b_{1,k}(\rho_ky + x_k),$$
$$\tilde{c}_k(y) := \rho_k^2 c_k(\rho_ky + x_k),\quad \tilde{f}_k(y) :=  \rho_k^2 f(\rho_ky + x_k),\
\quad g_k(y):=e^{-w_k(y)}e^{-v_k(x_k)}\tilde f_k(y).$$
For $|y|\le k$, by \eqref{comprk}, we have
$d(x_k+\rho_ky)\le d(x_k)+\rho_k|y|\le  \frac32 d(x_k)$. This combined wih
\eqref{defMk}, \eqref{defrk} and \eqref{rescalingvk} implies
$$
|Dw_k(y)| = \rho_k |Dv_k(x_k+\rho_ky)|= \frac{|Dv_k(x_k+\rho_ky)|}{|Dv_k(x_k)|}
\,\le \frac32\frac{L_k(x_k+\rho_ky)}{L_k(x_k)},
$$
hence, owing to \eqref{doublingMk},
\be\label{boundDwk}
|Dw_k(y)| \le 3\quad\hbox{for $|y|<k$},
\ee
and
\be\label{boundDwk0}
|Dw_k(0)| = \rho_k |Dv_k(x_k)| = 1,\qquad w_k(0)=0.
\ee
As a consequence of the second part of \eqref{boundDwk0} and of \eqref{boundDwk}, we have
\be\label{boundwk}
|w_k(y)|\le 3|y|\quad\hbox{for $|y|<k$}.
\ee

 {Fix any $R_0>1$. Note that \eqref{comprk} and \eqref{relMr0} guarantee that for $k$ large (depending on $R_0$),
we have
{$\rho_k^{-1}\ge 2k\max\{M_k,d(x_k)^{-1}\}>
R_0(r_{B_1}^{-1}+M_k)=R_0r_k^{-1}$, hence} $R_0\rho_k<r_k$. We then obtain, as $k\to\infty$:}
\be\label{cvbk}
\|\tilde{b}_k\|_{L^q(B_{R_0})}\le
\rho_k^{1-n/q} \|{b}_k\|_{L^q({B_1 \cap B_{\rho_kR_0}(x_k))}}\le
\rho_k^{1-n/q} \|{b}_k\|_{q,r_k,B_1}\le (\rho_kM_k)^{1-n/q} \to 0,
\ee
\be\label{cvb11k}
\|\tilde{b}_{1,k}\|_{L^\infty(B_{R_0})}=\rho_k\|b_{1,k}\|_{L^\infty(B_{R_0})} \le
(2k)^{-1} M_k^{-1} \|b_{1,k}\|_{L^\infty(B_{R_0})} \le (2k)^{-1}
\to 0,
\ee
and similarly
\be\label{cvck}
\|\tilde{c}_k\|_{L^q(B_{R_0})}\to 0,\quad [\tilde{A}_k]_{\alpha, B_{R_0}}\to 0,\quad [\tilde{b}_{1,k}]_{\alpha, B_{R_0}}\to 0.
\ee
{Therefore, passing to a subsequence, we may assume that $x_k\to x_\infty\in \overline B_1$
and that $A_k(x_k)\to A_\infty$, where $A_\infty$ is a
positive definite matrix.
Moreover, by the compact embedding $C^\alpha\hookrightarrow C^0$
\be\label{cvAk2}
\hbox{$\tilde A_k(y)\to A_\infty$ uniformly for $y$ in any compact set.}
\ee

\medskip

{\bf Case 1: $n\ge 2$.}
If $f_k\not\equiv 0$, we also need to estimate the right hand side $g_k$.
As above, we first write
\be\label{cvfk}
\|\tilde{f}_k\|_{L^q(B_{R_0})}\le
\rho_k^{2-n/q} \|{f}_k\|_{L^q({B_1\cap B_{\rho_kR_0}(x_k))}}
\le  \rho_k^{2-n/q} \|{f}_k\|_{q,r_k,B_1}.
\ee
We then use the quantitative Hopf estimate in Theorem~\ref{OptimizedHopf} and the second part of \eqref{largenessMk} to write
$$u_k(x_k)\ge e^{-C_1(1+M_k)}\|f_k\|_{L^1_d(B_1)}\,d(x_k)\ge e^{-C_1(1+M_k)}\|f_k\|_{L^1_d(B_1)} \ 2k L_k^{-1}(x_k)$$
hence,
$$\begin{aligned}
e^{-v_k(x_k)}
&=u_k^{-1}(x_k)\le C_0e^{C_1M_k}\|f_k\|_{L^1_d(B_1)}^{-1} (2k)^{-1} L_k(x_k). \\
\end{aligned}$$
Therefore, using
 \eqref{boundwk}, \eqref{cvfk}, and the definition \eqref{defMk} of $L_k(x_k)$ we get, for any $R_0>1$ and $k>R_0$,
$$
\begin{aligned}
\|g_k\|_{L^q(B_{R_0})}
&= e^{-v_k(x_k)} \|e^{-w_k}\tilde f_k\|_{L^q(B_{R_0})}
\le  C_0e^{C_1M_k}\|f_k\|_{L^1_d(B_1)}^{-1} L_k(x_k) e^{3R_0}\rho_k^{2-n/q} \|{f}_k\|_{q,r_k,B_1}\\
&\le {(2k)^{-2+n/q} e^{3R_0}L_k(x_k)e^{C_1M_k} \frac{ \|{f}_k\|_{q,r_k,B_1}}{\|f_k\|_{L^1_d(B_1)}}  d^{1-n/q}(x_k) |Dv_k(x_k)|^{-1}} \\
&\le (2k)^{-2+n/q} e^{3R_0} \to 0.
\end{aligned}
$$
By the last inequality, together with  \eqref{boundDwk} and \eqref{cvbk}-\eqref{cvAk2}, the equation \eqref{eqforwk} is in the form
$$
\mathrm{div} (\tilde A_k\nabla w_k) = -\mathrm{div}(\tilde{b}_{1,k}) + h_k
$$
where $\tilde{b}_{1,k}$ is bounded in $C^\alpha(B_{R_0})$ while $h_k$ is a function bounded in $L^q(B_{R_0})$, as $k\to\infty$. It follows from the $C^{1,\nu}$ interior
estimate in \cite[Theorem 8.32]{GT}, \cite[Chapter 5.5]{Mo},
that, for some $\nu\in(0,1)$,
\be\label{boundDeltawk}
\hbox{$w_k$ is bounded in $C^{1,\nu}_{loc}(\R^n)$}
\ee
and then, up to a subsequence,
$w_k$ converges in $C^1_{loc}(\R^n)$ to a weak solution $w\in C^{1,\nu}_{loc}(\R^n)$ of the limiting equation of \eqref{eqforwk}, namely
$$\nabla\cdot(A_\infty\nabla w)+\nabla w\cdot A_\infty\nabla w=0,\quad
x\in\R^n$$
with $|Dw(0)|=1$.
Thus the function $U=e^w$ is  a positive solution of
$\nabla\cdot(A_\infty\nabla U)=0$ in $\R^n$ (and $U$ is in fact a classical solution by elliptic regularity). Since $A_\infty$ is a constant matrix we have $\nabla\cdot(A_\infty\nabla U) = \mathrm{tr}(A_\infty D^2U)$, so after a linear change of the independent variable we obtain a positive harmonic function in $\rn$,
which must be constant, by the mean value property.
Since $|DU(0)|\ne 0$, this is a contradiction.}

\medskip

{\bf Case 2: $n=1$.}
Pick any nonnegative, compactly supported $\varphi\in H^1(\R)$.
Testing \eqref{eqforwk} for $n=1$ with $\varphi$, we get
\be\label{eqforwk2}
\int\tilde A_kw'_k\varphi'-\int\tilde A_k(w_k')^2\varphi-\int\tilde b_k w'_k \varphi+\int \tilde b_{1,k}\varphi'-
\int\tilde c_k\varphi=\int g_k\varphi \ge 0.
\ee
In particular, taking $R_0>1$ and assuming
$supp(\varphi)\subset (-2R_0,2R_0)$, $0\le\varphi\le 1$,
$|\varphi'|\le1$ and $\varphi=1$ on $[-R_0,R_0]$, we obtain by \eqref{boundDwk}, \eqref{cvbk}-\eqref{cvck},
$$\int_{|y|<R_0} g_k\le 9 \Lambda R_0+3\|\tilde b_k\|_{L^1(B_{2R_0})}+4R_0\|b_{1,k}\|_\infty+\|\tilde c_k\|_{L^1(B_{2R_0})}
\le C(R_0).$$
Therefore, $g_k$ is bounded in $L^1_{loc}(\R)$.
{In view of \eqref{eqforwk}-\eqref{boundDwk} and \eqref{cvbk}-\eqref{cvck}, it follows that
$(\tilde A_kw_k'+\tilde b_{1,k})'$  is bounded in $L^1_{loc}(\R)$ hence, using again \eqref{boundDwk}-\eqref{cvb11k},
$\tilde A_kw_k'+\tilde b_{1,k}$  is bounded in $W^{1,1}_{loc}(\R)$.
By this along with \eqref{boundDwk} and \eqref{boundwk} (which say $w_k$ is bounded in $W^{1,\infty}_{loc}(\R)$), up to extracting a subsequence, there exist functions
$w\in W^{1,\infty}_{loc}(\R)$ and $z\in L^\infty_{loc}(\R)$ such that $\tilde w_k\to w$ in $L^\infty_{loc}(\R)$
and $\tilde A_kw_k'+\tilde b_{1,k}\to z$ in
$L^m_{loc}(\R)$ for all finite $m$ and a.e. Up to extracting a further subsequence, we may also assume
that $\tilde A_k(y)\to a_\infty$ in $L^\infty_{loc}(\R)$ for some number $a_\infty>0$. Therefore, recalling \eqref{cvb11k},
$w_k'$ converges in $L^m_{loc}(\R)$ and we deduce that
\be\label{convwkw}
w_k\to w \quad\hbox{ in $W^{1,m}_{loc}(\R)$ for all finite $m$.}
\ee

We next claim that $w$ is nonconstant. Namely we will show that
\be\label{nonvanishw}
\int_{-1}^1 |w'(x)|dx>0.
\ee
To this end, to overcome the lack of $C^1$ convergence,
we need to ``thicken'' the normalization condition \eqref{boundDwk0}.
By extracting a further subsequence, we may assume that $w_k'(0)=-1$ for all $k$
or that $w_k'(0)=1$  for all $k$.
In the first case, by \eqref{eqforwk}-\eqref{boundDwk} and \eqref{cvbk}-\eqref{cvck}, there exists a sequence $\eps_k\to 0^+$ such that, for all $x\in [0,1]$,
$$\bigl[\tilde A_kw_k'\bigr]_0^x\le \int_0^x|\tilde b_kw_k'+\tilde c_k|ds+\bigl|\bigl[\tilde b_{1,k}\bigr]_0^x\bigr|\le \eps_k.$$
Therefore, there exists $\eta>0$ such that, for all $k$ sufficiently large,
$$w_k'(x)\le \frac{-\tilde A_k(0)+\eps_k}{\tilde A_k(x)}\le -\eta,\ \hbox{ for all $x\in[0,1]$.}$$
In the second case, we similarly get $w_k'(x)\ge \eta$ for all $x\in[-1,0]$.
In either case, by the convergence \eqref{convwkw}, we  have
$\int_{-1}^1 |w'(x)|dx\ge \eta$, which proves the claim.

Now, by virtue of \eqref{convwkw}, we may pass to the limit
in \eqref{eqforwk2} with the help of \eqref{cvbk}--\eqref{cvck}, to deduce that, for any nonnegative, compactly supported $\varphi\in H^1(\R)$,
$\int a_\infty w'\varphi'-\int a_\infty(w')^2\varphi \ge 0$, hence
$\int w'\varphi'- (w')^2\varphi \ge 0$.
For any nonnegative $\psi\in C^\infty_0(\R)$, we may then take $\varphi:=e^w\psi$, which gives
$$\int (e^w)'\psi'=\int w'(e^w(w'\psi+\psi'))- (w')^2e^w\psi =\int w'\varphi'- (w')^2\varphi\ge 0.$$
In other words, $(e^w)''\le 0$ in the distributional sense.
Fix a mollifying sequence $(\rho_j)$. For all $j$, it follows in particular that
$(e^w\ast\rho_j)''\le 0$ in $\R$. The smooth function $e^w\ast\rho_j\ge 0$ is thus concave on $\R$ and
must therefore be constant, hence $(e^w\ast\rho_j)'=0$.
Recalling that $e^w\in W^{1,\infty}_{loc}(\R)$, we have $0=(e^w\ast\rho_j)'=(e^w)'\ast\rho_j\to (e^w)'$
in $L^1_{loc}(\R)$. Consequently $w^\prime e^w=(e^w)'=0$ a.e., hence $w'=0$ a.e.: a contradiction with \eqref{nonvanishw}.}
\end{proof}

\section{Proof of optimality of $L^\infty$ and Hopf estimates}
\label{sec-proofoptim}

We will  use the following simple lemma.

\smallskip

\begin{lem} \label{Mitoinfty}
Let $R,\eta>0$ and let $\ld_i$ be a sequence of operators of the form \eqref{defdiv}.
Assume that
$\displaystyle\lim_{i\to\infty}\,\Bigl\{\sup_{x\in\overline B_R} \inf_{B_\eta(x)\cap B_R} |c_i|\Bigr\}=\infty$
(resp. $|b_i|$).
Then $\displaystyle\lim_{i\to\infty} M(\ld_i,B_R)=\infty$.
\end{lem}

\begin{proof}
Set $M_i=M(\ld_i,B_R)$ and $r_i:=r_0(\ld_i,B_R)$.
If the conclusion fails we may assume that $\sup_i M_i<\infty$.
By \eqref{relMr0} it follows that $\bar r:=\inf_i r_i>0$.
Therefore, setting $\rho=\min(\bar r,\eta)$, we deduce that
$$M_i^{q/\gamma_q}\ge \|c_i\|^q_{q,r_i,B_R}\ge \sup_{x\in\overline B_R} \int_{B_\rho(x)\cap B_R} |c_i|^q\,dx
\ge c(n)\rho^n \sup_{x\in\overline B_R} \inf_{B_\eta(x)\cap B_R} |c_i|^q\to\infty$$
by our assumption: a contradiction. The same argument applies in the case of $b_i$.
\end{proof}

\begin{proof}[Proof of Proposition~\ref{Prop-Optim-Linfty}]
By rescaling as in \eqref{scaledcoeff} and by using Proposition \eqref{scaleinvtilde} (ii), we easily see that it suffices to consider the case $R=1$.
\smallskip

(i)
Let $\lambda\ge 2n$, $\phi=\frac12(1-|x|^2)$ and $v=e^{\lambda\phi}$.
By direct computation we have
$$\nabla\phi=-x,\quad \Delta\phi=-n,$$
$$\nabla e^{\lambda\phi}=\lambda e^{\lambda\phi}\nabla\phi=-\lambda e^{\lambda\phi}x,\qquad
\Delta e^{\lambda\phi}=e^{\lambda\phi}\bigl(\lambda \Delta\phi+\lambda^2 |\nabla\phi|^2\bigr)
=e^{\lambda\phi}\bigl(-\lambda n+\lambda^2 |x|^2\bigr).$$
Taking $c(x)=\lambda^2|x|^2-\lambda n$, it follows that
\be \label{optimLinfty1a}
-\ld v:=-\Delta v+cv=e^{\lambda\phi}\bigl(\lambda n-\lambda^2|x|^2+c\bigr)=0,\quad x\in\R^n.
\ee
Since $v>0$, we deduce in particular that $\lambda_1(-\ld,\rn)\ge 0$.
Next,
by using  \eqref{defM} and \eqref{culinfty}, we get
\be \label{optimLinfty1b}
M:=M(\ld,B_2)
=\|c\|^{\gamma_q}_{q,r_0,B_{2}}\le C(n) \|c\|_{L^\infty(B_{2})}^{1/2}\le C(n)\lambda.
\ee
Then, by \eqref{optimLinfty1a}, the function $u=v-1$ solves \eqref{optimLinfty1} with $f\equiv -c$
and we have
$$\frac{u(0)}{\|f\|_{L^q(B_1)}}\ge C(n)\frac{u(0)}{\|f\|_{L^\infty(B_1)}}\ge {C(n)\lambda^{-2}}\bigl(e^{\lambda/2}-1\bigr)\ge  C(n)e^{C(n)M}.$$
Since $M\to\infty$ as $\lambda\to\infty$ owing to Lemma~\ref{Mitoinfty} and the choice of $c$,
assertion (i) follows.
\smallskip

(ii) Fix any $\lambda\ge 1$ and set
$$u=e^{\lambda\sqrt{2}}-e^{\lambda\phi},\quad\hbox{where  } \phi(x)=\sqrt{1+|x|^2}.$$
By direct computation we have
$$\nabla\phi=\frac{x}{\phi},\quad \Delta\phi=\frac{n}{\phi}-\frac{|x|^2}{\phi^3},$$
\quad
$$\nabla e^{\lambda\phi}=\lambda e^{\lambda\phi}\nabla\phi=\lambda e^{\lambda\phi}\frac{x}{\phi},\quad
\Delta e^{\lambda\phi}=e^{\lambda\phi}\bigl(\lambda^2 |\nabla\phi|^2+\lambda \Delta\phi\bigr)
=\lambda e^{\lambda\phi}\Bigl(\lambda\frac{|x|^2}{\phi^2}+\frac{n}{\phi}-\frac{|x|^2}{\phi^3}
\Bigr),$$
hence
$$-\ld u:=-\Delta u+b\cdot \nabla u
=\lambda {\phi}^{-1}e^{\lambda\phi}\Bigl(\lambda\frac{|x|^2}{\phi}+n-\frac{|x|^2}{\phi^2}
-b\cdot x\Bigr).
$$
Let $\psi$ be a smooth cut-off such that $0\le \psi\le1$, with $\psi(x)=1$ for $|x|\le 3/8$ and $\psi(x)=0$ for $|x|\ge 3/4$.
We then choose
$$b(x)=\frac{x}{|x|}\Bigl(\lambda\frac{|x|}{\phi}+\frac{n}{|x|}-\frac{|x|}{\phi^2}\Bigr)(1-\psi(x)),\quad  x\in\R^n$$
and
$$
f:=-\Delta u+b\cdot \nabla u
=\lambda  e^{\lambda\phi}\Bigl(\lambda\frac{|x|^2}{\phi^2}+\frac{n}{\phi}-\frac{|x|^2}{\phi^3}\Bigr)\psi(x).
$$
By the support properties of $\psi$, we see that $b\in C^\infty(\R^n)$ and, using \eqref{bulinfty}, that
\be \label{optimLinfty1c}
M:=M(\ld,B_2)
=\|b\|^{\beta_q}_{q,r_0,B_{2}}\le C(n) \|b\|_{L^\infty(B_{2})}\le C_1\lambda
\quad\hbox{for all $\lambda\ge 1$,}
\ee
for some constant $C_1=C_1(n)>0$.
Also, since $\phi(x)=\sqrt{1+|x|^2}\le 5/4$ for $|x|\le 3/4$, we have $\|f\|_{L^q(B_1)}\le C\|f\|_{L^\infty(B_1)}\le C\lambda^2  e^{5\lambda/4}$.
Noting that
$$u(0)=e^{\lambda\sqrt{2}}-e^\lambda\ge Ce^{\lambda\sqrt{2}}$$
it follows from \eqref{optimLinfty1c} that
$$\frac{u(0)}{\|f\|_{L^q(B_1)}}\ge C\lambda^{-2}e^{{\lambda(\sqrt{2}-\frac54)}}
\ge Ce^{C\lambda}\ge Ce^{CM}.$$
Since $M\to\infty$ as $\lambda\to\infty$ owing to Lemma~\ref{Mitoinfty},
the assertion follows
(we of course have $\lambda_1(-\Delta+b\cdot\nabla,\R^n)\ge 0$ owing to $(-\Delta+b\cdot\nabla)1=0$).
\end{proof}

\begin{rem}\label{gradopt} Observe that the functions $u,f$ and the operator $\ld$ just constructed are such that
$
\sup_{B_1}u = e^{\lambda\sqrt{2}}- e^\lambda$, $\sup_{B_1}|\nabla u| = C\lambda e^{\lambda\sqrt{2}}$, so a similar argument as above gives the optimality of the gradient  estimate  in Theorem \ref{thm2} for drift operators. On the other hand, the optimality of the same estimate for Schr\"odinger operators is exhibited even by the eigenfunctions of the Laplacian for large eigenvalues. This is trivial to see if $\Omega$ is a cube, since then the eigenfunctions are products of trogonometric functions; for the unit ball it similarly follows from the explicit form of the eigenfunctions in terms of Bessel functions. A computationally simpler example is given by the function $u(x) = \sin(\lambda \phi(x))$, $\phi=\frac12(1-|x|^2)$, which satisfies $\Delta u + cu= f$, with $c= \lambda^2|x|^2$ and $f= -n\lambda\cos(\lambda\phi)$; then an argument similar to the one in (i) above shows \eqref{sharpC1gen} is optimal, since $\sup_{B_1}u =1$, $\sup_{B_1}|\nabla u| = \lambda$. The optimality of \eqref{sharpC1alphagen} can be proved similarly too.
\end{rem}

\begin{proof}[Proof of Proposition~\ref{Optim-spectral1}]
By rescaling as in the previous Proposition,  since $\lambda_1(-\mathcal{L},B_R)=R^{-2}\lambda_1(-\tilde{\mathcal{L}},B_1)$,
 it suffices to consider the case $R=1$.
\smallskip

(i) Let $\varphi_1>0$ be the first eigenfiunction of $-\ld^*$ in $B_1$  (normalized by $\|\varphi_1\|_\infty=1$).
Pick any $u\in C^2(\overline {B_1})$ such that $u>0$ in $B_1$ and $u=0$ on $\partial B_1$, and set $f=-\ld u$.
Since $\lambda_1=\lambda_1(\ld^*,B_1)$, we have
$$\lambda_1\int_{\{f>0\}} u\varphi_1\le\lambda_1\int_{B_R} u\varphi_1=-\int_{B_1} u\ld^*\varphi_1
=-\int_{B_1} \varphi_1\ld u=\int_{B_1} f\varphi_1\le \int_{\{f>0\}} f\varphi_1,$$
where we used integration by parts (cf.~e.g.~Lemma~\ref{weakdiv2}) in the second equality,
hence
\be \label{ratiol1}
\lambda_1\le \frac{\sup_{B_1} f}{\inf_{\{f>0\}} u}.
\ee

We first apply this to the function $u=\exp[\frac\lambda 2(1-|x|^2)]-1$ and the Schr\"odinger operator $\ld$
in the proof of Proposition~\ref{Prop-Optim-Linfty}(i),
with $\lambda\ge 2n$, for which $f=\lambda(n-\lambda|x|^2)$, hence $\{f>0\}\subset B_{1/\sqrt 2}$.
Inequalities \eqref{ratiol1} and \eqref{optimLinfty1b} then yield
$\lambda_1\le n\lambda e^{-\lambda/4}\le C_1e^{-C_0M}$.

We next apply this to the function $u=e^{\lambda\sqrt 2}-e^{\lambda\sqrt{1+|x|^2}}$
and the drift operator $\ld$ in the proof of Proposition~\ref{Prop-Optim-Linfty}(ii), for which $f\le C\lambda^2e^{\lambda\sqrt{1+|x|^2}}\chi_{\{|x|<3/4\}}$. Inequalities \eqref{ratiol1} and \eqref{optimLinfty1c} yield
$\lambda_1\le C\lambda^2e^{\lambda\frac54}(e^{\lambda\sqrt{2}} - e^{\lambda\frac54})^{-1}
\le C_1e^{-C_0\lambda}\le C_1e^{-C_0M}$.

\smallskip

(ii) Let $\lambda\ge 1$. First take $c\equiv-\lambda$ and $\ld=\Delta+c$. By \eqref{culinfty}, we get
$M:=M(\ld,B_1)
=\|c\|^{\gamma_q}_{q,r_0,B_1}\le C(n)\lambda^{1/2}$.
Denoting by $\mu_1$ the first eigenvalue of the Dirichlet Laplacian on $B_1$,
we  have $\lambda_1(-\ld,B_1)=\lambda+\mu_1\ge C_0(M+1)^2$.
\smallskip

Next take $b=-2\lambda e_1$ and $\ld=\Delta+b\cdot\nabla$. By \eqref{culinfty}, we get
$M:=M(\ld,B_1)=\|b\|^{\beta_q}_{q,r_0,B_1}\le C(n)\lambda$.
Setting $\phi(x)=e^{\lambda x_1}>0$, we have
$-\ld\phi=(-\lambda^2-\lambda b\cdot e_1)\phi=\lambda^2\phi$, hence
$\lambda_1(-\ld,B_1)\ge \lambda^2\ge C_0(M+1)^2$.
\end{proof}

\begin{proof}[Proof of Proposition~\ref{Prop-Optim-Hopf1}]
We can assume again that $R=1$. Indeed, applying the transformation $\tilde u(x)=u(y)=u(Rx)$, by the scaling property \eqref{scaleinvtilde5} in Proposition \ref{scaleinvtilde}, the fact that $Rd_1(x)=d_R(Rx)$ (here $d_R(y)=\mathrm{dist}(y, \partial B_R)$), and
$\int_{B_1}\tilde fd_1=
R^{1-n}\int_{B_R}fd_R$, we  have
$$u(y)= \tilde u(x)
\ge e^{-C_0(1+\tilde M)}\left(\int_{B_1}\tilde fd_1\right)d_1(x)=e^{-C_0(1+MR)}R^{-n}\left(\int_{B_R}fd_R\right)d_R(y).$$

Fix any $\lambda>0$ and let $\phi=\frac12(1-|x|^2)$.
We look for a (radial) solution of the form
$$u=e^{\lambda\phi}-1+K\phi,$$
with $K\ge 0$ to be determined.
We note  that $u>0$ in $B_1$ and $u=0$ on~$\partial B_1$.
By direct computation we have
$$\nabla\phi=-x,\quad \Delta\phi=-n,$$
$$\nabla e^{\lambda\phi}=\lambda e^{\lambda\phi}\nabla\phi=-\lambda e^{\lambda\phi}x,\qquad
\Delta e^{\lambda\phi}=e^{\lambda\phi}\bigl(\lambda \Delta\phi+\lambda^2 |\nabla\phi|^2\bigr)
=e^{\lambda\phi}\bigl(-\lambda n+\lambda^2 |x|^2\bigr).$$
For any smooth functions $b, c$ such that $b\cdot x\le 0$ and $c\ge 0$,
choosing $K=n^{-1}\|c\|_\infty$, it follows that
$$\begin{aligned}
f:=-\Delta u+b\cdot\nabla u+cu
&=e^{\lambda\phi}\bigl(\lambda n-\lambda^2|x|^2-\lambda b\cdot x+c\bigr)-c+K(n-b\cdot x+c\phi) \\
&\ge e^{\lambda\phi}\bigl(\lambda n-\lambda^2|x|^2-\lambda b\cdot x+c\bigr).
\end{aligned}$$
We now choose
$$
\begin{cases}
c=c_\lambda=\lambda^2|x|^2,\ \ b=0,& \hbox{in case of assertion (i)}\\
\noalign{\vskip 1mm}
b=b_\lambda=-\lambda x,\ \ c=0,& \hbox{in case of assertion (ii),}
\end{cases}
$$
hence
\be \label{optimHopf3}
\lambda=\|c\|_\infty^{1/2} \quad\hbox{(resp., $\|b\|_\infty$).}
\ee
In either case, we have $f\ge\lambda n e^{\lambda\phi}$, hence
\be \label{optimHopf4}
\int_{B_{1}} fd
\ge \lambda n\ds\int_{0<|x|<1/2} e^{\lambda(1-|x|^2)/2} (1-x)dx\ge C(n)\lambda e^{3\lambda/8}.
\ee
On the other hand, on the boundary $\partial B_1$ we have $|u_\nu|=-x\cdot \nabla u = \lambda+K=\lambda+n^{-1}\lambda^2$ (resp. $|u_\nu|=\lambda$ in case of assertion (ii)),
hence $|u_\nu|\le \lambda(1+\lambda)$ in either case.
Consequently,
we deduce from \eqref{optimHopf4}  that
$$|u_\nu|\,\Bigl(\int_{B_1} fd\Bigr)^{-1}\le C(n)\lambda(1+\lambda)e^{-3\lambda/8}
\le C(n)e^{-\lambda/4}\le C(n)e^{-C_1(n)M},$$
where we used \eqref{optimHopf3} and Proposition~\ref{lemr0ul}(i) to get the last inequality.
Finally, by Proposition~\ref{lemr0ul}(ii), the function
$M(\lambda):=\|c_\lambda\|^{\gamma_q}_{q,r_\lambda,B_1}$ {with $r_\lambda=r_0(\Delta-c_\lambda, B_1)$}
is continuous on $(0,\infty)$ and $\lim_{\lambda\to \infty}M(\lambda)=\infty$,
whereas $\lim_{\lambda\to 0}M(\lambda)=0$ owing to \eqref{bulinfty}-\eqref{culinfty}.
Therefore the range of $M(\lambda)$ is the whole half-line $(0,\infty)$ (and similarly for $b_\lambda$),
which completes the proof.
\end{proof}

\begin{rem} \label{Hopf-optL1d}
As mentioned near the end of Section~\ref{sec-optim}, the $L^1_d$ norm in the right-hand side of the quantitative Hopf inequality \eqref{conclHopf} cannot be replaced by the stronger norms $\|\cdot\|_{L^p_d}$ or $\|\cdot\|_{L^1_{d^\epsilon}}$
for $p>1$ or $\epsilon<1$.
Indeed, consider for instance the case of the Laplacian, and assume for contradiction that the estimate
$u(x)\ge C_1 \|f\|d(x), \ x\in\Omega,$ is true for one such norm with some constant $C_1>0$,
for each $u$ and $f$ such that $-\Delta u \ge f\ge0$ in $B_1$.
Denoting by $G(x,y)$ the Green function of the Dirichlet Laplacian in $B_1$
we have, by classical estimates (see e.g.~\cite{Zh} and the references therein):
$$G(x,y)\le Cd(x)d(y)\quad\hbox{for all $x,y\in B_1$ such that $|x-y|\ge 1/4$}.$$
Consequently, for any nonnegative $f\in C(\overline B_1)$ with support in $\overline B_1\setminus B_{1/2}$,
$$C_1\|f\| \le  u(0)=\int_\Omega G(0,y)f(y)dy\le C\int_\Omega f(y)d(y)dy= C\|f\|_{L^1_d},$$
a contradiction.
\end{rem}

\section{Optimality of global Harnack and log-grad estimates}
\label{sec-proofoptim2}

We first state and prove the optimality of the exponential constant in front of the right-hand side $f$ in the Harnack inequality with respect to the coefficients.

\begin{prop}\label{Optimality_Harnack2}
Let $n\ge 1$, $n<q\le\infty$, and $R>0$.
There exist a constant $\lambda_0=\lambda_0(n)>0$,  sequences $b_j, c_j,f_j\in C^\infty(\overline{B_R})$,
and a classical solution $u_j>0$ of
\be \label{optimLinfty1d}
\left\{\hskip 2mm\begin{aligned}
-\ld_ju_j:=-\Delta u_j+c_ju_j&=f_j, &\quad&x\in B_R,\\
u_j&=0, &\quad&x\in \partial B_R,
\end{aligned}
\right.
\ee
resp. of
\be \label{optimLinfty1db}
\left\{\hskip 2mm\begin{aligned}
-\ld_ju_j:=-\Delta u_j + b_j\cdot\nabla u_j - \lambda_0R^{-2} u_j&=f_j, &\quad&x\in B_R,\\
u_j&=0, &\quad&x\in \partial B_R,
\end{aligned}
\right.
\ee
such that
\be \label{optimLinfty24}
\inf_{B_R}(u_j/d)=0\qquad\mbox{and}\qquad\sup_{B_R}(u_j/d)\ge CR^{1-n/q}\exp\left(CM_jR\right) \|f\|_{q,r_j,B_R},
\ee
where $M_j=\|c_j\|^{\gamma_q}_{q,r_j,B_R}\to\infty$, resp.~$M_j=\|b_j\|^{\beta_q}_{q,r_j,B_R}\to\infty$,
with $r_j=r_0(\ld_j,B_R)$ and $C=C(n)>0$.

\end{prop}

\begin{proof}

It suffices to assume $R=1$.
Indeed, as in the proof of Proposition~\ref{Prop-Optim-Hopf1} above, if $\tilde u_j(x)=u_j(y)=u_j(Rx)$, by the properties \eqref{scaleinvtilde4}-\eqref{scaleinvtilde5}
and  $Rd_1(x)=d_R(Rx)$, we  have
$$R\sup_{y\in B_R}\frac{u_j(y)}{d_R(y)}=\sup_{x\in B_1}\frac{\tilde u_j(x)}{d_1(x)}\ge C\exp(C\tilde M_j) \|\tilde f_j\|_{q,\tilde r_j,B_1}
=C\exp(CM_jR) R^{2-n/q}\|f_j\|_{q,r_j,B_R}.$$

Let $\lambda_0>0$ be the first eigenvalue of $-\Delta$ in $B_1$ with Dirichlet boundary conditions
and $\varphi(x)=\tilde\varphi(r)\in C^\infty(\overline B_1)$ be the corresponding (radial) eigenfunction with $\varphi(0)=1$.

\smallskip

{\bf Case 1:} Problem \eqref{optimLinfty1d}.
Let $h=\varphi^2$, $u_j=e^{j h}-1$, $c_j=-j\Delta h-j^2|\nabla h|^2$ and $f_j=c_j$. Then
$$
-\Delta u_j=
-\bigl[j\Delta h+j^2|\nabla h|^2\bigr]e^{j h}=c_ju_j+f_j.
$$
Since $u_j=e^{j h}-1\sim j \varphi^2=O(d^2)$ as $x\to\partial B_1$, we have $\inf_{B_1}(u_j/d)=0$.
On the other hand, $\sup_{B_1}(u_j/d)\ge u_j(0)=e^j-1$.
Since
{$M_j\le C \|c_j\|_\infty^{1/2} \le Cj$ owing to \eqref{culinfty}, we get
$$\bigl(\|f_j\|_{L^q(B_1)}\bigr)^{-1}\sup_{B_1}(u_j/d)=\bigl(\|c_j\|_{L^q(B_1)}\bigr)^{-1}\sup_{B_1}(u_j/d)\ge j^{-1}(e^j-1)\ge Ce^{j/2}\ge Ce^{CM_j}.$$
Since $M_j\to\infty$ by Lemma~\ref{Mitoinfty}}, the conclusion
follows.

\smallskip

{\bf Case 2:} Problem \eqref{optimLinfty1db}. This case is more delicate.
Let $h(x)=\tilde h(r)$ and $\theta(x)=\tilde\theta(r)$ be radial nonincreasing $C^\infty$ functions, such that
$$\hbox{$\tilde\theta(r)=1$ for $r\le 1/2$, \ \ $\tilde\theta_1(r)<0$ for $1/2<r<1$,\ \ $\tilde\theta(1)=\tilde\theta^\prime(1)=0$,}$$
$$\hbox{$\tilde h(r)=1$ for $r\le 1/3$, \ \ $\tilde h(r)=0$ for $2/3\le r\le 1$, \ \ $h(1/2)=1/2$}$$
and
\be \label{optimLogg-nabla0}
\hbox{$\tilde h'(r)<-1/10$ for $3/8\le r\le 5/8$}.
\ee
For given integer $j\ge 1$, we set
$$u_j:=\theta\varphi e^{j h},$$
which in particular satisfies $u_j=0$ on $\partial B_1$ and $u_j>0$ in $B_1$.
We compute
\be \label{optimLogg-nabla1}
\nabla u_j=\bigl[\nabla(\theta\varphi)+j\theta\varphi\nabla h\bigr]e^{j h},
\ee
$$\Delta u_j
=\bigl[\Delta(\theta\varphi)+2j\nabla(\theta\varphi)\cdot\nabla h+\theta\varphi(j\Delta h+j^2|\nabla h|^2)\bigr]e^{j h}.$$
Setting $f_j=-\bigl(2\nabla\theta\cdot\nabla\varphi+\varphi\Delta\theta\bigr)e^{j h}$ and using $\Delta(\theta\varphi)=-\lambda_0\theta\varphi+2\nabla\theta\cdot\nabla\varphi+\varphi\Delta\theta$,
it follows that
\be \label{optimLogg-nabla11}
\Delta u_j+\lambda_0 u_j+f_j
=\bigl[2j\nabla(\theta\varphi)\cdot\nabla h+\theta\varphi(j\Delta h+j^2|\nabla h|^2)\bigr]e^{j h(x)}.
\ee
Also we observe that
\be \label{optimLogg-nabla11a}
\inf_{B_1}\frac{u_j}{d}\le\lim_{|x|\to 1}\frac{u_j}{d}=|(\tilde\theta\tilde\varphi)'(1)|=0
\ee
and
$$\sup_{B_1}\frac{u_j}{d}\ge \frac{u_j}{d}(0)=e^j,
\qquad |f_j|\le C\chi_{\{1/2\le |x|\le 1\}}e^{j h(x)}\le Ce^{j/2},$$
hence
\be \label{optimLogg-nabla11b}
\sup_{B_1}\frac{u_j}{d}\ge Ce^{j/2}\|f_j\|_\infty.
\ee
Now, since $\varphi$ is radially strictly decreasing and  $\theta$, $h$ are also radially decreasing, we have
\be\label{gradsy}
|\nabla(\theta\varphi)+j\theta\varphi\nabla h|
\ge \max\{|\nabla(\theta\varphi)|, j|\theta\varphi\nabla h|\}\ge \theta|\nabla\varphi|>0, \quad \mbox { for }0<|x|<1.
\ee
We may then set
\be \label{Deltaujbj21}
b_j:=\frac{\nabla(\theta\varphi)+j\theta\varphi\nabla h}{|\nabla(\theta\varphi)+j\theta\varphi\nabla h|^2}
\bigl[2j\nabla(\theta\varphi)\cdot\nabla h+\theta\varphi(j\Delta h+j^2|\nabla h|^2)\bigr],\quad x\in B_1\setminus\{0\}.
\ee
Since $h$ is constant, hence $b_j=0$, in $B_{1/3}$ and in $\overline B_1\setminus B_{2/3}$, the function $b_j$ extends to a function $b_j\in C^\infty(\overline B_1)$,
and \eqref{optimLogg-nabla1}, \eqref{optimLogg-nabla11} and \eqref{Deltaujbj21} imply
$$-\Delta u_j-\lambda_0 u_j+b_j\cdot\nabla u_j=f_j.$$
Moreover, since $h=$const in $B_{1/3}$ and $B_1\setminus B_{2/3}$, by \eqref{gradsy} we obtain
$$\begin{aligned}|b_j|
&=\frac{\bigl|2j\nabla(\theta\varphi)\cdot\nabla h+\theta\varphi(j\Delta h+j^2|\nabla h|^2)\bigr|}{|\nabla(\theta\varphi)+j\theta\varphi\nabla h|}
\le j\frac{\bigl|2\nabla(\theta\varphi)\cdot\nabla h+\theta\varphi\Delta h\bigr|}{|\nabla(\theta\varphi)|}
+\frac{j^2\theta\varphi|\nabla h|^2}{j|\theta\varphi\nabla h|}\\
&\le j\frac{\bigl|2\nabla(\theta\varphi)\cdot\nabla h+\varphi\Delta h\bigr|}{|\theta\nabla \varphi|}+j|\nabla h|\le j\frac{C}{\inf_{1/3\le|x|\le 2/3}\theta|\nabla \varphi|}+Cj \le Cj,
\end{aligned}$$
hence
$M_j\le C\|b_j\|_\infty \le Cj$ owing to \eqref{culinfty}.
Since also by~\eqref{optimLogg-nabla0}, for large $j$
$$
|b_j(x)|\ge \frac{Cj^2|\nabla h|^2- Cj}{C(1+j)}\ge Cj\quad\mbox{ for } \;3/8\le |x|\le 5/8,$$
 we have $M_j\to\infty$ by Lemma~\ref{Mitoinfty}.
This along with \eqref{optimLogg-nabla11a}-\eqref{optimLogg-nabla11b}
yields the conclusion.
\end{proof}

We next turn to the proof of Proposition~\ref{Optimality_Harnack}.

\begin{proof}[Proof of Proposition~\ref{Optimality_Harnack}]
Again it suffices to consider the case $R=1$.
\smallskip

Let $\lambda_0>0$ be the first eigenvalue of $-\Delta$ in $B_1$ with Dirichlet boundary conditions
and $\varphi(x)=\tilde\varphi(r)\in C^\infty(\overline B_1)$ be the corresponding (radial) eigenfunction with $\varphi(0)=1$.
Let $h\ge 0$ be a radial nonincreasing $C^\infty$ function to be fixed below, such that
$h(0)=1$, $h(x)=0$ for $|x|=1$ and
$|\nabla h|\ge c_0>0$ in a neighborhood of $x=1/2$.
For any given integer $j\ge 1$, we set
$$u_j(x):=\varphi(x)e^{j h(x)},$$
which in particular satisfies $u_j=0$ on $\partial B_1$ and $u_j>0$ in $B_1$.
We compute
\be \label{optimLogg-nabla}
\nabla u_j=\bigl[\nabla\varphi+j\varphi\nabla h\bigr]e^{j h(x)},
\ee
$$\begin{aligned}
\Delta u_j
&=\bigl[\Delta\varphi+2j\nabla\varphi\cdot\nabla h+\varphi(j\Delta h+j^2|\nabla h|^2)\bigr]e^{j h(x)}\\
&=\bigl[2j\nabla\varphi\cdot\nabla h+\varphi(j\Delta h+j^2|\nabla h|^2-\lambda_0)\bigr]e^{j h(x)}.
\end{aligned}$$
On the other hand, since $\varphi, h$ are radially decreasing, we have
$$\frac{|\nabla u_j|}{u_j}=|j\nabla h+\varphi^{-1}\nabla\varphi|\ge j |\nabla h|,$$
hence there exists a constant $C_0>0$ such that
\be \label{optimLogg}
\frac{|\nabla u_j(x)|}{u_j(x)}\ge C_0j\quad\hbox{on $\{|x|=1/2\}$}.
\ee
Moreover, $\frac{u_j}{d}(0)=e^j$ and $\displaystyle\lim_{|x|\to 1}\frac{u_j}{d}=|\tilde\varphi'(1)|>0$ by Hopf lemma applied to $\varphi$,
hence
\be \label{optimBHI}
\frac{\sup_{B_1}(u_j/d)}{\inf_{B_1}(u_j/d)}\ge|\tilde\varphi'(1)|^{-1}e^j.
\ee

$\bullet$ For problem \eqref{optimLinfty1ci}, we choose $h=\varphi^2$, hence
$$\Delta u_j
=\varphi\bigl[4j|\nabla\varphi|^2+j\Delta h+j^2|\nabla h|^2-\lambda_0\bigr]e^{j h(x)}$$
and we set
$$c_j:=\frac{-\Delta u_j}{u_j}=\lambda_0-4j|\nabla\varphi|^2-j\Delta h-j^2|\nabla h|^2,$$
hence $\Delta u_j+c_ju_j=0$.
There exist $C_1, C_2>0$ such that $C_1j^2\le\|c_j\|_\infty\le C_2j^2$ for all sufficiently large $j$.
This along with \eqref{optimLogg}-\eqref{optimBHI} and Proposition~\ref{lemr0ul}(i) yields \eqref{optimLinfty23}-\eqref{optimLinfty2b}.
\smallskip

$\bullet$ For problem \eqref{optimLinfty1bi}, we take $h$ a radial decreasing $C^\infty$ function such that $h=1$ for $0\le |x|\le 1/3$
and $h=0$ for $2/3\le |x|\le 1$ and $|\nabla h|>C_0$ for $x=1/2$. We have
\be \label{Deltaujbj1}
\Delta u_j+\lambda_0 u_j=
\bigl[2j\nabla\varphi\cdot\nabla h+\varphi(j\Delta h+j^2|\nabla h|^2)\bigr]e^{j h(x)}.
\ee
Similarly to \eqref{gradsy}, since $\varphi$ is strictly radially decreasing and $h$ is radially decreasing, we have $|\nabla\varphi+j\varphi\nabla h|
\ge \max\{|\nabla\varphi|, j|\varphi\nabla h|\}>0$ in $\overline B_1\setminus\{0\}$.
We may then set
\be \label{Deltaujbj2}
b_j=-\frac{\nabla\varphi+j\varphi\nabla h}{|\nabla\varphi+j\varphi\nabla h|^2}
\bigl[2j\nabla\varphi\cdot\nabla h+\varphi(j\Delta h+j^2|\nabla h|^2)\bigr],\quad x\in\overline B_1\setminus\{0\}.
\ee
Since $h=1$, hence $b_j=0$, in $B_{1/3}$, the function $b_j$ extends to a function $b_j\in C^\infty(\overline B_1)$,
and \eqref{optimLogg-nabla}, \eqref{Deltaujbj1} and \eqref{Deltaujbj2} imply
$$\Delta u_j+\lambda_0 u_j+b_j\cdot\nabla u_j=0.$$
Moreover,
$$\begin{aligned}|b_j|
& =\frac{\bigl|2j\nabla\varphi\cdot\nabla h+\varphi(j\Delta h+j^2|\nabla h|^2)\bigr|}{|\nabla\varphi+j\varphi\nabla h|}
\le j\frac{\bigl|2\nabla\varphi\cdot\nabla h+\varphi\Delta h\bigr|}{|\nabla\varphi|}
+\frac{j^2\varphi|\nabla h|^2}{j|\varphi\nabla h|}\\
&\le j\frac{\bigl|2\nabla\varphi\cdot\nabla h+\varphi\Delta h\bigr|}{\min_{1/3\le|x|\le1}|\nabla\varphi|}+j|\nabla h|\le Cj.
\end{aligned}$$
We have $M_j\to\infty$ as in the previous proof. This along with \eqref{optimLogg}-\eqref{optimBHI} and Proposition~\ref{lemr0ul}(i)
yields \eqref{optimLinfty23}-\eqref{optimLinfty2b}.
\end{proof}

We now turn to the partial optimality of the log-grad estimate \eqref{concllgeu2} in the inhomogeneous case $f\ne 0$.

First of all, as mentioned in Remark~\ref{remloggrad1d2}, \eqref{concllgeu2} is false without the assumption $f\ge 0$, even for $n=1$.
It suffices to consider for instance the functions $u,f\in C^\infty([0,1])$, with $u>0$ in $(0,1)$, given by
\be\label{counterfpositive}
u(x)=\exp(-x^{-1}),\qquad f(x):=-x^{-3}\bigl(2+x^{-1}\bigr)\exp(-x^{-1}),
\ee
which satisfy
$-u''=f$ in $(0,1)$, whereas $\frac{x|u'|}{u}=x^{-1}$ is unbounded.

Next, we have the following proposition, announced in Remark~\ref{remloggrad1d}. It shows that, in the case $n=2$, for any $p>1$ estimate \eqref{concllgeu2} fails if $\|f\|_{L^1_d}$
is replaced by $\|f\|_{L^p_d}$ and $q$ is strictly larger but close to $n$.
In other words, in this situation $C^1$-estimates are available but the log-grad estimate fails.
In dimension $n\ge 3$, we have the same conclusion only for $p>n/2$, so that the optimality
of the $L^1_d$ norm in \eqref{concllgeu2} remains unclear in this case.

\begin{prop} \label{PropLoggradOptimality}
Let $n\ge 2$ and $\frac{n}{2}<p<n<q<\frac{np}{n-p}$.
There exists a sequence of functions $f_k\in C^\infty(\overline B_1)$
with $f_k>0$ in $\overline B_1$, such that the solutions $u_k>0$ of
$$
\left\{\hskip 2mm\begin{aligned}
-\Delta u_k&=f_k, &\quad&x\in B_1,\\
u_k&=0, &\quad&x\in \partial B_1
\end{aligned}
\right.
$$
satisfy, for all $k\ge 1$,
$$\sup_{B_1}\frac{d|\nabla u_k|}{u_k} \ge k\Bigl(1+\frac{\|f_k\|_{L^q(B_1)}}{\|f_k\|_{L^p_d(B_1)}}\Bigr).$$
\end{prop}

\begin{proof} [Proof of Proposition~\ref{PropLoggradOptimality}]
Let $\eps>0$. Define
$$\phi(r)=(r^2+\eps)^\gamma,\quad u(x)=u_\eps(x)=(1+\eps)^\gamma-\phi(r)$$
with $\gamma\in(0,1)$ to be chosen below. For $r>0$, we compute
\be\label{computphiprime}
\phi'=2\gamma r(r^2+\eps)^{\gamma-1},
\ee
$$\begin{aligned}
\phi''
&=2\gamma (r^2+\eps)^{\gamma-1}+4\gamma(\gamma-1)r^{2}(r^2+\eps)^{\gamma-2}\\
&=2\gamma (r^2+\eps)^{\gamma-2}\bigl[(r^2+\eps)+2(\gamma-1)r^2\bigr].
\end{aligned}$$
Therefore,
$$\begin{aligned}
f:=-\Delta u
&=\phi''+(n-1)r^{-1}\phi'
=2\gamma (r^2+\eps)^{\gamma-2}\bigl[n(r^2+\eps)+2(\gamma-1) r^2\bigr] \\
&=2\gamma (r^2+\eps)^{\gamma-2}\bigl[(n-2+2\gamma)r^2+n\eps],
\end{aligned}$$
hence
\be\label{computf}
2\gamma n (r^2+\eps)^{\gamma-1}\ge f\ge 4\gamma(r^2+\eps)^{\gamma-2}\bigl[\gamma r^2+\eps\bigr]\ge 4\gamma^2(r^2+\eps)^{\gamma-1}>0.
\ee
Let $p\ge 1$ and
$\eps\in(0,1/4)$.
Denoting by $C$ a generic positive constant depending only on $p$ and $\gamma$, we compute
$$
\|f\|^p_{L^p_d}\ge C \int_0^1 (r^2+\eps)^{(\gamma-1)p}r^{n-1}(1-r)dr
\ge C\eps^{(\gamma-1)p} \int_0^{\sqrt{\eps}} r^{n-1}dr
= C\eps^{(\gamma-1)p+n/2}
$$
hence
$\|f\|_{L^p_d}\ge C\eps^{\gamma-1+\frac{n}{2p}}$.
Also, if $q\ge 1$ and ${2(\gamma-1)q+n}<0$, then
$$\begin{aligned}
\|f\|_q^q
&\le C \int_0^1 (r^2+\eps)^{(\gamma-1)q}r^{n-1}dr\\
&{\le C\eps^{(\gamma-1)q} \int_0^{\sqrt{\eps}} r^{n-1}dr
+C \int_{\sqrt{\eps}}^1 r^{2(\gamma-1)q+n-1}dr\le C\eps^{(\gamma-1)q+n/2}},
\end{aligned}$$
hence $\|f\|_q\le C\eps^{\gamma-1+\frac{n}{2q}}$,
so that
$$\frac{\|f\|_q}{\|f\|_{L^p_d}}
\le C \eps^{-\frac{n}{2}(\frac{1}{p}-\frac{1}{q})}.$$
But on the other hand, by \eqref{computphiprime}, we have
$$\sup_{B_1}\frac{d|\nabla u|}{u}\ge \Bigl[\frac{d|\nabla u|}{u}\Bigr]_{|x|=\eps^{1/2}} \ge C|\nabla u|_{|x|=\eps^{1/2}}
\ge C\eps^{\gamma-\frac12}.$$
By our assumption, which implies $\frac{1}{p}-\frac{1}{q}<\frac{1}{n}$,
we may choose $\gamma>0$ small such that $0<\gamma<1-\frac{n}{2q}$ and
$\frac{n}{2}(\frac{1}{p}-\frac{1}{q})<\frac12-\gamma$, hence
$\eps^{-\frac{n}{2}(\frac{1}{p}-\frac{1}{q})}\ll \eps^{\gamma-\frac12}$ as $\eps\to 0$.
The proposition is proved.
\end{proof}

\section{Appendix} \label{sec-app}

In this appendix we state and/or prove a number of auxiliary or technical results that we have used, and which
were postponed in order not to interrupt the main line of arguments.

\subsection{First eigenvalue}
Let $\Omega$ be a bounded
domain of $\rn$. Using the weak version of the Krein-Rutman theorem (see for instance \cite[Proposition 5.4.32]{DM}), it was proved in \cite{Ch1}, \cite{Ch2}, that any operator $\ld$ satisfying  \eqref{hyp1}, $b_1,b_2\in L^q(\Omega)$, $c,f\in L^{q/2}(\Omega)$, $q>n$,  possesses a first eigenvalue $\lambda_1=\lambda_1(-\ld,\Omega)$ (for the reader's convenience we note that what we call $\lambda_1$ is $-\lambda_1$ in \cite{Ch1}, \cite{Ch2}). This eigenvalue has the usual basic properties, specifically, it is proved in these works that  $\lambda_1$ is a simple eigenvalue, has smallest real part among all eigenvalues, decreases (strictly) with respect to the domain, and corresponds to a positive eigenfunction $\varphi_1\in H^1_0(\Omega)$. We have the characterization
\begin{equation}\label{charlambda1}
\lambda_1=\mathrm{sup}\{\lambda>0\::\: \mbox{there exists } w\in H^1(\Omega),\; w>0,\;  (\ld+\lambda)w\le0 \mbox{ in }\Omega\}.
\end{equation}
Furthermore the validity of the maximum principle for $\ld$ in $\Omega$ is equivalent to the existence of a nonnegative solution $w$ of $\ld[w]<0$ (or nontrivial nonnegative solution of $\ld[w]\le0$) and hence  {\it the positivity of $\lambda_1$ is equivalent to the validity of the maximum principle for $\ld$ in $\Omega$}.  The latter in turn easily implies that $\lambda_1>0$ guarantees the general solvability of the Dirichlet problem for \eqref{defeq}  (see for instance the beginning of the proof of Proposition 4.1 in \cite{SS2}). We also know that $\varphi_1\in C^\alpha(\Omega)$, by \cite[Theorem 8.29]{GT} and the remark at the end of \cite[Section 8.10]{GT}.

For a more general approach to the existence and properties of the first eigenvalue, see the recent preprint \cite{FGM}. It is also worth observing that under \eqref{hyp1}-\eqref{hyp2} the strong version of the Krein-Rutman theorem (see for instance \cite[Theorem 5.4.33]{DM}) is also applicable, since for sufficiently large $C>0$ the operator $-\ld+C$ is invertible, its inverse is a compact operator from $C^1(\overline{\Omega})$ to itself (by the $C^{1,\alpha}$ estimates in \cite[Section~8.11]{GT}, \cite[Chapter 5.5]{Mo}), as well as strictly monotone on the positive cone of $C^1(\overline{\Omega})$ which has non-empty interior (by the strong maximum principle and the Hopf lemma).

In the case $\Omega=\rn$,
if $\ld$ is an operator whose coefficients are defined on $\rn$ and satisfy the above assumptions
in $B_R$ for every $R>0$, we define
$$ \lambda_1(-\mathcal{L},\R^n):=\lim_{R\to\infty}\lambda_1(-\mathcal{L},B_R)\in[-\infty,\infty).$$

The following proposition gives a basic lower bound on the first eigenvalue, which is used in the proof of
the optimized quantitative Hopf lemma in Theorem~\ref{OptimizedHopf}.

\begin{prop}\label{lowerbdeig}
We have
$$
\lambda_1(-L, \Omega) \ge \frac{c_0}{|\Omega|^\sigma} - \|b_1\|_{L^q(\Omega)} - \|c\|_{L^{q/2}(\Omega)},
$$
where $\sigma$ depends on $n, q$, and $c_0$ depends on $n,\lambda, \Lambda, q$, and upper bounds for
$\|b_2\|_{L^q(\Omega)} $, $|\Omega|$.
\end{prop}

\begin{proof} We have
$$
-\mathrm{div}(A(x) D\varphi_1) - b_2(x) D\varphi_1  = (\lambda_1+ c(x))\varphi_1 + \mathrm{div}(b_1(x)\varphi_1),
$$
in $\Omega$, $\varphi_1=0$ on $\partial \Omega$, so by
\cite[Theorem 8.16]{GT},
$$
\sup_\Omega \varphi_1\le
C( \| (\lambda_1+ c)\varphi_1 \|_{L^{\tilde q/2}(\Omega)} + \|b_1\varphi_1\|_{L^{\tilde{q}}(\Omega)} )
$$
where $\tilde{q} = (q+n)/2>n$ and $C$ is bounded in terms of $n,\lambda, \Lambda,  q$, and upper bounds for
$\|b_2\|_{L^q(\Omega)} $ and  $|\Omega|$. By H\"older inequality the latter term is smaller than
$$
C|\Omega|^\sigma (\lambda_1+ \|c\|_{L^{q/2}(\Omega)} + \|b_1\|_{L^{q}(\Omega)}) \sup_\Omega \varphi_1$$
and the result follows by cancelling $\sup_\Omega \varphi$.
\end{proof}

We next prove Proposition~\ref{upperboundlambda1}, which provides an optimized upper bound on the first eigenvalue in terms
of the  uniformly local norms of the coefficients of $\ld$ and the domain,
and is required in the proof of some of the main theorems above.

\begin{proof}[Proof of Proposition~\ref{upperboundlambda1}]
{We will show that there exist $C_0, K>0$, depending only on $n,\lambda, \Lambda, \alpha, q$, such that
\be\label{upperbdeig0}
\begin{aligned}
&\lambda_1:=\lambda_1(-\ld, B_R)\le C_0(R^{-1}+\eta)^2,\ \hbox{ where } \\
&\eta:=[A]^{\frac{1}{\alpha}}_{\alpha,\frac{K}{\sqrt{\lambda_1}},B_R}
+\|b_1\|_{L^\infty(B_R)}+[b_1]^{\frac{1}{1+\alpha}}_{\alpha,\frac{K}{\sqrt{\lambda_1}},B_R}
+\|b_2\|^{\beta_q}_{q,\frac{K}{\sqrt{\lambda_1}},B_R}+\|c\|^{\gamma_q}_{q,\frac{K}{\sqrt{\lambda_1}},B_R}.
\end{aligned}
\ee
We observe that \eqref{upperbdeig0} is equivalent to \eqref{upperbdeig1}.
Indeed, if \eqref{upperbdeig0} holds and $\eta\le M$ then \eqref{upperbdeig1}
is true, whereas
$\eta> M$ implies $K/\sqrt{\lambda_1}\ge r_0:=r_0(\ld,B_R)$ by the definition of $M$, and then $\lambda_1\le K^2r_0^{-2}\le C(R^{-1}+M)^2$ by \eqref{relMr0},
hence again~\eqref{upperbdeig1}.}
Conversely, \eqref{upperbdeig1}
implies $\lambda_1\le Cr_0^{-2}$ by \eqref{relMr0}, i.e. $r_0\le \sqrt{C/\lambda_1}$,
so that $M\le \eta$ with $K= \sqrt{C}$ and \eqref{upperbdeig0} is true.

Let us show \eqref{upperbdeig0}.
By rescaling
\begin{equation}\label{loc11}\tilde A(y)=A(Ry),\quad \tilde b_i(y)=Rb_i(Ry),\quad \tilde c(y)=R^2c(Ry)\end{equation}
we see that $\lambda_1=R^{-2}\tilde\lambda_1$ where $\tilde\lambda_1$ is the eigenvalue of
the corresponding operator $\tilde\ld$ in $B_1$ and we have
$$\|\tilde c\|^{\gamma_q}_{q,K(\tilde\lambda_1)^{-1/2},B_1}=\Bigl\{R^2\|c(R\cdot)\|_{q,\frac{K}{R\sqrt{\lambda_1}},B_1}\Bigr\}^{\gamma_q}
=\Bigl\{R^{2-n/q}\|c\|_{q,\frac{K}{\sqrt{\lambda_1}},B_R}\Bigr\}^{\gamma_q}
=R\|c\|^{\gamma_q}_{q,\frac{K}{\sqrt{\lambda_1}},B_R}$$
\begin{equation}\label{loc12}[\tilde A]^{\frac{1}{\alpha}}_{\alpha,K(\tilde\lambda_1)^{-1/2},B_R}=
\Bigl\{[A(R\cdot)]_{q,\frac{K}{R\sqrt{\lambda_1}},B_1}\Bigr\}^{\frac{1}{\alpha}}
=\Bigl\{R^\alpha[A]_{q,\frac{K}{\sqrt{\lambda_1}},B_R}\Bigr\}^{\frac{1}{\alpha}}
=R[A]^{\frac{1}{\alpha}}_{q,\frac{K}{\sqrt{\lambda_1}},B_R}\end{equation}
and analogous relations for $\tilde b_1, \tilde b_2$.
Consequently it is sufficient to establish \eqref{upperbdeig0} for $R=1$.

Assume for contradiction that \eqref{upperbdeig0} fails. Then, for each integer $j\ge 1$, there exist coefficients $A_j,b_{1,j},b_{2,j},c_j$ such that,
for the corresponding operator $\ld_j$, the eigenvalue $\mu_j:=\lambda_1(-\ld_j, B_1)$ satisfies
\be\label{mujlarge}
\mu_j\ge j\Bigl(1+
[A]^{\frac{1}{\alpha}}_{\alpha,\frac{j}{\sqrt{\mu_j}},B_1}+
\|b_1\|_{L^\infty(B_1)}+[b_1]^{\frac{1}{1+\alpha}}_{\alpha,\frac{j}{\sqrt{\mu_j}},B_1}
+\|b_{2,j}\|^{\beta_q}_{q,\frac{j}{\sqrt{\mu_j}},B_1 }+ \|c_j\|^{\gamma_q}_{q,\frac{j}{\sqrt{\mu_j}},B_1} \Bigr)^2.
\ee
Let $\varphi_j>0$ be the corresponding eigenfunction normalized by
\be\label{mujlarge2}
\sup_\Omega \varphi_j= \varphi_j(x_j)=1
\ee
and rescale
$\psi_j(y) = \varphi_j(x)  = \varphi_j(x_j+r_jy)$, where $r_j = 1/\sqrt{\mu_j} \to 0$.
The function $\psi_j$ then satisfies
\be\label{mujlarge3}
-\tilde\ld_j\psi_j=\psi_j\quad\hbox{in $G_j:=r_j^{-1}(B_1-x_j)$,}
\ee
where the coefficients of $\tilde\ld_j$ are given by
$$\tilde A_j(y)=A(x_j+r_jy),\quad \tilde b_{i,j}(y)=r_jb_{i,j}(x_j+r_jy),\quad \tilde c_j(y)=r_j^2c(x_j+r_jy).$$
Note that $G_j\to G$ where $G$ is either the whole space or a half-space.
For each fixed $L\ge 1$ and large $j$, we have
$$\|\tilde c_j\|_{L^q(B_L\cap G_j)}=r_j^{2-\frac{n}{q}}\|c_j\|_{L^q( B_{Lr_j}(x_j)\cap B_1)}
\le r_j^{2-\frac{n}{q}}\|c_j\|_{q,Lr_j,B_1}=\bigl(\mu_j^{-1}\|c_j\|^{2\gamma_q}_{q,L/\sqrt{\mu_j},B_1}\bigr)^{\frac{1}{2\gamma_q}},$$
$$[\tilde A_j]_{\alpha, B_L\cap  G_j}=r_j^\alpha[A_j]_{\alpha, B_{Lr_j}(x_j) \cap B_1}
\le r_j^\alpha[A_j]_{\alpha,Lr_j,B_1}=\bigl(\mu_j^{-1}[A_j]^{\frac{2}{\alpha}}_{\alpha,L/\sqrt{\mu_j},B_1}\bigr)^{\frac{\alpha}{2}},$$
hence $\|\tilde c_j\|_{L^q(B_L\cap G_j)}, [\tilde A_j]_{\alpha, B_L\cap  G_j}\to 0$ as $j\to\infty$ by \eqref{mujlarge},
and similarly we obtain that $\|b_{1,j}\|_{C^\alpha(B_L\cap  G_j)}\to 0$, $\|b_{2,j}\|_{L^q(B_L\cap G_j)}\to0$.
By \eqref{mujlarge2}, \eqref{mujlarge3}, the $C^{1,\alpha}$ estimates and compact embeddings, up to a subsequence $x_j\to x_0\in \bar B_1$,
we have $A_j\to A^0$ in $C_{loc}^{\alpha/2}(\overline G)$,
where $A^0=A(x_0)$ is a constant matrix satisfying \eqref{hyp1},
and $\psi_j\to \psi^0\ge 0$ in $C^1_{loc}(\overline G)$, where $\psi^0$ satisfies $\psi^0(0)=1$ and
$$
-\ld^0\psi^0= -\mathrm{tr}(A^0D^2\psi^0)= -\mathrm{div}(A^0D\psi^0) = \psi^0\quad \hbox{ in $G$.}
$$
By the standard characterization of the first eigenvalue \eqref{charlambda1} this implies that the first eigenvalue of
$-\ld^0$ is larger or equal to $1$ in any subdomain of $G$.
But $G$ contains balls of arbitrary radius and
$$
\lambda_1(-\ld^0, B_\rho) = \frac{\lambda_1(-\ld^0, B_1)}{\rho^2}
\to 0 \; \mbox{ as } \; \rho\to \infty,
$$
a contradiction.
\end{proof}

\subsection{Properties associated with uniformly local norms and scaling}

In this subsection we prove Propositions~\ref{basicr0}--\ref{lemr0ul}.

\begin{proof}[Proof of Proposition~\ref{basicr0}]
We only prove \eqref{relMr0}, the proof of \eqref{relMstarr0} and \eqref{relMhatr0} being similar.
The functions
$$h_1(r)=[A]_{\alpha, r, \Omega}, \quad  h_2(r)= [b_1]_{\alpha, r, \Omega},
\quad h_3(r) =\|b_2\|_{q,r,\Omega}, \quad h_4(r) = \|c\|_{q, r, \Omega}$$
are clearly nondecreasing on $[0,r_\Omega]$.
We claim that they are continuous.

The continuity of $h_3, h_4$ on the left is easy to show by using monotone convergence, whereas the continuity on the right follows from the fact that $\|c\|_{L^q(B_r(x_r)\setminus B_{r_0}(x_r)\cap\Omega)}\to 0$ if $r\searrow r_0$, $x_r\in \Omega$.

The continuity of $h_1, h_2$ on the left follows easily from the continuity of $A, b_1$.
Assume for contradiction that $h_2$ is not continuous on the right (the argument for $h_1$ is the same). Then there exist $r_0\in[0,r_\Omega)$,
$\eta>0$ and sequences $r_i\to r_0$ and $x_i\in\overline \Omega$, $y_i\ne z_i\in\Omega$ with $|y_i-x_i|, |z_i-x_i|<r_i$,
such that $|y_i-z_i|^{-\alpha}|b_1(y_i)-b_1(z_i)|\ge h_2(r_0)+\eta$.
By passing to a subsequence we may assume that $x_i\to x_0$, $y_i\to y_0$, $z_i\to z_0$ for some $x_0,y_0,z_0\in \overline \Omega$.
Note also that $b_1$ extends to a $C^\alpha$ function on $\overline \Omega$ and that $B_r(x)\cap \Omega$ can be replaced by
$\overline B_r(x)\cap \overline \Omega$ in definition \eqref{defHbracket}.
If $y_0\ne z_0$, then $h_2(r_0)+\eta\le |y_0-z_0|^{-\alpha}|b_1(y_0)-b_1(z_0)|\le h_2(r_0)$ (where the last inequality follows
from $|y_0-x_0|, |z_0-x_0|\le r_0$, and definition \eqref{defHbracket} with $x=x_0$): a contradiction.
If $y_0=z_0$, then we have $|y_i-y_0|, |z_i-y_0|<r_0$ for $i$ large enough, hence
$|y_i-z_i|^{-\alpha}|b_1(y_i)-b_1(z_i)|\le h_2(r_0)$, which is again a contradiction.
The claim is proved.

Let now
$$h(r)=r\Bigl(r^{-1}_\Omega +
[A]^{1/\alpha}_{\alpha, r, \Omega} + \|b_1\|_{L^\infty(\Omega)} + [b_1]^{1/(\alpha+1)}_{\alpha, r, \Omega}+
\|b_2\|^{\beta_q}_{q,r,\Omega} + \|c\|^{\gamma_q}_{q, r, \Omega}\Bigr), \quad r\in [0,r_\Omega].$$
By the above, $h$ is strictly increasing and continuous on $[0,r_\Omega]$ and, moreover, $h(0)=0$ and $h(r_\Omega)\ge 1$.
Consequently there exists a unique $r\in (0,r_\Omega]$ such that $h(r)=1$ and $r_0=r$, which implies \eqref{relMr0}.
\end{proof}

\begin{proof}[Proof of Proposition~\ref{scaleinvtilde0}]
If $r_0(\ld,\omega)\ge \theta^{-1} r_0(\ld,\Omega)$,
then \eqref{relMr0} and the assumption $r_{\Omega}\ge \theta r_{\omega}$ yield
$$M(\ld,\omega)=r_0^{-1}(\ld,\omega)-r_{\omega}^{-1}\le  \theta r_0^{-1}(\ld,\Omega)-\theta r_{\Omega}^{-1}
= \theta M(\ld,\Omega)\le M(\ld,\Omega).$$
If $r_0(\ld,\omega)\le \theta^{-1} r_0(\ld,\Omega)$ then, since any ball of radius
$r_0(\ld,\omega)$ can be covered by $C(n) \theta^{-n}$ balls of radius $r_0(\ld,\Omega)$,
it follows from \eqref{deful}-\eqref{defM}
that $M(\ld,\omega)\le C(n,p,q,\alpha,\theta)M(\ld,\Omega)$.
\end{proof}

\begin{proof}[Proof of Proposition~\ref{scaleinvtilde}]
For any $\tilde r>0$,
we have
$$\begin{aligned}
[\tilde A]_{\alpha, \tilde r, B_1}
&=\sup_{x\in\overline{B_1}}\sup_{y,z\in B_{\tilde r}(x)\cap \Omega} |y-z|^{-\alpha}|A(Ry)-A(Rz)| \\
&=R^{\alpha}\sup_{\hat x\in\overline{B_R}}\sup_{\hat y,\hat z\in B_{R\tilde r}(\hat x)\cap \Omega} |\hat y-\hat z|^{-\alpha}|A(\hat y)-A(\hat z)|
= R^{\alpha}[A]_{\alpha, R\tilde r, B_R}.
\end{aligned}$$
This and a similar argument for $b_1$ yields
\begin{equation}\label{scaleinvtildePf0}
[\tilde A]_{\alpha, \tilde r, B_1}= R^{\alpha}[A]_{\alpha, R\tilde r, B_R},\
[\tilde b_1]_{\alpha, \tilde r, B_1}= R^{1+\alpha}[b_1]_{\alpha, R\tilde r, B_R},\
\|\tilde b_1\|_{L^\infty(B_1)}=R\|b_1\|_{L^\infty(B_R)}.
\end{equation}
On the other hand, for $x\in B_1$, we have
$$\|\tilde b_2\|_{L^q(B_{\tilde r}(x)\cap B_1)}=R^{1-n/q}\|b_2\|_{L^q(B_{R\tilde r}(Rx)\cap B_R)},\quad
\|\tilde c\|_{L^q(B_{\tilde r}(x)\cap B_1)}=R^{2-n/q}\|c\|_{L^q(B_{R\tilde r}(Rx)\cap B_R)}$$
hence
\begin{equation}\label{scaleinvtildePf1}
\|\tilde b_2\|_{q,\tilde r,B_1}=R^{1-n/q}\|b_2\|_{q,R\tilde r,B_R},\quad
\|\tilde c\|_{q,\tilde r,B_1}=R^{2-n/q}\|c\|_{q,R\tilde r,B_R)}.
\end{equation}
Using also $r_{B_R}=c(n)R$, it follows that, for any $\tilde r>0$,
$$\tilde r\bigl(r^{-1}_{B_1} + [\tilde A]^{\frac{1}{\alpha}}_{\alpha, \tilde r, B_1} + \|\tilde b_1\|_{L^\infty(B_1)} + [\tilde b_1]^{\frac{1}{\alpha+1}}_{\alpha, \tilde r, B_1}
+\|\tilde b_2\|^{\beta_q}_{q,\tilde r,B_1} + \|\tilde c\|^{\gamma_q}_{q, \tilde r,B_1}\bigr)\le 1 \Longleftrightarrow $$
$$\hbox{$r:=R\tilde r$ satisfies }
r \bigl(r^{-1}_{B_R}+ [A]^{\frac{1}{\alpha}}_{\alpha, r, B_R} + \|b_1\|_{L^\infty(B_R)} + [b_1]^{\frac{1}{\alpha+1}}_{\alpha, r, B_R}
+\|b_2\|^{\beta_q}_{q,r,B_R} + \|c\|^{\gamma_q}_{q,r,B_R}\bigr)\le 1.$$
This yields \eqref{scaleinvtilde2}.
Combining \eqref{scaleinvtilde2} and \eqref{scaleinvtildePf0} gives \eqref{scaleinvtilde2b}.
Next, by \eqref{scaleinvtilde2}, \eqref{scaleinvtildePf1},
$$\|\tilde b_2\|_{q,\tilde r_0,B_1}=R^{1-n/q}\|b_2\|_{q,r_0,B_R}$$
which, together with the similar properties for $\tilde c$, $\tilde f$, gives \eqref{scaleinvtilde4}.
\end{proof}

 \begin{proof}[Proof of Proposition~\ref{lemr0ul}]
(i) Recalling the definition \eqref{defM} and Proposition \ref{basicr0},
we have
$$ r_0^{1-n/q} \|b_2\|_{q,r_0,B_R}\le 1,\quad
r_0^{2-n/q} \|c\|_{q,r_0,B_R}\le 1.$$
By H\"older's inequality it follows that
$$ \|b_2\|_{q,r_0,B_R} \le C(n)r_0^{n/q} \|b_2\|_{L^\infty(B_R)}
\le C(n)\|b_2\|^{-\frac{n}{q-n}}_{q,r_0,B_R}\|b_2\|_{L^\infty(B_R)},$$
hence
$$ \|b_2\|^{\frac{q}{q-n}}_{q,r_0,B_R}
\le C(n)\|b_2\|_{L^\infty(B_R)}$$
i.e., \eqref{bulinfty}, and \eqref{culinfty} is obtained similarly.

\smallskip

(ii) The upper estimates follow from assertion (i) (since the $L^\infty$ norm does not involve $r_0$).
To prove the lower estimate for $b$ (the case of $c$ is similar)
set $K_\lambda=\|\lambda b\|_{q,r_\lambda,B_R}\equiv \sup_{x\in B_R} \|\lambda b\|_{L^q(B_R\cap B_{{r_\lambda}}(x))}$.
By \eqref{relMr0}, we have
\be\label{Krlambda0}
r_\lambda\Bigl(r_{B_R}^{-1}+\lambda^{\beta_q}\sup_{x\in B_R} \|b\|^{\beta_q}_{L^q(B_R\cap B_{{r_\lambda}}(x))}\Bigr)
=r_\lambda(r_{B_R}^{-1}+K_\lambda^{\beta_q})=1=r_1(r_{B_R}^{-1}+K_1^{\beta_q}).
\ee
Assume $\lambda\ge 1$. Then \eqref{Krlambda0} implies
$$r_\lambda\Bigl(r_{B_R}^{-1}+\sup_{x\in B_R} \|b\|^{\beta_q}_{L^q(B_R\cap B_{{r_\lambda}}(x))}\Bigr)
\le r_1\Bigl(r_{B_R}^{-1}+\sup_{x\in B_R} \|b\|^{\beta_q}_{L^q(B_R\cap B_{{r_1}}(x))}\Bigr).
$$
By the definition of $r_1$ we get $r_\lambda\le r_1$ and, going back to \eqref{Krlambda0}, that $r_\lambda K_\lambda^{\beta_q} \ge r_1K_1^{\beta_q}$, hence
\be\label{Krlambda}
r_1^{-n/q}K_1^{-n\beta_q/q}K_\lambda^{n\beta_q/q}\ge r_\lambda^{-n/q}.
\ee
On the other hand, since any ball with radius $r_1$ can be covered by $c(n)(r_1/r_\lambda)^n$ balls with radius $r_\lambda$, we have
$$K_1=\sup_{x\in B_R} \|b\|_{L^q(B_R\cap B_{{r_1}}(x))}
\le C(n)  \Bigl(\frac{r_1}{r_\lambda}\Bigr)^{n/q}\sup_{x\in B_R} \|b\|_{L^q(B_R\cap B_{{r_\lambda}}(x))}
=C(n) \Bigl(\frac{r_1}{r_\lambda}\Bigr)^{n/q}\lambda^{-1} K_\lambda.$$
Combining this with \eqref{Krlambda}, we get
$$r_1^{-n/q}K_1^{-n\beta_q/q}K_\lambda^{1+n\beta_q/q} \ge r_\lambda^{-n/q}K_\lambda\ge C(n) \lambda r_1^{-n/q}K_1,$$
hence $K_\lambda^{1+n\beta_q/q} \ge C(n) \lambda K_1^{1+n\beta_q/q}$.
Since $1+\frac{n\beta_q}{q}=(1-\frac{n}{q})^{-1}$ this yields the lower estimate for $b$.
\smallskip

Let us finally check the continuity statement.
 As above, by \eqref{Krlambda0} $r_\lambda$ is nonincreasing with respect to $\lambda\in(0,\infty)$
and \eqref{Krlambda0} guarantees that
$$r_\lambda H(\lambda)=\lambda^{-\beta_q},
\quad\hbox{ where }
H(\lambda):=\lambda^{-\beta_q}{r_{B_R}^{-1}}+ \sup_{x\in B_R} \|b\|^{\beta_q}_{L^q(B_R\cap B_{{r_\lambda}}(x))}.$$
Since the functions $r_\lambda$ and $H(\lambda)$ are nonincreasing on $(0,\infty)$
and their product is a continuous function,
both functions are necessarily continuous. This implies that $K_\lambda=\|\lambda b\|_{q,r_\lambda,B_R}
=\lambda \sup_{x\in B_R} \|b\|_{L^q(B_R\cap B_{{r_\lambda}}(x))}$ is itself continuous.
\end{proof}

\subsection{Harnack inequalities}

 In this section, we prove the optimized interior and boundary Harnack inequalities that
are used throughout the whole paper.
 They rely on modifications of Harnack chain arguments from \cite{SS2},
adapted to general domains (and not only for balls).
To this end, we need the following proposition,
which guarantees the existence of suitable coverings and associated Harnack chains,
with optimal length estimate in terms of the geodesic diameter.
 This may be known but we could not find a reference in the literature,
so we provide a proof.
Here we use the notation given at the beginning of Section \ref{sec-ext}.

\begin{prop}\label{geodes}
 (i) Let $\Omega$ be an arbitrary bounded domain of $\R^n$ and let $r\in(0,D)$.
There exist constants $c_0,c_1>0$ depending only on $n$, integers $N,I\ge 2$ satisfying
\bel{coverA0}
N=N_r\le 2+\frac{6D}{r},\qquad I=I_r\le c_0\Bigl(1+\frac{D_0}{r}\Bigr)^n
\ee
and points $x_1,\dots,x_I\in \Omega$ such that
\bel{coverA1}
\Omega\subset\bigcup_{j=1}^I B_r(x_j)
\ee
and
\bel{coverA3}
\begin{aligned}
&\hbox{For each $k,l\in \{1,\ldots,I\}$ there exist $p\in\{2,\dots,N\}$ and $(j_1,\dots,j_p)\in \{1,\dots,I\}^p$}\\
&\hbox{such that $j_1=k$, $j_p=l$ and $|B_r(x_{j_{i-1}})\cap B_r(x_{j_i})|\ge c_1r^n$ for all $i\in \{2,\ldots,p\}$.}
\end{aligned}
\ee
\smallskip

(ii) Let $\Omega$ be a bounded domain of $\R^n$ with $C^{1,\bar\alpha}$ boundary and let $r\in(0,r_\Omega)$.
There exist constants $c_0,c_1>0$ depending only on $n$, integers $N,I\ge 2$
satisfying \eqref{coverA0}
and points $x_1,\dots,x_I$ satisfying \eqref{coverA1}, \eqref{coverA3},
\bel{coverA2}
\hbox{for each $j\in\{1,\dots,I\}$, either dist$(x_j, \partial \Omega)\ge 3r/2$ or $x_j\in \partial \Omega$}
\ee
and
\bel{coverA2b}
\hbox{for each $j\in\{1,\dots,I\}$, $|\Omega\cap B_r(x_j)|\ge c_0r^n$}.
\ee
\end{prop}

\begin{rem}
The upper bound \eqref{coverA0} on $N$ in Proposition~\ref{geodes}
(which implies $N\le \frac{8D}{r}$) is qualitatively optimal,
since any $N$ with such properties necessarily satisfies
\bel{optimN}
N\ge \frac{D}{2r}.
\ee
Indeed, fix any $x,y\in \Omega$, take $k,l$ such that $x\in B_r(x_k)$, $y\in B_r(x_l)$
and a chain $k=j_1<j_2<\dots<j_p=x_l$ with $p\le N$.
Then, using that $B_r(x_{j_{i+1}})\cap B_r(x_{j_i})\ne\emptyset$, we see that the geodesic distance of $x,y$ can be estimated by
$$d_\Omega(x,y)\le |x-x_k|+|y-x_l|+\sum_{i=2}^p |x_{j_{i}}-x_{j_{i-1}}|\le r+r+2(p-1)r=2pr\le 2Nr.$$
Taking supremum over $x,y\in \Omega$, we obtain \eqref{optimN}.
As for the upper bound on $I$ in \eqref{coverA0}, it is clearly optimal in general
(unless $\Omega$ has a specific, e.g.~tubular, geometry).
\end{rem}

\begin{proof}[Proof of Proposition~\ref{geodes}]
 We will only prove assertion~(ii), the proof of assertion~(i) being similar and easier.
We first claim that
\be\label{coverA00}
\exists I\le c_0(n)\Bigl(1+\frac{D_0}{r}\Bigr)^n
\hbox{and $x_1,\dots,x_I$ satisfying (\ref{coverA2}) and }
\Omega\subset\bigcup_{j=1}^I B_{5r/6}(x_j)
\ee
(hence in particular \eqref{coverA1}).

For $\eps>0$, we denote
$$\Omega_\eps=\{x\in\Omega;\ {\rm dist}(x,\partial\Omega)>\eps\},\quad
\omega_\eps=\{x\in\Omega;\ {\rm dist}(x,\partial\Omega)<\eps\}.$$
\indent $\bullet$ Since $\{B_{4r/5}(x);\ x\in\partial\Omega\}$ is an open covering of the compact $\overline\omega_{3r/4}$,
we can cover $\overline\omega_{3r/4}$ by finitely many such balls, whose set of centers we denote by $\Sigma_1$.

$\bullet$ Also, we can obviously cover the compact $\overline\Omega_{3r/2}$ by finitely many balls
of radius $4r/5$ and centered at points $x$ with $d(x)\ge3r/2$. We denote the set of their centers by $\Sigma_2$.

$\bullet$ Next consider the case when $x\in K:=\overline\omega_{3r/2}\setminus\omega_{3r/4}=\{3r/4\le d(x)\le 3r/2\}$.
Set $d=d(x)$, denote by $p_x$ the projection of $x$ on $\partial\Omega$ and by $\nu_x$ the outer normal vector at $p_x$,
hence
$x=p_x-d\nu_x$.
Let $\ell=\frac{10}{13}$ and $z=p_x-(d+\ell r)\nu_x$.
Note that $|x-z|=\ell r<4r/5$.
We claim that
\be\label{claimdz}
d(z)\ge 3r/2.
\ee
 Indeed, recalling the first paragraphs of Section \ref{sec-ext} and the notation therein,
after an orthonormal change of coordinates,
we may assume that $p_x=0$, $x=(0,d)\in\R^{n-1}\times\R$, $z=(0,d+\ell r)$
and that
\be\label{claimdz2}
\Omega\supset\bigl\{y=(y';y_n)\in \R^{n-1}\times\R:\ |y'|< \bar\rho_\Omega\ \hbox{and}\ k_\Omega |y'|^{1+\alpha}<y_n< \bar\rho_\Omega\bigr\}.
\ee
Working in the new coordinate, for any $y$ such that $|y-z|<3r/2$,
using $r<r_\Omega\le\min(\bar\rho_\Omega/4,$ $(120 k_\Omega)^{-1/\alpha})$, we obtain $y_n<d+\ell r+3r/2<4r<\bar\rho_\Omega$, $|y'|<3r/2<\bar\rho_\Omega$
and
$$y_n>d+\ell r-3r/2\ge \ell r-3r/4=r/52\ge k_\Omega (3r/2)^{1+\alpha}\ge k_\Omega |y'|^{1+\alpha},$$
hence $y\in\Omega$ owing to \eqref{claimdz2}. This proves \eqref{claimdz}.
Consequently,
 $\{B_{4r/5}(z)\::\: d(z)\ge 3r/2\}$ is an open covering of the compact
$K$ and we can extract a finite covering, whose set of centers we denote by $\Sigma_3$.

Enumerating $\Sigma:=\Sigma_1\cup\Sigma_2\cup\Sigma_3$ as $\{x_1,\dots,x_m\}$, we then have
$$\Omega\subset\bigcup_{i=1}^m B_{4r/5}(x_i).$$
Since $m$ need not be bounded by $c_0(n)\bigl(\frac{D_0}{r}\bigr)^n$, we will now define a subset of $\{1,\dots,m\}$ as follows.
We may fix $\eps=\eps(n)>0$ such that, for any $y,z\in\R^n$, $|y-z|\le \eps r$ implies $B_{4r/5}(y)\subset B_{5r/6}(z)$.
Set $i_1=1$.
We first remove all $j$ with $1<j\le m$ such that $|x_j-x_{i_1}|< \eps r$,
and we note that $B_{4r/5}(x_j)\subset B_{5r/6}(x_{i_1})$ for all such $j$.
Let $i_2$ be the smallest remaining index $>i_1$ (if any).
We then remove all $j$ with $i_2<j\le m$ such that $|x_j-x_{i_2}|< \eps r$,
and we note that $B_{4r/5}(x_j)\subset B_{5r/6}(x_{i_2})$ for all such~$j$.
Repeating the process, we obtain $I\le m$ and $1=i_1<\dots<i_I\le m$ such that
\bel{disteps}
|x_{i_j}-x_{i_k}|\ge \eps r\quad\hbox{ for all $1\le j< k\le I$}
\ee
 and we have
\bel{cover1}
\Omega\subset\bigcup_{j=1}^I B_{5r/6}(x_{i_j}).
\ee
Picking $x_0$ such that $\Omega\subset B(x_0,D_0)$, we have
 $B_{\eps r/2}(x_{i_j})\subset B(x_0,D_0+\eps r/2)$.
Since \eqref{disteps} guarantees that $B_{\eps r/2}(x_{i_j})\cap B_{\eps r/2}(x_{i_k})=\emptyset$ for all $1\le j<k\le I$,
it follows that $I |B_{\eps r/2}(0)|\le |B_{D_0+\eps r/2}(0)|$, hence
$$I\le \Bigl(1+\frac{2D_0}{\eps r}\Bigr)^n\le c_1(n)\Bigl(1+\frac{D_0}{r}\Bigr)^n.$$
Relabelling these points, we have thus proved claim \eqref{coverA00}.

\smallskip

Let us next prove  \eqref{coverA3} for some $N\le 1+\frac{6D}{r}$.
We observe that, for any $x,y\in\R^n$,
\bel{measure-c0}
\bar B_r(x)\cap \bar B_{5r/6}(y)\ne\emptyset\Longrightarrow |B_r(x)\cap B_r(y)|\ge c_0(n)r^n.
\ee
Fix $1\le k<l\le I$.
By the definition of $D$, there exists a Lipschitz curve $\gamma:[0,1]\to \Omega$
such that $\gamma(0)=x_k$, $\gamma(1)=x_l$ and $s(1)\le D$,
where $[0,1]\ni t\mapsto s(t)$ denotes the increasing curvilinear abscissa along $\gamma$.
We set $t_1=0$, $j_1=k$.
\smallskip

$\bullet$
If $\gamma([t_1,1])\subset B_r(x_{j_1})$ then we set $j_2=l$ and $p=2$.
Since $x_l=x_{j_2}\in B_r(x_{j_1})$, \eqref{measure-c0} guarantees that
\bel{measure-c010}
|B_r(x_{j_1})\cap B_r(x_{j_2})|\ge c_0(n)r^n.
\ee

$\bullet$ Otherwise, there exists a minimal $t_2\in(t_1,1]$ such that $\gamma(t_2)\in \partial B_r(x_{j_1})$.
In particular, $s(t_2)-s(t_1)\ge|\gamma(t_2)-x_{j_1}|=r$.
And owing to \eqref{cover1}, there exists $j_2\in\{1,\dots,I\}$ such that $\gamma(t_2)\in B_{5r/6}(x_{j_2})$.
Moreover, since $\gamma(t_2)\in \bar B_r(x_{j_1})\cap \bar B_{5r/6}(x_{j_2})$, \eqref{measure-c0} guarantees that
\eqref{measure-c010} is still true.

\smallskip
$\bullet$ If $\gamma([t_2,1])\subset B_r(x_{j_2})$  then we set $j_3=l$, $p=3$
and, similar to the case $p=2$, we obtain
$|B_r(x_{j_2})\cap B_r(x_{j_3})|\ge c_0(n)r^n$.

\smallskip
$\bullet$ Otherwise, there exists a minimal $t_3\in(t_2,1]$ such that $\gamma(t_3)\in \partial B_r(x_{j_2})$.
In particular,
\bel{curvilineardiff}
s(t_3)-s(t_2)\ge |\gamma(t_3)-\gamma(t_2)| \ge |\gamma(t_3)-x_{j_2}|-|\gamma(t_2)-x_{j_2}|\ge r-(5r/6)=r/6.
\ee
Then, owing to \eqref{cover1}, there exists $j_3\in\{1,\dots,I\}$ such that $\gamma(t_3)\in B_{5r/6}(x_{j_3})$.

\smallskip
$\bullet$ We can repeat this process as long as $\gamma([t_i,1])\not\subset B_r(x_{j_i})$.
Also, at the $i$-th step,
since $\gamma(t_i)\in \bar B_r(x_{j_{i-1}})\cap \bar B_{5r/6}(x_{j_i})$, \eqref{measure-c0} guarantees that
\bel{measure-c02}
|B_r(x_{j_{i-1}})\cap B_r(x_{j_i})|\ge c_0(n)r^n
\ee
and, similar to \eqref{curvilineardiff}, we have $s(t_i)-s(t_{i-1})\ge r/6$.
Since $s$ is an increasing function with $s(1)-s(0)\le D$, it follows that $i$ cannot exceed the value $1+(6D/r)$
and we eventually reach $i$ such that $\gamma([t_i,1])\subset B_r(x_{j_i})$.
Consequently we obtain an integer $p\le 2+(6D/r)$ and indices
$k=j_0,j_1,\dots,j_p=l$ such that \eqref{measure-c02} holds for all $i\in\{1,\dots,l\}$.

Finally, property \eqref{coverA2b} is trivial in case dist$(x_i, \partial \Omega)\ge 3r/2$.
In case $x_i\in \partial \Omega$, it easily follows from
\eqref{claimdz2} and the definition of $k_\Omega, \bar\rho_\Omega$ at the beginning of Section~\ref{sec-ext}.
This completes the proof.
\end{proof}

We turn to the proof of Theorem \ref{BHIoptim}, for which we will use the following particular case of the results in~\cite{GSS}. Here we denote $B_R^+ = B_R\cap \{x\::\: x_n>0\}$.

\begin{thm}[\cite{GSS}] \label{fromGSS} Assume $u\ge0$ in $B_1^+$ is a weak solution of $-\mathcal{L}u\ge f $ in $B_1^+$, for some $f\in L^q(B_1^+)$, and the coefficients of $\mathcal{L}$ satisfy \eqref{hyp1}-\eqref{hyp2} in $\Omega = B_1^+$.
There exist  constants $\epsilon, C>0$ depending only on $n,q,\alpha,\lambda,\Lambda$, and upper bounds on the $C^\alpha$-norms of $A,b_1$ and the $L^q$-norms of $b_2,c$ in $B_1^+$, such that for any $\omega\subset B_{3/4}^+$
\begin{equation}\label{fixWBHI}
\left(\int_{\omega} \left(\frac{u}{d}\right)^\epsilon\right)^{1/\epsilon} \le
C \left( \inf_{\omega} \frac{u}{d} +
\|f\|_{L^{q}(B_1^+)}\right).
\end{equation}
If on the other hand  $-\mathcal{L}u\le f $  in $B_1^+$, $u\le0$ on $\partial B_1^+$, then
\begin{equation}\label{fixlocmax}
\sup_{\omega} \frac{u^+}{d}\le C\left( \left(\int_{\omega} \left(\frac{u^+}{d}\right)^\epsilon\right)^{1/\epsilon} +
\|f\|_{L^{q}(B_1^+)} \right).
\end{equation}
\end{thm}

The inequality \eqref{fixWBHI} is a particular case of  \cite[Theorem 1.1]{GSS}, while \eqref{fixlocmax} is a simple application of the local maximum principle, see for instance \cite[p. 9]{GSS}.

\begin{proof}[Proof of Theorem \ref{BHIoptim}]
We
consider the collection of balls with radius $r_0/2$ which covers $\overline{\Omega}$ ($r_0$ is the number from \eqref{defr0}, \eqref{relMr0}) and the numbers $I,N$
 given by Proposition~\ref{geodes} (ii) applied with $r=r_0/2$.

Fix one such ball $B=B_{r_0/2}(\bar x)$ whose center $\bar x$ is on $\partial\Omega$. Let $B^\prime=B_{r_0}(\bar x)$ and $\Phi$ be the $C^{1,\bar\alpha}$ diffeomorphism defined in Section~\ref{sec-ext} which sends $(B^\prime\cap\Omega-\bar x)/r_0$ to $B_1^+$.

For $x\in B^\prime\cap\Omega$, let $y=\Phi(\frac{x-\bar x}{r_0})\in B_1^+$. Obviously for any $\hat x, \tilde x\in B^\prime\cap\Omega$ we have
\begin{equation}\label{compdist0}(1/2r_0)|\hat x- \tilde x|\le \frac{1}{\sup|D\Phi^{-1}|r_0}|\hat x- \tilde x|\le |\hat y -\tilde y|\le \frac{\sup|D\Phi|}{r_0}|\hat x- \tilde x|\le (2/r_0)|\hat x- \tilde x|,
\end{equation}
and in particular for each $y\in B_1^+$
\begin{equation}\label{compdist}
(1/2r_0)\mathrm{dist}(x,\partial\Omega)\le\mathrm{dist}(y, B_1^0) \le (2/r_0)\mathrm{dist}(x,\partial\Omega)
\end{equation}

We now make the change of variable $y= \Phi(\frac{x-\bar x}{r_0})$, $x= \bar x+ r_0\Phi^{-1}(y)$ in the inequality $-\mathcal{L}u\ge f $ in $B^\prime\cap\Omega$. Setting $\hat u (y) = u(x)$, a straightforward computation shows that $\hat u$ satisfies an equation $-\widehat{\mathcal{L}}\hat u\ge \hat f $ in $B_1^+$, where the coefficient $\hat A$ satisfies \eqref{hyp1} with modified constants $0<\hat \lambda\le \hat \Lambda$ depending only on upper bounds for $|D\Phi|$, $|D\Phi^{-1}|$ (which in our case are bounded by 2), and coefficients $\hat b_1, \hat b_2, \hat c$ which satisfy \eqref{hyp2}. In addition,  the $C^\alpha$-norms of $\hat A,\hat b_1$ as well as the $L^q$-norms of $\hat b_2,\hat c$ in $B_1^+$ are bounded above by a universal constant which again depends only on $|D\Phi|$, $|D\Phi^{-1}|$. The latter fact is due to the choice of $r_0$ -- see  Proposition \ref{scaleinvtilde} and its proof where the particular case $\Phi=I$ is considered in detail. For instance, we have for some universal $C_0$
$$\begin{aligned}
\|\hat b_2\|_{L^q(B_1^+)}
&\le C_0 r_0\Bigl(\int_{B_1^+} |b_2(\bar x+r_0\Phi^{-1}(y))|^q\,dy\Bigr)^{1/q}\\
&\le C_0 r_0^{1-n/q}\Bigl(\int_{B^\prime\cap\Omega} |b_2(x)|^q |\mathrm{det}\,D\Phi|\,dx\Bigr)^{1/q}
\le C_0 r_0^{1-n/q}\|b_2\|_{L^q(B'\cap\Omega)}\le C_0
\end{aligned}
$$
(the last inequality follows from \eqref{defr0}). Note also that the change of variables sends $B$ to a subset $\omega$ of $B_{3/4}^+$, because of \eqref{compdist0}. Hence we can apply Theorem \ref{fromGSS} to $\hat u$ and $-\widehat{\mathcal{L}}\hat u\ge \hat f $ in $B_1^+$, getting \eqref{fixWBHI} for $\hat u(y) $ and $\hat d(y) = \mathrm{dist}(y, B_1^0)$. Changing back into the $x$ variable and using \eqref{compdist} we get
\begin{equation}\label{ineqbdry11}
\left(\int_{B\cap \Omega}\left(\frac{u}{d}\right)^\epsilon\right)^{1/\epsilon} \le
C_0 r_0^{n/\epsilon} \left( \inf_{B\cap \Omega} \frac{u}{d} +  r_0^{1-n/q}\|f\|_{L^{q}(B'\cap\Omega)}\right).
\end{equation}

In the simpler case when a given ball $B=B_{r_0/2}(\bar x)$ is such that $\bar x\in \Omega$ (with dist$(\bar x,\partial\Omega)\ge 3r_0/4$) we use the change of variables $y= \frac{x-\bar x}{\tilde r_0}$, with $\tilde r_0 = \min\{r_0, \mathrm{dist}(\bar x,\partial\Omega)\}\in [3r_0/4, r_0]$, which sends $B_{\tilde r_0}$ to $B_1$ (and $B$ to $B_{\delta}$, $\delta\in[1/2,2/3]$), and the new operator again has uniformly bounded coefficients in the corresponding norms. By the interior weak Harnack inequality (as in Step 1 in the proof of  \cite[Theorem 2.1]{SS2}) this gives
\begin{equation}\label{ineqinside}
\left(\int_{B} u^\epsilon\,dx\right)^{1/\epsilon} \le C_0 {\tilde r_0}^{n/\epsilon} \left( \inf_{B} u + {\tilde r_0}^{2-n/q}\|f\|_{L^{q}(B')}\right).
\end{equation}
which implies
\begin{equation}\label{ineqbdry12}
\left(\int_{B}\left(\frac{u}{d}\right)^\epsilon\right)^{1/\epsilon} \le
C_0 r_0^{n/\epsilon} \left( \inf_{B} \frac{u}{d} +  r_0^{1-n/q}\|f\|_{L^{q}(B'\cap\Omega)}\right),
\end{equation}
since for $x\in B$ we have  dist$( x,\partial\Omega)\ge r_0/4$, hence $\sup_B d/\inf_B d\le (\inf_B d + r_0)/\inf_B d\le 5$.

Thanks to \eqref{ineqbdry11} and \eqref{ineqbdry12} we can repeat the iteration argument on \cite[p. 12]{SS2}, and deduce that for any two balls $B_k, B_l$ from the covering, $k,l\in \{1,\ldots,I\}$,
\begin{equation}\label{iterineq}
\left(\int_{B_l\cap\Omega} u^\epsilon\,dx\right)^{1/\epsilon} \le C_0^N  \left( \left(\int_{B_k\cap\Omega} u^\epsilon\,dx\right)^{1/\epsilon} + { r_0}^{1-n/q+n/\epsilon}\|f\|_{L^{q}_{ul}(\Omega)}\right)
\end{equation}
(we note that $m$ and $d$ in \cite{SS2} are our $I$ and $N$ here). Summing over $l\in \{1,\ldots,I\}$, and then using  \eqref{ineqbdry11}-\eqref{ineqbdry12} again for each $k\in \{1,\ldots,I\}$,
along with \eqref{coverA0} and \eqref{relMr0}, we obtain \eqref{sharpWBHI}.

The optimized full Harnack inequality \eqref{sharpBHI} follows exactly as in \cite[Step 3]{SS2} after the same rescaling as above and using \eqref{fixlocmax}, combined with \eqref{sharpWBHI}.
\end{proof}

\begin{rem} \label{remnoloss}
We note that the above argument actually yields $r_0^{1-n/q}$ and $\|f\|_{q,r_0,\Omega}$
on the right hand side of \eqref{sharpBHI},
instead of $D^{1-n/q}$ and $\|f\|_{L^q(\Omega)}$, respectively.
However, in view of the exponential factor in \eqref{sharpBHI}, of the inequality
$$\|f\|_{r_0,q,\Omega} \le \|f\|_{L^q(\Omega)}\le  c(n)(D/r_0)^{n/q} \|f\|_{r_0,q,\Omega}$$
(since one can cover $\Omega$ by $c(n)(D/r_0)^n$ balls of radius $r_0$),
and of the fact that
$(M+r^{-1}_\Omega)D=D/r_0$ by \eqref{relMr0},
there is no loss coming from this replacement.
\end{rem}

\subsection{Green's formula.}

We finally give a suitable integration by parts formula, used at several places in this work,
which handles situations when one of the two functions
has limited regularity near the boundary. It may be known but we give a proof since we have not found
a statement suitable to our needs in the literature.

\begin{lem} \label{weakdiv2}
Let $\Omega\subset \rn$ be a bounded $C^1$ domain,
let $\ld$ be an operator as in \eqref{defdiv}, where $A\in C(\overline \Omega)$,
$b_1,c \in L^1(\Omega)$, $b_2 \in L^2(\Omega)$ and let
\begin{equation}\label{Hypweakdiv2}
u\in C(\overline \Omega)\cap H^1(\Omega),\quad v\in C^1(\overline \Omega),\quad\hbox{ with $v=0$ on $\partial \Omega$.}
\end{equation}
Assume that $\mathcal{L}u, \mathcal{L^*}v\in L^1(\Omega)$,
where  $\mathcal{L}u, \mathcal{L^*}v$ are understood in the sense of distributions. Then
$$\int_{\Omega}\bigl\{v\mathcal{L}u-u\mathcal{L}^*v\bigr\}\, dx
=-\int_{\partial\Omega}\nu\cdot A^T(x)u\nabla v\, d\sigma.$$
\end{lem}

\begin{proof} By \cite{Liebe}, there exists a regularized distance, namely a function $\rho\in C^1(\overline\Omega)$
such that, for some constants $C_1,C_2,\eps_0>0$, there holds  $C_1d\le\rho\le C_2d$ in $\Omega$
and $C_1\le |\nabla\rho|\le C_2$ in $\{x\in\overline{\Omega}, d(x)\le\eps_0\}$.
For $\eps\in(0,\eps_0]$, we set
$$G_\eps=\{x\in\Omega,\ \rho(x)>\eps\},\quad \Gamma_\eps=\{x\in\Omega,\ \rho(x)=\eps\}.$$
By the implicit function theorem,
$G_\eps$ is a $C^1$-smooth open set for each $\eps\in(0,\eps_0]$.
We denote the outer normal vector field of $\partial G_\eps$ by $\nu_\eps$
and the surface measure on~$\Gamma_\eps$ by $d\sigma_\eps$.
We also have
$\nu_\eps\,d\sigma_\eps\rightharpoonup \nu\,d\sigma$ as $\eps\to 0$, weakly in the sense of measures, that is
\be\label{CvSurfMeas}
\lim_{\eps\to  0}\int_{\Gamma_\eps} V\cdot\nu_\eps\,d\sigma_\eps=\int_{\partial\Omega} V\cdot\nu\,d\sigma,
\quad V\in C(\overline\Omega).
\ee
Indeed, if $V\in C^1(\overline\Omega)$, by Green's formula and dominated convergence. we have
$$\int_{\partial\Omega} V\cdot\nu\,d\sigma-\int_{\Gamma_\eps} V\cdot\nu_\eps\,d\sigma_\eps
=\int_{\Omega\setminus G_\eps} {\rm div}\, V\, dx\to 0,$$
and the general case $V\in C(\overline\Omega)$ then follows by using the density of $C^1(\overline\Omega)$ in $C(\overline\Omega)$,
along with the fact that $|\Gamma_\eps|\le C$ for $\eps\in(0,\eps_0]$.

Let $\psi\in C^\infty_0(\R^n)$ with ${\rm Supp}(\psi)= B_1$ satisfy $0\le\psi\le 1$ and $\int_{B_1}\psi\,dx=1$.
Let $(\psi_j)_{j\ge 1}$ be the corresponding sequence of mollifiers, defined by $\psi_j(x)=j^n\psi(jx)$.
Set
$$U=-A^T(x) u\nabla v\in C(\overline\Omega),\quad \phi=A(x)\nabla u+(b_1+b_2)u \in L^1(\Omega),\quad W=U+v\phi.$$
An easy computation (in the distribution sense) yields
\begin{equation}\label{DefWstar}
{\rm div}\, W=v\mathcal{L}u-u\mathcal{L}^*v \in L^1(\Omega).
\end{equation}
Let $\eps\in(0,\eps_0/3)$. Then, for $j> C_2\eps^{-1}$, $W_j:=W\ast \psi_j$ is well defined and smooth  in $\overline G_\eps$
and the divergence formula yields
$$\int_{G_\eps} {\rm div}\, W_j\, dx=\int_{\Gamma_\eps} W_j\cdot\nu_\eps\,d\sigma_\eps.$$
By assumption \eqref{Hypweakdiv2}, we have $|v|\le Cd(x)$ for some constant $C>0$.
Setting $\phi_j=|\phi|\ast \psi_j$, it follows that
$$\delta_j(\eps)
:=\Bigl|\int_{G_\eps} {\rm div}\, W_j\, dx-\int_{\Gamma_\eps} (U\ast \psi_j)\cdot\nu_\eps\,d\sigma_\eps\Bigr|
=\Bigl|\int_{\Gamma_\eps} \bigl((v\phi)\ast \psi_j\bigr)\cdot\nu_\eps\,d\sigma_\eps\Bigr|
\le  C\eps\int_{\Gamma_\eps} \phi_j \,d\sigma_\eps.$$
Since $\Gamma_\eps$ is the $\eps$-level set of the function $\rho$, by the co-area formula, we have
$$\int_\eps^{2\eps} \left(\int_{\Gamma_t}\phi_j\,d\sigma_t\right)\,dt
=\int _{G_\eps\setminus G_{2\eps}} \phi_j|\nabla \rho(x)|\,dx
 \le C_2\int _{G_\eps\setminus G_{2\eps}} \phi_j\,dx.$$
Also, for all $j>C_2\eps^{-1}$, by Fubini's theorem, we have
$$
\int _{G_\eps\setminus G_{2\eps}} \phi_j\,dx
=\int _{G_\eps\setminus G_{2\eps}}\int_{|y|<1/j} |\phi(x-y)| \psi_j(y) \,dy\,dx
\le \int _{\Omega\setminus G_{3\eps}} |\phi(x)| \,dx.
$$
Combining the last three formulas, we obtain
\be\label{Limweakdiv0}
\frac{1}{2\eps}\int_\eps^{2\eps} \delta_j(r)\,dr
\le CC_2\int_\eps^{2\eps} \left(\int_{\Gamma_t}\phi_j\,d\sigma_t\right)\,dt
\le CC_2\int _{\Omega\setminus G_{3\eps}} |\phi(x)| \,dx,\quad j> C_2\eps^{-1}.
\ee

Now, by \eqref{DefWstar}, we have
$$
\hbox{${\rm div}\,W_j= ({\rm div}\,W)\ast\psi_j\to {\rm div}\,W$ in $L_{loc}^1(\Omega)$},\quad j\to\infty.
$$
Since $U\in C(\overline \Omega)$, we also have
$$
\hbox{$U\ast \psi_j\to U$ in $C_{loc}(\Omega)$}, \quad j\to\infty.
$$
Setting
$\delta_\infty(r):=\bigl|\int_{\Omega_r} {\rm div}\, W\, dx-\int_{\Gamma_r} U\cdot\nu_r\,d\sigma_r\bigr|$,
it follows that $\lim_{j\to\infty}\delta_j(r)=\delta_\infty(r)$ uniformly for $r\in(\eps,2\eps)$ hence,
in view of \eqref{Limweakdiv0},
\be\label{Limweakdiv1}
\frac{1}{\eps}\int_\eps^{2\eps} \delta_\infty(r)\,dr\le 2CC_2\int _{\Omega\setminus G_{3\eps}} |\phi(x)| \,dx.
\ee

Finally, by
\eqref{CvSurfMeas}, we have
$\lim_{\eps\to  0}\int_{\Gamma_\eps} U\cdot\nu_\eps\,d\sigma_\eps=\int_{\partial\Omega} U\cdot\nu\,d\sigma$,
 and since also
$\lim_{\eps\to  0}$ $\int_{G_\eps} {\rm div}\, W\, dx=\int_\Omega {\rm div}\, W\, dx$
by dominated convergence, we deduce
\be\label{Limweakdiv2}
\lim_{\eps\to  0}\delta_\infty(\eps)=\int_\Omega {\rm div}\, W\, dx-\int_{\partial\Omega} U\cdot\nu\,d\sigma.
\ee
Passing to the limit $\eps\to 0$ in \eqref{Limweakdiv1}, respectively using \eqref{Limweakdiv2} on the left hand side
and dominated convergence on the right hand side, we conclude that
$$\int_\Omega {\rm div}\, W\, dx-\int_{\partial\Omega} U\cdot\nu\,d\sigma=0,$$
which is what we wanted to prove.
\end{proof}

\end{document}